\documentclass{article}

\usepackage[latin1]{inputenc} 
\usepackage{amsfonts} 
\usepackage{amsmath} 

\usepackage{graphicx} 
\usepackage{subfigure} 
\usepackage{psfrag} 

\usepackage{pstricks} 

\usepackage{multirow} 
\usepackage{booktabs}

\def\R{\mathbb{R}}

\def\S{\mathbb{S}}

\def\bx{\boldsymbol{x}}
\def\bq{\boldsymbol{q}}
\def\bu{\boldsymbol{u}}
\def\br{\boldsymbol{r}}
\def\eps{\varepsilon}
\newcommand{\norm}[1]{\lVert#1\rVert}
\newcommand{\abs}[1]{\lvert#1\rvert}
\DeclareMathOperator{\tv}{Tot.Var.}

\newtheorem{Def}{Definition}
\newtheorem{Lemma}{Lemma}
\newtheorem{Prop}[Lemma]{Proposition}  
\newtheorem{rem}{Remark}

\newenvironment{proof}[1][Proof]{\begin{trivlist}
\item[\hskip \labelsep {\bfseries #1}]}{\end{trivlist}}

\begin{document}

\title{Numerical simulations of the Euler system with congestion
  constraint}
\author{Pierre Degond\footnote{Universit\'e de Toulouse; UPS, INSA, UT1, UTM; 
   Institut de Math\'ematiques de Toulouse; F-31062 Toulouse,
   France. Email: pierre.degond@math.univ-toulouse.fr}, Jiale Hua\footnote{CNRS; Institut de Math\'ematiques de Toulouse UMR 5219;
   F-31062 Toulouse, France.  Email: jiale.hua@math.univ-toulouse.fr}, Laurent
 Navoret\footnote{CNRS; Institut de Math\'ematiques de Toulouse UMR
   5219; F-31062 Toulouse, France.  Email: laurent.navoret@math.univ-toulouse.fr}}
\date{}
\maketitle
\begin{abstract}

In this paper, we study the numerical simulations for Euler system with maximal density
constraint. This model is developed in
\cite{2008_Traffic_DegondRascle,2010_Congestion} with the constraint 
 introduced into the system by a singular
pressure law, which causes the transition of different asymptotic
dynamics between different regions.  To overcome these difficulties, we adapt and
implement two asymptotic preserving (AP) schemes originally designed
for low Mach number limit
\cite{2007_Low_mach_D_SJ_JGL,2009_Degond_LowMach} to our
model. These schemes work for  the different dynamics and capture the
transitions well. Several numerical tests both in one dimensional and
two dimensional cases are carried out for our schemes.
\end{abstract}


\section{Introduction}

Several models involve congestion constraints: concentration
constraints occur in multi-phase flow modeling
\cite{2PhaseFlow_Bouchut_al}, maximal density constraints occur when
dealing with finite-size interactive agents in herds of gregarious
mammals \cite{2010_Congestion}, in cars or pedestrians flows
\cite{2008_Traffic_DegondRascle,2008_Trafficflow_Berthelin_D,2008_ModCrowd_Bellomo},
flux constraints occur for supply chains
\cite{2006_ArmDegRing},... The dynamics in congested regions strongly differ from the dynamics
dynamics in free regions. To study the transitions between congested and free regions, a general methodology was first carried out in \cite{2PhaseFlow_Bouchut_al} for multiphase flows and later on generalized to traffic \cite{2008_Traffic_DegondRascle} or herding problems \cite{2010_Congestion}: the stiffness of the constraint leads to a singular perturbation problem and then the limit problem provides a clear cut-off between the two dynamics. In this paper, we will consider the Euler system with a singular pressure which encodes a maximal density constraint. As in \cite{2010_Congestion}, the limit problem is a two-phase model between incompressible regions, where the maximal density is reached, and compressible regions for lower densities. Our goal is to provide numerical schemes that are able to capture this limit problem and these transitions. In this paper, we adapt and compare two numerical methods presented in \cite{2009_Degond_LowMach} and in \cite{2007_Low_mach_D_SJ_JGL} for the low Mach number limit.

A lot of efforts have been made to devise numerical schemes valid for all Mach numbers, that is, for both compressible and incompressible flows. They avoid the switch between different methods, when different Mach numbers occur in different sub-domains. Among such schemes, one approach is the extension of compressible conservative methods to incompressible flows thanks to preconditioning techniques \cite{1999_Turkel_Precond,1999_GuillardViozat,2001_GuillardMurrone,2008_LiGu}. The second approach is the extension of incompressible methods to compressible flows and leads to pressure correction methods on staggered grids \cite{1971_HarlowAmsden_ICE_LM} and their conservative versions \cite{1980_Patankar_book,1998_BijlWesseling,2003_HeulWesseling}. Their adaptation to the conservative frameworks has led to time semi-implicit schemes : the implicit discretization of the fluxes (the mass flux and the pressure part of the momentum flux are taken implicitly) is combined with the resolution of the elliptic equation satisfied by the pressure. The implicit treatment of the pressure flux ensures stability with respect to the propagation of fast acoustic waves in the low-Mach number limit but induces a lot of diffusion. We can cite numerous works following this methodology \cite{1995_Klein_LM,2002_WallPierceMoin,2003_Munz_MPV,2005_Nerinckx,2009_Rauwoens}. The methods we consider in this paper are among the simplest ones: the scheme in \cite{2009_Degond_LowMach} is a semi-implicit scheme with a division of the pressure into explicit and implicit parts and in \cite{2007_Low_mach_D_SJ_JGL}, the Gauge decomposition of the momentum enables the hydrostatic pressure to act only on the divergence-free part of the momentum. The former method will be called in the present paper the Direct method, while the Gauge method will refer to the latter.

The purpose of this paper is to present simple variants of the Direct
\cite{2009_Degond_LowMach} and the Gauge method
\cite{2007_Low_mach_D_SJ_JGL}, that are able to handle congestion
problems. As announced above, they are designed to solve the
isentropic Euler system supplemented by a pressure law $p(\rho)$,
which is singular as the density $\rho$ approaches a maximal density
denoted $\rho^{\ast}$. A small parameter $\eps$ is introduced to
control the stiffness of this maximal density constraint: the rescaled
pressure $\eps p(\rho)$ is of order $O(1)$ in congested regions $\rho
\sim \rho^{\ast}$ and of order $O(\eps)$ in low density regions $\rho
< \rho^{\ast}$. In the limit $\eps \rightarrow 0$, the system leads to
a two-phase model: the incompressible Euler system in maximal density
domains and the pressureless gas dynamics system for uncongested
densities domains. This asymptotic model was first proposed and
studied in \cite{2PhaseFlow_Bouchut_al,2002_ExistPressLess_Berth} in a
one-dimensional framework. However, this asymptotic model is only
partially defined since transmission conditions at the interfaces
between the two phases are lacking. Besides, unless one-dimensional
solutions can be provided (see \cite{2PhaseFlow_Bouchut_al} and
appendix~\ref{sec:one-dimens-riem}), their extensions to the
two-dimensional case are open problems (especially the dynamics of two
colliding congested domains). In this context, asymptotic preserving
scheme is a good tool for this problem.  

The Direct and the Gauge methods are called asymptotic-preserving (AP)
since they are uniformly consistent with the low-Mach number
limit. Besides, they are also uniformly stable. These methods are
expected to capture both the compressible and the incompressible
dynamics arising in the congestion limit of the Euler system with the
maximal density constraint. In this case, such AP numerical schemes
are very powerful since they provide the dynamics of transitions, for
which analytical results may be lacking. Moreover, they enable us to avoid  dealing with physical and numerical interface tools, such as front-tracking \cite{2001_FrontTrackMultiphase_tryggvason} or volume-of-fluid methods \cite{1981_HirtNichols_VOF}. Unlike these tracking methods, ours are front-capturing methods and then share some analogies with level-set \cite{2003_LevelSet_BookOsher} and diffusive interface methods \cite{1998_Anderson_DiffInterface}: like level-set method, the dynamics of the transition are implicitly embodied in the dynamics of an auxiliary function which here is the density and like the diffusive interface methods, the sharp interface is viewed as the limit of the smooth transitions of the perturbation problem.

The Direct method cannot be directly applied to the congestion
problem. Indeed, in \cite{2009_Degond_LowMach}, the pressure $p(\rho)$
is splitted into an explicit part $p_0(\rho)$ and an implicit part $p_1(\rho)$ in order to keep some numerical diffusion and avoid numerical oscillations. For the singular congestion pressure-law, we modify this splitting, such that it still ensures the stability of the scheme. Besides, it ensures the consistency of the explicit part of the scheme with the limit pressureless gas dynamics in the low density regions. Indeed, the pressureless gas dynamics system is weakly hyperbolic and there is no uniqueness of the entropic solution \cite{PGD_Bouchut}. Then, keeping an explicit pressure $p_0(\rho)$ makes the asymptotic numerical solutions consistent with the good entropic solutions. For the same reasons, this pressure splitting is introduced into the Gauge method. 

The AP property of the two methods for the congested Euler system is demonstrated for the congested domains. This analysis is hard to extend to coexistent congested and uncongested regions since the dynamics of the interfaces between the different regions is not explicitly implemented into the schemes. However, several numerical test-cases provides numerical evidence of the AP property. Comparisons of the two schemes are also carried out and different behaviours of the schemes at the interfaces are measured.
 
The paper is organized as follows. In section 2, we introduce the
Euler system with the maximal density constraint. To have some basic
idea of the solution, we give its formal asymptotic limit. In section
3, we describe the time semi-discretization of the Direct and Gauge
schemes. The fully space and time discretizations are exposed in
section 4, in one  dimensional setting: they are based on the
centered Rusanov scheme, also called the local Lax-Friedrichs scheme
\cite{leveque_book_2002}. The two dimensional setting is quite
similar. Therefore, we present it in the appendix. Finally, numerical
simulations are performed in section 5 to compare the two schemes:
several test cases in the one dimensional setting and one case in the two
dimensional setting.
Two appendices close this study: we describe one-dimensional solutions
and the two dimensional discretizations. 

\section{The Euler system with congestion and its asymptotic limit}
\label{sec:euler-with-cong}
\subsection{The model}
We consider the two-dimensional Euler system:
\begin{eqnarray}
&&\partial_{t}\rho + \nabla_{\bx}\cdot \bq = 0,\label{Eq:Euler_rho}\\
&&\partial_{t}\bq + \nabla_{\bx}\cdot\left(\frac{\bq\otimes \bq}{\rho}\right) + \nabla_{\bx} p(\rho) = 0,\label{Eq:Euler_q}
\end{eqnarray}
where $\rho(\bx,t) \in \R$ denotes the mass density, $\bq = \rho
\bu(\bx,t) \in \R^2$ is the momentum  field depending on the position $\bx \in \R^2$ and the time $t > 0$. The pressure $p(\rho)$ is an increasing function such that $p(\rho) \sim \rho^{\gamma}$ for densities $\rho \ll 1$ and $p(\rho) \rightarrow +\infty$ as $\rho$ tends to the congestion density $\rho^{\ast}$. In the following, we will consider the function:
\begin{equation}
p(\rho) = \frac{1}{\left(\frac{1}{\rho} -
    \frac{1}{\rho^{\ast}}\right)^{\gamma}}, \quad \gamma>0.
\end{equation}
This pressure prevents the density from exceeding the congestion density. A variant is the van der Waals equation of state \cite{1986_VanDerWaalsRiem_Hattori}. The operators $\nabla_{\bx}$ and $\nabla_{\bx}\cdot$ are the gradient and the divergence of vector fields or tensor. For two vectors $\boldsymbol{a}$ and $\boldsymbol{b}$, $\boldsymbol{a}\otimes \boldsymbol{b}$ denotes the tensor product.  

This model already appears in \cite{2010_Congestion} with the
additional constraint: $\bq/\rho = \bu \in \S^{1}$. In this paper, we
focus on the pressure singularity and the corresponding numerical schemes. 

The singular pressure induces two different dynamics: for regions with
densities near $\rho^{\ast}$, the pressure takes very large values in
comparison with the pressure in low-density regions. As in
\cite{2010_Congestion} and previous work on traffic modeling
\cite{2008_Traffic_DegondRascle}, we would like to clearly separate
the two different dynamics. To this aim, we rescale $p(\rho)$ into
$\eps p(\rho)$, where the parameter $\eps \ll 1$ is the scale of the
pressure in the low density regions: $p(\rho) = O(\eps)$ for density
$\rho \ll 1$ while $p(\rho) = O(1)$ for density $\rho \sim
\rho^{\ast}$, see Fig.~\ref{Fig:15}.
\begin{figure}
  \centering
  \psfrag{rho}{$\rho$}
  \psfrag{eps p}{$\varepsilon p$}
  \psfrag{p}{$p$}
 \includegraphics[scale=0.5]{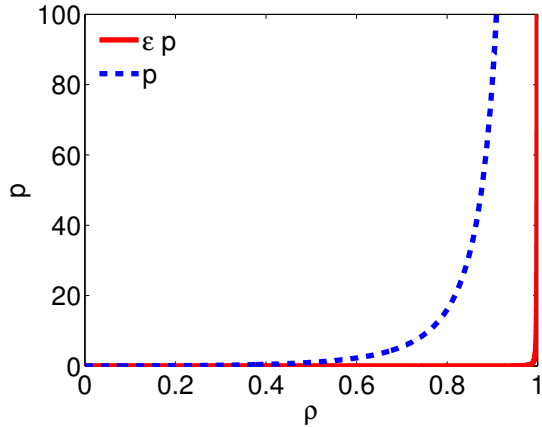}
  \caption{The pressure $p$ and the rescaled pressure $\varepsilon p$}
  \label{Fig:15}
\end{figure}

Denoting $\rho^{\eps}$ and $\bq^{\eps}$ the new unknowns, system (\ref{Eq:Euler_rho})-(\ref{Eq:Euler_q}) becomes:
\begin{eqnarray}
&&\partial_{t}\rho^{\eps} + \nabla_{\bx}\cdot \bq^{\eps} = 0,\label{Eq:Euler_rho_eps}\\
&&\partial_{t}\bq^{\eps} + \nabla_{\bx}\cdot\left(\frac{\bq^{\eps}\otimes \bq^{\eps}}{\rho^{\eps}}\right) + \nabla_{\bx} \left(\eps p(\rho^{\eps})\right) = 0,\label{Eq:Euler_q_eps}
\end{eqnarray}   
Moreover, taking the time derivative of (\ref{Eq:Euler_rho_eps}) and subtracting the divergence of equation (\ref{Eq:Euler_q_eps}), we easily obtain the wave-like equation satisfied by the density:
\begin{equation*}
\partial_{t^2}\rho - \nabla_{\bx}^2:\left(\frac{\bq^{\eps}\otimes \bq^{\eps}}{\rho^{\eps}}\right) - \Delta_{\bx} (\eps p(\rho^{\eps})) = 0,
\end{equation*}
where $\nabla_{\bx}^2$ denotes the tensor of the second derivatives
and for two tensors $\boldsymbol{a}$ and $\boldsymbol{b}$ and $\boldsymbol{a}:\boldsymbol{b}$ denotes the contracted product of tensor. The operator $\Delta_{\bx}$ denotes the Laplacian of a scalar function. Actually, system (\ref{Eq:Euler_rho_eps})-(\ref{Eq:Euler_q_eps}) is a strictly hyperbolic problem, with characteristic wave speeds in the $x$-direction (where $x$ is the first component of a basis of $\R^2$) given by: 
\begin{equation}
\lambda_{1}^{\eps} = u^{\eps}_{x} - \sqrt{\eps p'(\rho^{\eps})},\quad \lambda_{2}^{\eps} = u^{\eps}_{x},\quad \lambda_{3}^{\eps} = u^{\eps}_{x} + \sqrt{\eps p'(\rho^{\eps})},\label{Eq:caracteristic_speed}
\end{equation}
where $u_{x}^{\eps}$ is the $x$-component of the macroscopic velocity $\bu^{\eps}(x,t) = \bq^{\eps}(x,t)/\rho^{\eps}(x,t)$. Standard hyperbolic numerical schemes have to resolve the Courant-Friedrichs-Levy (CFL) condition:
\begin{equation}
\max(|\lambda_{1}^{\eps}|,|\lambda_{2}^{\eps}|,|\lambda_{3}^{\eps}|) \Delta t  \leq \Delta x.\label{Eq:CFL_cond}
\end{equation}  
In the next section, we will see that this constraint may be too stringent for these schemes to capture the asymptotic limit. 

\subsection{The asymptotic limit} 

The limit of the pressure term $\eps p(\rho^{\eps})$ depends on the limit of $\rho^{\eps}$. Indeed, if $\rho^{\eps} \rightarrow \rho$ with $\rho < \rho^{\ast}$, then $\eps p(\rho^{\eps})$ converges to $0$. Otherwise, $\rho^{\eps} \rightarrow \rho^{\ast}$ and the limit of $\eps p(\rho^{\eps})$, denoted $\bar{p}$, can be non zero and depends on the convergence rate of $\rho^{\eps}$. We assume that the limit $\bar{p}$ is always finite. Therefore, the formal limit of system (\ref{Eq:Euler_rho_eps})-(\ref{Eq:Euler_q_eps}) as $\eps$ goes to zero is:
\begin{eqnarray}
&&\partial_{t}\rho + \nabla_{\bx}\cdot \bq = 0,\label{Eq:mass}\\
&&\partial_{t}\bq + \nabla_{\bx}\cdot\left(\frac{\bq\otimes \bq}{\rho}\right) + \nabla_{\bx} \bar{p} = 0,\label{Eq:momentum}\\
&&(\rho - \rho^{\ast})\bar{p} = 0.\label{Eq:constraint_CongNum}
\end{eqnarray}  
A one-dimensional version of this asymptotic model was proposed for two-phase flow modeling in \cite{2PhaseFlow_Bouchut_al}, where the density plays the role of the volume fraction of liquid in a liquid-gas model. The derivation of the model lies on a relaxation to zero of the relative velocities of the gas and liquid and is therefore different from the one studied in this article.  Existence and stability of solutions are proved for the one-dimensional version of system (\ref{Eq:mass})-(\ref{Eq:momentum})-(\ref{Eq:constraint_CongNum}) in \cite{2002_ExistPressLess_Berth}, and in all dimensions with viscous term in \cite{1999_barotropic_LionsMasmoudi}.

As regards the characteristic speeds, we note that if $\eps
p(\rho^{\eps})$ tends to $\bar{p} < + \infty$,  then we have
$\rho^{\ast} - \rho^{\eps} = O(\eps^{\frac{1}{\gamma}})$ and then
$\eps p'(\rho^{\eps}) = O(\eps^{\frac{1}{\gamma} - 1})$. Therefore, if
$\gamma > 1$, $\lambda_{\pm}^{\eps}$ can become infinite and waves
with infinite speed can occur. It is the low-Mach number asymptotics
that leads to incompressible dynamics. Actually, in the congested domain where $\rho = \rho^{\ast}$, system (\ref{Eq:mass})-(\ref{Eq:momentum})-(\ref{Eq:constraint_CongNum}) yields the incompressible Euler equation:  
\begin{eqnarray}
&&\rho = \rho^{\ast},\nonumber\\
&&\nabla_{\bx}\cdot \bu = 0,\label{Eq:Incomp_constraint}\\
&&\partial_{t}\bu + \bu\cdot\nabla_{\bx}\bu + \frac{1}{\rho^{\ast}}\nabla_{\bx} \bar{p} = 0.\nonumber
\end{eqnarray} 
Equation (\ref{Eq:Incomp_constraint}) is the incompressible constraint and the Lagrange multiplier related to this constraint is the pressure $\bar{p}$. The CFL condition (\ref{Eq:CFL_cond}) degenerate into $\Delta t = 0$: standard hyperbolic schemes are unable to compute the asymptotic dynamics. Thus, numerical schemes with relaxed CFL condition have to be designed.  

In the free domain where $\rho < \rho^{\ast}$, the CFL condition (\ref{Eq:CFL_cond}) is not an obstacle although the system degenerates into a non-hyperbolic problem: $\lim \lambda_{1}^{\eps} = \lim \lambda_{3}^{\eps} = u$. This is a large-Mach number asymptotic. Numerical schemes, originally developed for hyperbolic systems, have to be proved to capture this singular limit. System (\ref{Eq:mass})-(\ref{Eq:momentum})-(\ref{Eq:constraint_CongNum}) yields the pressureless gas dynamics:
\begin{eqnarray*}
&&\rho < \rho^{\ast},\\
&&\partial_{t}\rho + \nabla_{\bx}\cdot \bq = 0,\\
&&\partial_{t}\bq + \nabla_{\bx}\cdot\left(\frac{\bq\otimes \bq}{\rho}\right) = 0,
\end{eqnarray*} 
Without upper-bound on the density, this system would lead to concentration phenomena even for smooth initial data and so the density may become a measure with singular part. It is related to the so-called sticky particle dynamics. Existence of solutions and numerical schemes have been developed in the one-dimensional case \cite{PGD_Bouchut,1998_BrGr_sticky}.

The asymptotic system
(\ref{Eq:mass})-(\ref{Eq:momentum})-(\ref{Eq:constraint_CongNum}) is
not complete: the well-posed definition of $\bar{p}$ requires boundary
conditions at the interface between congested regions $\left\{\rho =
  \rho^{\ast}\right\}$ and free regions $\left\{\rho <
  \rho^{\ast}\right\}$. They can not be directly obtained by the
formal derivation. One possible answer to this question is to look at
the theoretical asymptotic behaviour of solutions of the initial
system (\ref{Eq:Euler_rho_eps})-(\ref{Eq:Euler_q_eps}). Such an
approach is investigated in appendix~\ref{sec:one-dimens-riem} but only in the
one-dimensional case. Extensions to two dimensional settings are difficult
and will be the subject of future works. The second possible approach
is to use numerical schemes to capture the asymptotic dynamics: this
is the main methodology we develop in this paper. It has the advantage
to be applicable in any dimensions. However, one-dimensional problems are
still good test cases to valid these numerical schemes.   

\section{Time semi-discretization schemes}
This section is the center of this paper. It is dedicated to the presentation of two numerical schemes, which are able to capture the asymptotic limit of the Euler system with congestion presented in the previous section. For this purpose, we adapt the asymptotic-preserving (AP) schemes developed in \cite{2009_Degond_LowMach} and \cite{2007_Low_mach_D_SJ_JGL} for the low-Mach limit of the isentropic Euler system.

\subsection{The time semi-implicit discretization}

 We first define a
time semi-implicit discretization, which will be the building block of
the considered AP schemes.

Let $\rho^n,q^n$ be the approximations of the density $\rho$ and the
momentum $q$ at $t^n=n\Delta t,
n=0,1,\ldots$, where $\Delta t$ is the time step.
The semi-discretization of the AP scheme for the n-th time step is as follows:

\begin{align}
 & \frac{\rho^{n+1} - \rho^{n}}{\Delta t} + \nabla_{\boldsymbol{x}}\cdot \boldsymbol{q}^{n+1} = 0,\\
&\frac{\boldsymbol{q}^{n+1} - \boldsymbol{q}^{n}}{\Delta t} + \nabla_{\boldsymbol{x}}\cdot\left(\frac{\boldsymbol{q}^n \otimes \boldsymbol{q}^n }{\rho^{n}}\right) +   \nabla_{\boldsymbol{x}} (\varepsilon p(\rho^{n+1})) = 0.
\end{align}
The full discretization in time and space is postponed to the next
section. We want to show that the scheme is asymptotic-preserving. In
other words, it captures the correct behaviour of the limiting
equation as $\varepsilon\to 0$. To achieve this, the implicitness of
$\rho,\boldsymbol{q}$ is crucial. 

Observe that the explicit part of the above scheme is pressureless. 
However, the pressureless Euler system is  weakly hyperbolic, giving
rise to the formation of density concentrations known as
delta-shocks. Several numerical schemes for this system was proposed
in
\cite{2004_NumericPGD_BouchutJin,2004_Leveque_Dust,2006_BertBreussTit_pressureless,2007_ChertKurgRyk_StickyPart}. In
\cite{2004_NumericPGD_BouchutJin}, the authors proposed a kinetic
scheme, that is valid for the isothermal Euler system and leads to a
kinetic scheme for the pressureless system in the vanishing pressure
limit. Here, to avoid this difficulty and as already proposed in
\cite{2009_Degond_LowMach}, we split the pressure into an explicit and
an implicit part. Numerical tests in section~\ref{sec:Num_result} will
demonstrate how this splitting reduces oscillations. Thus, the scheme is written as follows:
\begin{align}
&\frac{\rho^{n+1} - \rho^{n}}{\Delta t} + \nabla_{\boldsymbol{x}}\cdot \boldsymbol{q}^{n+1} = 0,\label{eq:1}\\
&\frac{\boldsymbol{q}^{n+1} - \boldsymbol{q}^{n}}{\Delta t} + \nabla_{\boldsymbol{x}}\cdot\left(\frac{\boldsymbol{q}^n\otimes \boldsymbol{q}^n }{\rho^{n}}\right) +  \nabla_{\boldsymbol{x}} (\varepsilon p_0(\rho^{n})) +\nabla_{\boldsymbol{x}} (\varepsilon p_1(\rho^{n+1})) = 0,\label{eq:2}
\end{align}
where the explicit part is given as
\begin{gather}
       p_0(\rho)=
  \begin{cases}
    \cfrac{1}{2} p(\rho),& \text{ if } \rho \leq \rho_*-\delta,\\[3ex]
    \begin{split}
      \frac{1}{2}\big( &p( \rho_*-\delta)+p'( \rho_*-\delta)(\rho-
      \rho_*+\delta)\\
      &+\frac{1}{2}p''(\rho_*-\delta)(\rho-
      \rho_*+\delta)^2\big)
    \end{split}
,& \text{ if } \rho>  \rho_*-\delta,
  \end{cases}
\end{gather}
and the implicit part is
\begin{gather}
p_1(\rho)=p(\rho)-p_0(\rho), \quad \delta=\varepsilon^{\frac{1}{\gamma+2}}.
\end{gather}
The choice of $\delta$ makes sure that $p_0$ and its derivatives up to
second order are always bounded.   To make sure all the coefficients
appearing in the elliptic equation we will derive in the next section are continuous,  we choose $p_0$ to be a
second order approximation to $p$,  instead of a first order one. For
later usage, also note that the function
$p_1(\rho)$ is invertible. This is easily seen from the property of
function $p$ and $p_0$, see Fig~\ref{fig:9}.
\begin{figure}
  \centering
  \psfrag{rho}{$\rho$}
  \psfrag{eps p}{$\varepsilon p$}
  \psfrag{eps p0}{$\varepsilon p_0$}
  \includegraphics[scale=0.5]{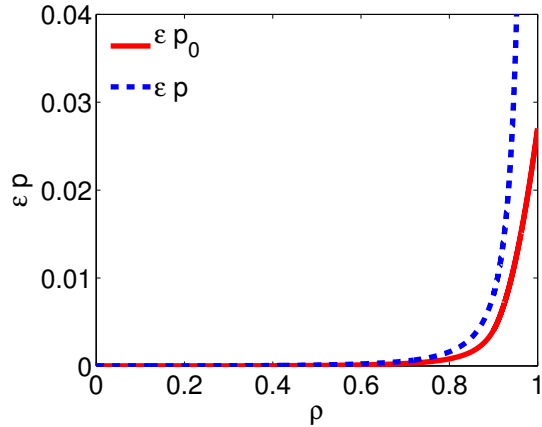}
\caption{The plots of $\varepsilon p$ and $\varepsilon p_0$ as functions of
  $\rho$ with $\varepsilon=10^{-4}, \gamma=2$. }
\label{fig:9}
 \end{figure}
 By the definition, the
Courant-Friedrich-Lewy (CFL) condition for the explicit part is
\begin{gather}
   \Delta t\leq  \frac{\sigma\Delta x}{\max \{|\frac{\boldsymbol{q}}{\rho}|+\sqrt{
       \varepsilon p_0'(\rho)}\}}, 
\end{gather}
where $\sigma$ is the Courant number. Since $\varepsilon p_0'$ is always bounded,
 the CFL condition can be satisfied uniformly in $\varepsilon$.

\subsection{The Direct method}
\label{sec:direct-method}
To get the solution, we will rewrite the above  scheme into another  form.  By
inserting \eqref{eq:2} into \eqref{eq:1}, we can  get an elliptic equation
for $\rho$:
\begin{gather}
  \begin{split}
    &\frac{\rho^{{n+1}} - \rho^{{n}}}{\Delta t} +
    \nabla_{\boldsymbol{x}}\cdot \boldsymbol{q}^{n}-\Delta t 
    \Delta_{\boldsymbol{x}}(\varepsilon p_1(\rho^{{n+1}})) -\Delta t
    \nabla_{\boldsymbol{x}}^2:\left(\frac{\boldsymbol{q}^{n}\otimes \boldsymbol{q}^n }{\rho^{n}}\right) - \Delta t\Delta_{\boldsymbol{x}}\left(\varepsilon p_0(\rho^{{n}})\right)= 0.\label{eq:3}
  \end{split}
\end{gather}
From this equation, we can solve $\rho^{n+1}$. However, if we solve
$\rho$ directly, the density constraint $\rho\leq \rho^*$ may not be
satisfied due to discretization errors. Thus, we write $\rho^{n+1}=\rho(p_1^{n+1})$ in \eqref{eq:3}
and solve the equation  in terms of $p_1$. The density constraint
$\rho\leq \rho^*$  will be automatically satisfied.  Moreover, the
positivity of $\rho$ can be ensured by the fact that the discretized
equation satisfies  the maximal principle. 

Once  $\rho^{n+1}$ is obtained, we can obtain $\boldsymbol{q}^{n+1}$ from
\eqref{eq:2} easily.
\begin{gather}\label{eq:22}
 \boldsymbol{q}^{n+1} = \boldsymbol{q}^{n} -{\Delta t} \left\{
   \nabla_{\boldsymbol{x}}\cdot\left(\frac{\boldsymbol{q}^n\otimes \boldsymbol{q}^n }{\rho^{n}}\right)+\nabla_{\boldsymbol{x}}\left(\varepsilon
     p_0(\rho^{n})\right) +  \nabla_{\boldsymbol{x}} (\varepsilon p_1(\rho^{n+1}))\right\}.
\end{gather}

\begin{rem}
  In the numerical simulation, we can also improve the accuracy by
implementing a fully implicit scheme, which iterates the above scheme
to solve \eqref{eq:1} and \eqref{eq:2} with $\frac{\boldsymbol{q}\otimes \boldsymbol{q}}{\rho}$
implicit. Suppose $\rho^{n+1,0}=\rho^n$ and $\boldsymbol{q}^{n+1,0}=\boldsymbol{q}^n$  and $\rho^{n+1,k}$, $\boldsymbol{q}^{n+1,k}$ are the solutions to
the following equations.
\begin{align}
&\frac{\rho^{{n+1},k+1} - \rho^{{n}}}{\Delta t} +
     \nabla_{\boldsymbol{x}}\cdot \boldsymbol{q}^{n}-\Delta t
    \Delta_{\boldsymbol{x}}(\varepsilon p_1(\rho^{{n+1},k+1})) \nonumber \\
&\qquad\qquad\qquad -\Delta t\
    \nabla_{ \boldsymbol{x}}^2:\left(\frac{\boldsymbol{q}^{n+1,k}\otimes \boldsymbol{q}^{n+1,k} }{\rho^{n+1,k}}\right)+\Delta_{ \boldsymbol{x}}\left(\varepsilon p_0(\rho^{{n}})\right)= 0,
  \label{eq:17}\\
&\boldsymbol{q}^{n+1,k+1} = \boldsymbol{q}^{n} -{\Delta t} \left\{
  \nabla_{\boldsymbol{x}}\left(\frac{\boldsymbol{q}^{n+1,k}\otimes
      \boldsymbol{q}^{n+1,k} }{\rho^{n+1,k}}\right)+
  \nabla_{\boldsymbol{x}}\left(\varepsilon p_0(\rho^{n})\right) +
  \nabla_{\boldsymbol{x}}(\varepsilon p_1(\rho^{n+1,k+1}))\right\}.\label{eq:18}
\end{align}
As $k\to \infty$, the solution approximates to the one solving the
fully implicit scheme (both in $\rho$ and $\frac{\boldsymbol{q}\otimes \boldsymbol{q}
}{\rho}$). This modification provides little improvement compared to
the additional computational  cost.
\end{rem}

\subsection{The Gauge method}
\label{sec:gauge-method}

Another way to implement the AP scheme is the Gauge method developed
in \cite{2007_Low_mach_D_SJ_JGL}. It can be obtained by applying the Gauge
decomposition
\begin{gather}
  \boldsymbol{q}=\boldsymbol{a}-\nabla_{\boldsymbol{x}} \varphi, \quad \nabla_{\boldsymbol{x}} \cdot \boldsymbol{a}=0
\end{gather}
where $\boldsymbol{a}$ is the incompressible part of field $\boldsymbol{q}$ and $\varphi$ is the
irrotational one. This decomposition is expected to be more robust
for capturing incompressibility constraint. We will see that it is
partially right. By including this decomposition into equations \eqref{eq:1} and
\eqref{eq:2}, we get 
{\allowdisplaybreaks
\begin{align}
  \begin{split}
    &\frac{\rho^{{n+1}} - \rho^{{n}}}{\Delta t} +
    \nabla_{\boldsymbol{x}}\cdot  \boldsymbol{q}^{n}-\Delta t
    \Delta_{\boldsymbol{x}}(\varepsilon p_1(\rho^{{n+1}})) \\
    &\hspace{3cm} -\Delta t
    \nabla_{\boldsymbol{x}}^2:\left(\frac{\boldsymbol{q}^{n}\otimes \boldsymbol{q}^{n}}{\rho^{n}}\right)-\Delta t \Delta_{\boldsymbol{x}} (\varepsilon p_0(\rho^{{n}})) = 0,
  \end{split}\label{eq:4}\\
     & \Delta_{\boldsymbol{x}}\varphi^{n+1}=\frac{1}{\Delta t}(\rho^{n+1}-\rho^{n}),\label{eq:5}\\
     &\Delta_{\boldsymbol{x}} P^{n+1}=-\nabla_{\boldsymbol{x}}^2: \left (\frac{\boldsymbol{q}^n\otimes \boldsymbol{q}^n }{\rho^n}\right)- \Delta_{\boldsymbol{x}}(\varepsilon p_0(\rho^{n})),\label{eq:6}\\
    &\frac{\boldsymbol{a}^{n+1}-\boldsymbol{a}^{n}}{\Delta t}+
    \nabla_{\boldsymbol{x}}\cdot \left(\frac{\boldsymbol{q}^n\otimes \boldsymbol{q}^n }{\rho^n}\right)+ \nabla_{\boldsymbol{x}} (\varepsilon
    p_0(\rho^{n})) + \nabla_{\boldsymbol{x}} P^{n+1}=0,\label{eq:7}\\
    &\boldsymbol{q}^{n+1} = \boldsymbol{a}^{n+1}- \nabla_{\boldsymbol{x}}\varphi^{n+1}.\label{eq:8}
   \end{align}
Indeed, the equation  \eqref{eq:4} for $p_1^{n+1}$ is derived from \eqref{eq:1} and \eqref{eq:2} similarly
as in the Direct method. The Laplace equation  \eqref{eq:5} for
$\varphi^{n+1}$ is the direct consequence of
applying the decomposition \eqref{eq:8} and $\nabla_{\boldsymbol{x}}
\cdot \boldsymbol{a}^{n+1}=0$ to the density equation \eqref{eq:1}. 
 The equation \eqref{eq:6} for $P^{n+1}$  and the equation \eqref{eq:7} for
 $\boldsymbol{a}^{n+1}$  are obtained by inserting
\eqref{eq:8} and $\nabla_{\boldsymbol{x}} \cdot
\boldsymbol{a}^{n+1}=\nabla_{\boldsymbol{x}} \cdot
\boldsymbol{a}^{n}=0$ into the momentum equation \eqref{eq:2}. Here a new unknown $P$ is introduced, which is  defined by
\begin{gather}
  P^{n+1}=\varepsilon p_1(\rho^{n+1})-\frac{\varphi^{n+1}-\varphi^{n}}{\Delta t},\label{eq:9}
\end{gather}
 since \eqref{eq:4}-\eqref{eq:5}  and \eqref{eq:7} imply \eqref{eq:2}.

The original equations \eqref{eq:1} and \eqref{eq:2} can also be recovered from
\eqref{eq:4}-\eqref{eq:8} by assuming $\nabla_{\boldsymbol{x}} \cdot \boldsymbol{a}^{n}=0$. In
fact, the equations for  $P^{n+1}$ and $\boldsymbol{a}^{n+1}$ (\eqref{eq:6}-\eqref{eq:7}) and $\nabla_{\boldsymbol{x}} \cdot \boldsymbol{a}^{n}=0$ imply
$\nabla_{\boldsymbol{x}} \cdot \boldsymbol{a}^{n+1}=0$. This leads to
the density equation \eqref{eq:1} from the $\varphi^{n+1}$ equation \eqref{eq:5}
and the decomposition \eqref{eq:8}. \eqref{eq:1} combined with the
$p_1^{n+1}$ equation \eqref{eq:4} will then allow us to recover the
momentum equation \eqref{eq:2}.

\paragraph{The boundary conditions} By solving the equations
\eqref{eq:4}-\eqref{eq:8} in sequential order, we can update the value
of $\rho,\boldsymbol{q}$. To do this, we need to provide boundary conditions for the
Laplace equations for $\varphi$ and $P$. The boundary conditions for
$P^{n+1}$ are somehow straightforward due to the implicit relation
\eqref{eq:9}, once $\varphi$ is known. And in solving $\varphi^{n+1}$, we may impose Dirichlet boundary condition on
$\varphi^{n+1}$ as follows:
\begin{gather}\label{eq:16}
  \varphi^{n+1}|_{\Omega}=0.
\end{gather}
Indeed, other non-homogeneous  Dirichlet boundary conditions can be
chosen. This  makes the unknown $\varphi^{n+1}$  determined up  to a linear
function in space. However, this uncertainty can be removed by
redefining $\boldsymbol{a}^{n+1}$. So we can always impose the homogeneous  Dirichlet boundary
conditions.

\paragraph{Simplification in the one dimensional case}
Observe that in the one dimensional case, $\nabla_{\boldsymbol{x}} \cdot a=0$ implies that $a$ is
independent of $x$. We may thus rewrite \eqref{eq:6} and \eqref{eq:7}
into a simpler equation by using \eqref{eq:9}
\begin{gather}
  a^{n+1}=a^{n}-\frac{\Delta t}{c-b}\left(
    \frac{(q^n )^2}{\rho^n}+ \varepsilon  p_0(\rho^{n})+\varepsilon p_1(\rho^{n+1})\right)\bigg|^c_b+\frac{1}{c-b}\left(\varphi^{n+1}|^c_b-\varphi^{n}|^c_b\right),\label{eq:20}
\end{gather}
where the space-domain is $[b,c]$ and $f|^c_b=f(c)-f(b)$. This
equation can be further simplified, since we impose the homogeneous  Dirichlet boundary
condition on $\varphi$ all the time. This lead to
\begin{gather}
  a^{n+1}=a^{n}-\frac{\Delta t}{c-b}\left(
    \frac{(q^n )^2}{\rho^n}+ \varepsilon  p_0(\rho^{n})+\varepsilon p_1(\rho^{n+1})\right)\bigg|^c_b.\label{eq:10}
\end{gather}
 With this reformulation, in the one dimensional case, we can
reduce the number of elliptic equations to be solved to one and update the space independent variable
$a$ more efficiently. Also as a consequence of \eqref{eq:16}, $a^n$ should be defined as the average of
$q^n$.
\begin{gather}
  a^n=\frac{1}{c-b}\int_b^c q^n dx.
\end{gather}
In summary, the Gauge method in the one dimensional case is implemented through equations
\eqref{eq:4}, \eqref{eq:5}, \eqref{eq:8} and \eqref{eq:10}. 

\subsection{Discussion of the AP property}

As for the AP property of the scheme, we give a formal
proof. To be precise, we want to show that the system \eqref{eq:3} and
\eqref{eq:2} becomes the incompressible Euler system as $\varepsilon\to
0$ in the congested region. 

However, in contrast with the  low Mach number limit of the isentropic Euler equation discussed in
\cite{2009_Degond_LowMach} or \cite{2007_Low_mach_D_SJ_JGL}, the
singularity in our model is embedded in the definition of $p$. New
congestion regions may arise from non-congested ones. And by the
analysis in section \ref{sec:one-dimens-riem}, there is the
possibility that $\rho^{\varepsilon}\to \rho^*$ but the limit of the
pressure $\varepsilon p_1(\rho^{\varepsilon})\to 0$ as $\varepsilon\to 0$. This means that
although the congestion density is reached, there is no real congestion in this region. So it seems better to characterize the congestion
region by
defining $\bar{p}=\lim_{\varepsilon\to 0}\varepsilon
p_1(\rho^{\varepsilon})>0$.  

Currently, we can only show that in regions
where both $\bar{p}^{n+1}=\lim_{\varepsilon\to 0}\varepsilon
p_1((\rho^{n+1})^{\varepsilon})>0$ and $\bar{p}^{n}=\lim_{\varepsilon\to 0}\varepsilon
p_1((\rho^{n})^{\varepsilon})>0$,  \eqref{eq:3} and
\eqref{eq:22} tend to the incompressible Euler system as $\varepsilon\to
0$. The assumption $\bar{p}^{n+1}=\lim_{\varepsilon\to 0}\varepsilon
p_1((\rho^{n+1})^{\varepsilon})>0$ is somehow essential given that
declustering wave  may appear in our model, since the pressure $\bar{p}$ can change from positive
value to $0$ instantaneously due to a declustering wave. The assumption $\bar{p}^{n}=\lim_{\varepsilon\to 0}\varepsilon
p_1((\rho^{n})^{\varepsilon})>0$ is also needed in case of the appearance
of new
congestion regions. In regions where $\bar{p}^{n+1}>0$ and $\bar{p}^{n}>0$,  we have
naturally that $(\rho^{n+1})^{\varepsilon}\to \rho^*$ and
$(\rho^{n})^{\varepsilon}\to \rho^*$ as $\varepsilon\to 0$.  Taking the
divergence of \eqref{eq:22} and using \eqref{eq:3}, we can indeed
recover \eqref{eq:1}, which leads to the
incompressibility of $(\boldsymbol{q}^{n+1})^{\varepsilon}$: $\nabla_{\boldsymbol{x}}\cdot(
\boldsymbol{q}^{n+1})^{n+1}=0$. Here the  implicitness in \eqref{eq:1} is crucial.
Then \eqref{Eq:Incomp_constraint} follows. Although we can only prove
the AP property inside a congestion region, the numerical solutions
provide evidence that the scheme is globally AP, including at
transition between compressible and incompressible region.

As for the Gauge method, it is also AP since it is a direct consequence of
\eqref{eq:1} and \eqref{eq:2}.

In the numerical simulations, we will check the AP property for concrete
test cases.

\section{ Full time and space
  discretization}
In this section, we present the one dimensional full time and space
  discretization. The two dimensional discretization is the easy extension of the one
dimensional case. For the sake of completeness, we include it in the appendix.
\label{sec:full-time-space}
\paragraph{The Direct method: } In the following, we consider the domain $[b,c]=[0,1]$. Let the uniform
spatial mesh be $\Delta x=\frac{1}{M}$, where $M$ is a positive
integer. Denote by $U_j^{n+1}=(\rho_j^{n},q_j^{n})^T$ the approximations of
$U=(\rho,q)^T$ at time  $t^n=n\Delta t$ and positions $x_j=j\Delta x$, for $j = 0,1,\ldots,M$. We fully
discretize the scheme \eqref{eq:1} and \eqref{eq:2} in the spirit of a
local Lax-Friedrichs (or  Rusanov) method \cite{leveque_book_2002}  as follows:

 \begin{gather}
   \begin{split}
     \frac{\rho^{n+1}_{j} - \rho^{n}_{j}}{\Delta t} + \frac{1}{\Delta
       x}\bigg[ & Q_{j+1/2}(U_{j}^{n},U_{j+1}^{n},U_{j}^{n+1},U_{j+1}^{n+1})
       \\
&-
       Q_{j-1/2}(U_{j-1}^{n},U_{j}^{n},U_{j-1}^{n+1},U_{j}^{n+1})\bigg]
     = 0,
   \end{split}
\label{Eq:rho_discrete}\\
\begin{split}
  \frac{q^{n+1}_{j} - q^{n}_{j}}{\Delta t} &+ \frac{1}{\Delta x}\left[F_{j+1/2}(U_{j}^{n},U_{j+1}^{n}) - F_{j-1/2}(U_{j-1}^{n},U_{j}^{n})\right]\\
  & + \frac{1}{2\Delta x}\left[\varepsilon
    p_1(\rho_{j+1}^{n+1}) - \varepsilon p_1(\rho_{j-1}^{n+1})\right] = 0.
\end{split}
\label{Eq:q_discrete}
\end{gather}
where the fluxes are given by:
\begin{align}
&Q_{j+1/2}^{n+1/2} = \frac{1}{2}\left[q_{j}^{n+1} + q_{j+1}^{n+1}\right] - \frac{1}{2}C_{j+1/2}^{n}(\rho_{j+1}^{n}  - \rho_{j}^{n}),\label{eq:23}\\
&F_{j+1/2}^n =
\frac{1}{2}\left[\frac{(q_{j}^{n})^{2}}{\rho_{j}^{n}} +
  \frac{(q_{j+1}^{n})^{2}}{\rho_{j+1}^{n}}+
  \varepsilon p_0(\rho_{j+1}^n)+ \varepsilon p_0(\rho_{j}^n)\right]- \frac{1}{2}C_{j+1/2}^{n}(q_{j+1}^{n}  - q_{j}^{n}).\label{eq:24}
\end{align}
They consist of the sum of a centered flux, implicit in \eqref{eq:23}
, explicit in \eqref{eq:24} and  biased terms introducing
diffusion. The quantity $C_{j+1/2}^{n}$ is the local diffusion
coefficient and is given by:
\begin{align}
 & C_{j+1/2}^{n} =
\max\left\{\left|\frac{q_{j}^{n}}{\rho_{j}^{n}}\right|+\sqrt{ \varepsilon p_0'(\rho_{j}^{n})},\left|\frac{q_{j+1}^{n}}{\rho_{j+1}^{n}}\right|+\sqrt{ \varepsilon p_0'(\rho_{j+1}^{n})}\right\}. 
\end{align}
It is defined as the local maximal characteristic
speed related to the explicit pressure $p_0$. Therefore, it remains
bounded as $\varepsilon$ goes to zero. The quantity $C_{j+1/2}^{n}$
provides a numerical viscosity which is needed for scheme
stability. We note that only the central discretization part of the
flux is taken implicit, while the numerical viscosity part is kept
explicit. In the momentum flux, only the part of the flux which
relates to pressure is taken implicit.

Based on this  discretization, we can apply the same strategy as
described in section \ref{sec:direct-method} to get an elliptic equation in $\rho$
by substituting \eqref{Eq:q_discrete} into \eqref{Eq:rho_discrete}:
\begin{gather}
  \begin{split}
    \frac{\rho^{n+1}_{j} - \rho^{n}_{j}}{\Delta t} &+ \frac{q_{j+1}^{n} - q_{j-1}^{n}}{2\Delta x}- \frac{\Delta t}{4\Delta x^{2}}\left[\varepsilon p_1(\rho_{j+2}^{n+1}) -2 \varepsilon p_1(\rho_{j}^{n+1}) + \varepsilon p_1(\rho_{j-2}^{n+1})\right]\\
    &- \frac{1}{2\Delta x}\left[C_{j+1/2}(\rho_{j+1}^{n} - \rho_{j}^{n}) - C_{j-1/2}(\rho_{j}^{n} - \rho_{j-1}^{n})\right]\\
    &- \frac{\Delta t}{2\Delta x^2}\left[F_{j+3/2}^{n} - F_{j+1/2}^{n}
      - F_{j-1/2}^{n} + F_{j-3/2}^{n}\right] = 0.
  \end{split}
\end{gather}
This equation is consistent with  equation \eqref{eq:3} of the Direct method. Then we get a
nonlinear equation for $p_1$:
\begin{gather}
   \begin{split}
    \rho((p_1)^{n+1}_{j})&  - \frac{\Delta t^2}{4\Delta x^{2}}\left[\varepsilon (p_1)_{j+2}^{n+1} -2 \varepsilon (p_1)_{j}^{n+1} + \varepsilon (p_1)_{j-2}^{n+1}\right]\\
   = \ \rho^{n}_{j}& -\frac{\Delta t}{2\Delta x}(q_{j+1}^{n} - q_{j-1}^{n})+\frac{\Delta t}{2\Delta x}\left[C_{j+1/2}(\rho_{j+1}^{n} - \rho_{j}^{n}) - C_{j-1/2}(\rho_{j}^{n} - \rho_{j-1}^{n})\right]\\
    &+ \frac{\Delta t^2}{2\Delta x^2}\left[F_{j+3/2}^{n} - F_{j+1/2}^{n}
      - F_{j-1/2}^{n} + F_{j-3/2}^{n}\right].
  \end{split}\label{eq:13}
\end{gather}
As mentioned before, we will use Newton iterations to solve this
nonlinear equation and  get $p_1^{n+1}$. The density $\rho^{n+1}$ is
then obtained by inverting the nonlinear function $p_1=p_1(\rho)$ with
another Newton iteration. Once $\rho^{n+1}$
is solved, $q^{n+1}$ can be obtained by
\begin{equation}
q^{n+1}_{j} = \Phi(U_{j-1}^{n},U_{j}^{n},U_{j+1}^{n})  -\frac{\Delta t}{2\Delta x}\left[\varepsilon p_1(\rho_{j+1}^{n+1}) - \varepsilon p_1(\rho_{j-1}^{n+1})\right],\label{Eq:q_discrete_2}
\end{equation}
with
\begin{equation}
\Phi(U_{j-1}^{n},U_{j}^{n},U_{j+1}^{n}) = q^{n}_{j} - \frac{\Delta t}{\Delta x}\left[F_{j+1/2}^n - F_{j-1/2}^n\right].
\end{equation}

\paragraph{The Gauge method} Also based on \eqref{Eq:rho_discrete} and \eqref{Eq:q_discrete}, we
can have the full time-space discretization of  the Gauge method. Indeed,
for the Gauge method, we need to discretize  \eqref{eq:13}. This leads
to
\begin{align}
    &\frac{1}{4\Delta x^{2}}\left[\varphi^{n+1}_{j+2} -2
      \varphi^{n+1}_{j} + \varphi^{n+1}_{j-2}\right] =\nonumber\\
    &\quad\quad\quad \frac{1}{\Delta t}(\rho^{n+1}_{j}-\rho^{n}_{j}) - \frac{1}{2\Delta
      x}\left[C_{j+1/2}(\rho_{j+1}^{n} - \rho_{j}^{n}) -
      C_{j-1/2}(\rho_{j}^{n} - \rho_{j-1}^{n})\right],\label{eq:21}\\
  &a^{n+1}= a^{n}-{\Delta t}\left( \frac{(q^n\otimes q^n )}{\rho^n}+ \varepsilon
    p_0(\rho^{n})+\varepsilon
    p_1(\rho^{n+1})\right)\bigg|^1_0\nonumber\\
  &\quad\quad\quad + \frac{\Delta t}{2}\sum_{1}^{M}\left[C_{j+1/2}(q_{j+1}^{n} - q_{j}^{n}) -
    C_{j-1/2}(q_{j}^{n} - q_{j-1}^{n})\right],\label{eq:14}\\
  &q^{n+1}_j = a^{n+1}-\frac{1}{2\Delta x}(\varphi^{n+1}_{j+1}-\varphi^{n+1}_{j-1}).\label{eq:12}
\end{align}

The above equations are  the direct consequences of
\eqref{eq:4}-\eqref{eq:5}, \eqref{eq:8} and \eqref{eq:10}. However, in the
numerical simulation, we will mainly test the schemes with
\eqref{eq:21} replaced by
   \begin{gather}
     \begin{split}
       \frac{1}{\Delta x^{2}}\left[\varphi^{n+1}_{j+1} -2
         \varphi^{n+1}_{j} + \varphi^{n+1}_{j-1}\right]
       =&\ \frac{1}{\Delta
         t}(\rho^{n+1}_{j}-\rho^{n}_{j})\\
       &- \frac{1}{2\Delta x}\left[C_{j+1/2}(\rho_{j+1}^{n} -
         \rho_{j}^{n}) - C_{j-1/2}(\rho_{j}^{n} -
         \rho_{j-1}^{n})\right],
     \end{split}\label{eq:11}
   \end{gather}
which may be justified as being the direct discretization of \eqref{eq:5}
with some numerical  viscosity added to the right hand side.
The stencils are different in two cases. We call the Gauge method with
stencil \eqref{eq:21} \textit{Gauge 2} method and the one with
stencil \eqref{eq:11}\textit{ Gauge 1} method, since they use grids
$\varphi_{j\pm 2}$  and $\varphi_{j\pm 1}$ respectively in addition to
$\varphi_j$. There is a big
difference in performance between the two discretizations. In fact, we will see in
the next section that the Gauge 2 method ( with \eqref{eq:21}) yields almost the same
numerical result as the Direct method, while the Gauge 1 method (with
\eqref{eq:11}) performs quite differently from the Direct method. This
may be partially due to the fact that the Gauge 1 method introduces more
diffusion than the Gauge 2 method, which can be seen by inserting the
Taylor expansion of $\varphi_{j\pm 1}^{n+1},\varphi_{j\pm 2}^{n+1}$ around
$x=j\Delta x$ into the discretizations:
\begin{gather*}
  \begin{split}
  &  \left(-\frac{1}{\Delta x^{2}}\left[\varphi^{n+1}_{j+1} -2
      \varphi^{n+1}_{j} + \varphi^{n+1}_{j-1}\right]\right)-\left(- \frac{1}{4\Delta
      x^{2}}\left[\varphi^{n+1}_{j+2} -2 \varphi^{n+1}_{j} +
      \varphi^{n+1}_{j-2}\right]\right)\\
&\quad\quad = \frac{1}{4}\frac{d^{4} }{d x^{4}}\varphi^{n+1}(j\Delta x)\Delta x^2+O(\Delta x^4).
  \end{split}
\end{gather*}

\paragraph{Numerical diffusion}
An important issue about the scheme is the numerical diffusion. From
the  equation \eqref{eq:13}, it can be seen that the diffusion for $\rho$
is of the order of
\begin{gather}
   \left(\frac{1}{2}(|u^n|+\sqrt{ \varepsilon p_0'(\rho^n)})\Delta x+\Delta t \varepsilon p_1'(\rho^n)\right)\Delta_{\boldsymbol{x}}\rho^n+\Delta t \Delta_{\boldsymbol{x}}
  (\rho^nu^n\otimes u^n).
\end{gather}
And similarly, by inserting \eqref{Eq:rho_discrete} into the pressure $p_1$
term in \eqref{Eq:q_discrete} we can see that the diffusion for $q$ is of the order of
\begin{gather}
  \left(\frac{1}{2}(|u^n|+\sqrt{ \varepsilon p_0'(\rho^n)})\right)\Delta
  x\Delta_{{x}}{q}^n+\Delta t\varepsilon p_1'(\rho^n)\Delta_{{x}}{q}^n.
\end{gather}
To damp out the oscillations in the
mass and momentum equations, the required numerical diffusion is \cite{leveque_book_2002}
\begin{gather}
  \left(\frac{1}{2}(|u^n|+\sqrt{ \varepsilon p'(\rho^n)})\right)\Delta
  x\Delta_{{x}}\rho^n, \left(\frac{1}{2}(|u^n|+\sqrt{ \varepsilon p'(\rho^n)})\right)\Delta
  x\Delta_{{x}}{q}^n,
\end{gather}
respectively. To ensure that this numerical diffusion is achieved, we
need the condition
\begin{gather}
  \left(\frac{1}{2}(|u^n|+\sqrt{ \varepsilon p_0'(\rho^n)})\Delta x+\varepsilon p_1'(\rho^n)\Delta t \right)\geq \left(\frac{1}{2}(|u^n|+\sqrt{ \varepsilon p'(\rho^n)})\right)\Delta
  x,
\end{gather}
which leads to
\begin{gather}
  \frac{\Delta t}{\Delta x}\geq \frac{1}{2\left(\sqrt{ \varepsilon p_0'(\rho^n)}+\sqrt{ \varepsilon p'(\rho^n)}\right)}.
\end{gather}
This condition is automatically satisfied in the congested region
($\rho\to \rho^*$) for small $\varepsilon$, since $\varepsilon
p'(\rho^n)\to \infty$ as $\varepsilon \to 0$. However, it contradicts the
CFL condition in the non-congested region for small $\varepsilon$,  since $\varepsilon
p'(\rho^n)$ and $\varepsilon
p_0'(\rho^n) \to 0$ as $\varepsilon \to 0$. From this analysis, there
should be no oscillations in the congestion region while the numerical
diffusion may not be sufficient in  the non-congested region. However, 
numerical simulation seems to  indicate that the numerical viscosity
in this scheme is sufficient to damp out the oscillations in the
non-congested region. 

   \section{Numerical results}
\label{sec:Num_result}

   \subsection{One dimensional test cases}\label{sec:1d-case-num}
      In this section, we use several numerical examples to test the
   performance of the schemes. Corresponding to different situations,
   four examples are tested. All these examples are the compositions
   of Riemann problems. Since the exact solutions to Riemann problem can be determined as
   in section \ref{sec:euler-with-cong}, we can compare the exact and numerical
   solutions. Different measurements of the relative errors will be
   applied to test the performance of our schemes. And the numerical Courant number is computed.

In the following, we choose $\gamma=2$ and the maximal density $\rho^*=1$. The test problems are:
\begin{align}
&(P1): (\rho,q)(x,0)=
    \begin{cases}
      (0.7,0.8), & x\in [0,0.5),\\
 (0.7,-0.8), &  x\in (0.5,1],
    \end{cases}\\
 &(P2): (\rho,q)(x,0)=
    \begin{cases}
      (0.7,-0.8), & x\in [0,0.5),\\
  (0.7,0.8), & x\in (0.5,1],
    \end{cases}\\ 
&(P3): (\rho,q)(x,0)=
    \begin{cases}
      (0.7,0.8), & x\in [0,0.25),\\
 (0.8,-0.3), & x\in (0.25,0.75),\\
 (0.7,-1.2), & x\in (0.75,1],
    \end{cases}\\
&(P4):
(\rho,q)(x,0)=
    \begin{cases}
      (0.8,0.3), & x\in [0,0.5),\\
 (0.5,0.1), &  x\in (0.5,1],
    \end{cases}
  \end{align}
The  first example $(P1)$ illustrates how  the AP schemes
capture shocks near congestion.  The second example $(P2)$ shows how
the AP schemes work near vacuum. The third example $(P3)$ simulates
the interaction of two shocks near congestion. The last example $(P4)$
shows
some problems in the Gauge 1 method and will be used to justify the
splitting of $p$.

\textbf{Example 1. } The solution to the Riemann problem $(P1)$
consists of two shocks propagating in the  opposite directions. The density  of the  intermediate state
is close to the maximal density. In the following, we will test the
two methods described in section \ref{sec:full-time-space} with
different parameters $\varepsilon$ and different mesh sizes $\Delta x,
\Delta t$.

\begin{enumerate}
\item First, we choose $\varepsilon=10^{-4}, \Delta x=5\times 10^{-3}, \Delta
  t=5\times 10^{-4}$. We will compare the performances of the two methods proposed
  in section \ref{sec:full-time-space}.
\begin{figure}
  \centering
  \includegraphics[scale=0.7]{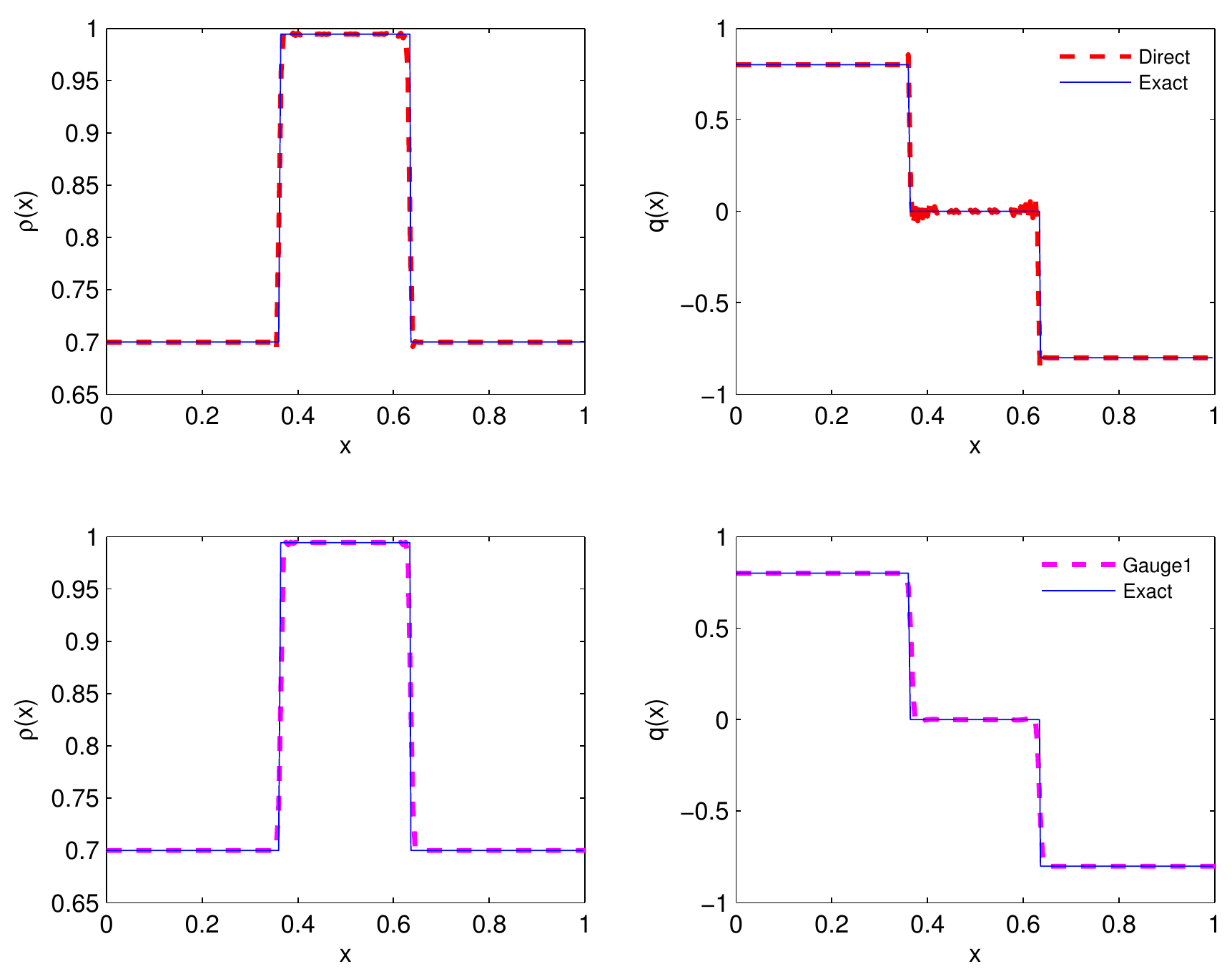}
\caption{The direct and  Gauge 1
  methods for problem (P1) at $t=0.05$ with $\varepsilon=10^{-4}, \Delta x=5\times 10^{-3}, \Delta
  t=5\times 10^{-4}$. The solid
  lines are the exact solutions. The dashed curves are the
  numerical solutions. The left graphs are for $\rho$, the right ones
  for $q$, both as functions of $x$. }
\label{fig:1}
 \end{figure}
It can be seen from Figure \ref{fig:1} that there is large oscillations of the momentum in the
congested region. But the propagation of the  shock is captured well. In comparison,
the Gauge 1 method as illustrated in Figure \ref{fig:1} eliminates all
the oscillation. 

\item We  look how the choices of the parameters $\varepsilon$ 
  $\Delta x$ and $ \Delta t$ affect the numerical result. We may fix
  $\Delta x, \Delta t$ but change the value of $\varepsilon$ so as to
  test the cases $\varepsilon<\Delta t$ and $\varepsilon>\Delta t$.
\begin{figure}
  \centering
\subfigure[$\varepsilon=10^{-2}$]{\label{fig:3:a} \includegraphics[scale=0.7]{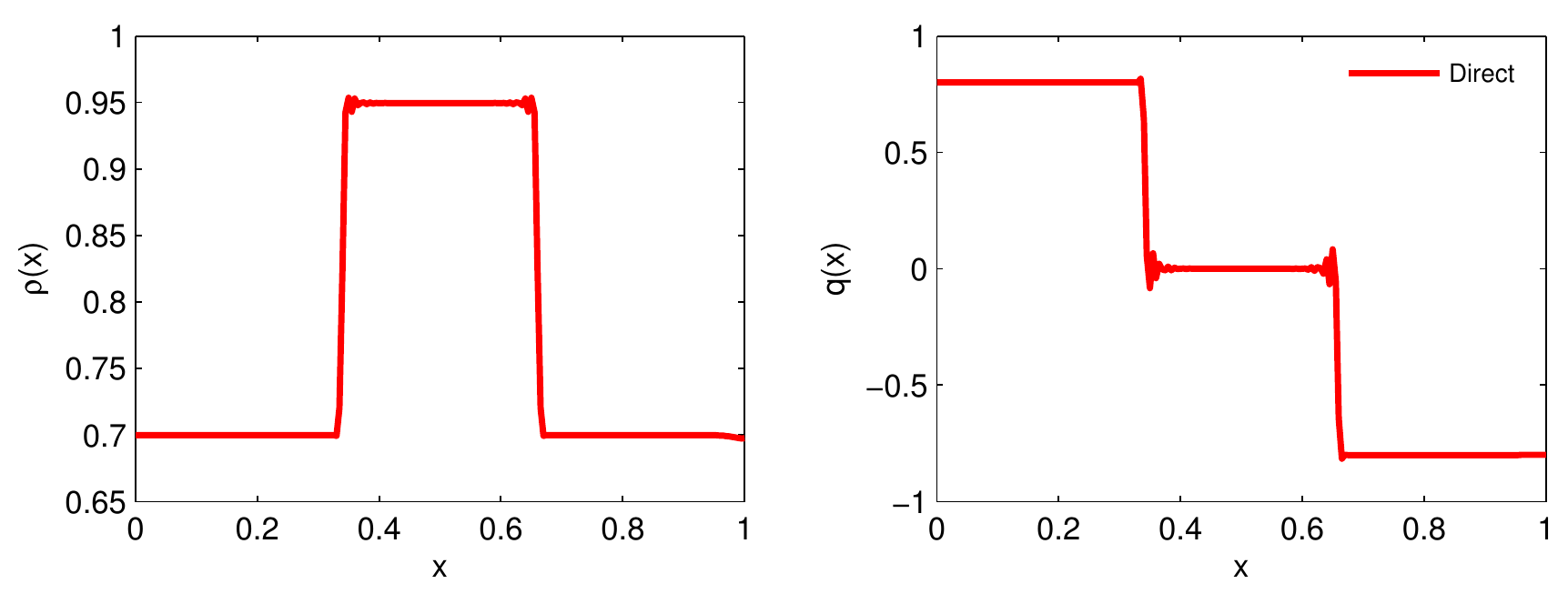}}\\
\subfigure[$\varepsilon=10^{-8}$]{\label{fig:3:b} \includegraphics[scale=0.7]{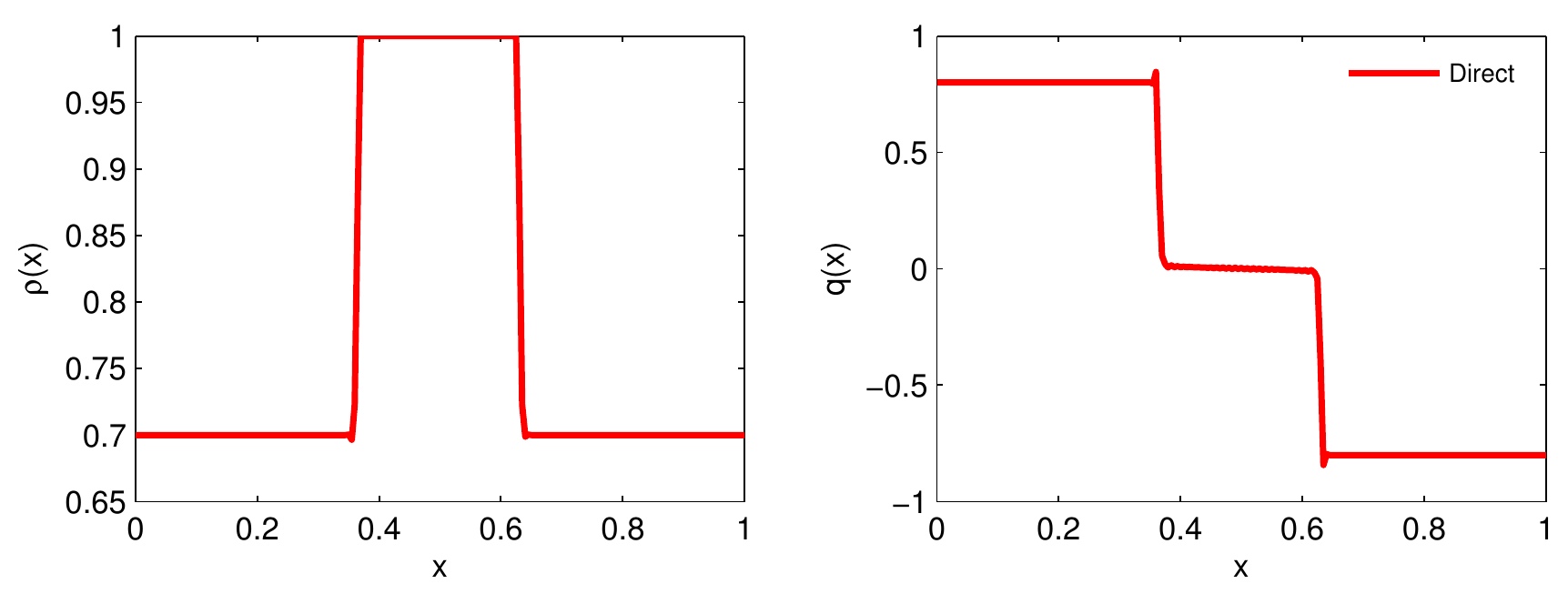}}
\caption{Fix $ \Delta x=5\times 10^{-3}, \Delta
  t=5\times 10^{-4}$. The numerical results of the Direct method for
  problem (P1) at $t=0.05$
  with $\varepsilon=10^{-2}$ and $\varepsilon=10^{-8}$. The left graphs are for $\rho$, the right ones
  for $q$, both as functions of $x$.}
 \label{fig:3}
 \end{figure}
From the numerical results, it can be seen that the oscillations in the
momentum always appear for different choices of $\varepsilon$ but are
smaller as $\varepsilon\to 0$ for this choice of parameter. This verifies the AP property. As for the Gauge 1
method, it has the same performance for all value of $\varepsilon$. Thus it
shares the same property.

\item The above observation can be quantitatively investigated by
  measuring the difference between  the numerical solution $W$
  and the theoretical one $w$. We  use two measurements: one is  the relative error of the numerical solution $W$
  compared with  $w$ in the sense of $L^1$ norm and the other is 
  the difference of their total variation:
 \begin{gather}
    e(W)=\frac{\norm{W-w}_{L^1}}{\norm{w}_{L^1}}, \text{ where
    }\norm{w}_{L^1}=\frac{1}{M}\sum_j\abs{w_j},\\
g(W)=\frac{\abs{\tv\{W\}-\tv\{w\}}}{\tv\{w\}}, \text{ where
    }\tv\{w\}=\sum_j\abs{w_{j+1}-w_j}.
  \end{gather}

\begin{table}
    \centering
    \begin{tabular}{ccccccc}
\toprule
\multicolumn{3}{c}{Parameters} &\multicolumn{2}{c}{Direct}&\multicolumn{2}{c}{Gauge 1}\\
\cmidrule(lr{.5em}){1-3}\cmidrule(lr{.5em}){4-5}\cmidrule(lr{.5em}){6-7}
      $\varepsilon$ & $\Delta x$ & $\Delta t$ & $e(W)$  &ratio & $e(W)$   &ratio    \\
\midrule
\multirow{10}{*}{$10^{-4}$}&  $1/200$ & $1/250$& $1.1012\times 10^{-2}$&-& $1.1209\times 10^{-2}$&-\\
&  $1/200$ & $1/500$& $8.0103\times 10^{-3}$& 1.3747&
$7.8670\times 10^{-3}$&1.4248\\
&  $1/200$ & $1/1000$& $4.2631\times 10^{-3}$&1.8790&
$4.9486\times 10^{-3}$&1.5897\\
&  $1/200$ & $1/2000$& $5.1528\times 10^{-3}$&0.8273&
$5.4107\times 10^{-3}$&0.9146\\
&  $1/200$ & $1/10000$& $5.6843\times 10^{-3}$&0.9065&
$1.5761\times 10^{-2}$&0.3433\\
\cmidrule(lr{.5em}){2-7}
&  $1/200$ & $1/1000$& $4.2631\times 10^{-3}$&-&
$4.9486\times 10^{-3}$&-\\
&  $1/400$ & $1/2000$& $3.5612\times 10^{-3}$&1.1971&
$3.1706\times 10^{-3}$&1.5608\\
&  $1/800$ & $1/4000$& $1.3085\times 10^{-3}$&2.7216&
$1.4253\times 10^{-3}$&2.2245\\
&  $1/1600$ & $1/8000$& $5.7676\times 10^{-4}$&2.2687&
$6.1850\times 10^{-4}$&2.3044\\
&  $1/1600$ & $1/16000$& $7.2302\times 10^{-4}$&0.79771&
$7.0713\times 10^{-4}$&0.87466\\
\midrule
\multirow{10}{*}{$10^{-2}$}&  $1/200$ & $1/250$& $1.3188\times 10^{-2}$& -&$1.3721\times 10^{-2}$&-\\
&  $1/200$ & $1/500$& $7.7490\times 10^{-3}$&1.7019&
$8.5433\times 10^{-3}$&1.6061\\
&  $1/200$ & $1/1000$& $7.6793\times 10^{-3}$&1.0091&
$7.8588\times 10^{-3}$&1.0871\\
&  $1/200$ & $1/2000$& $8.2751\times 10^{-3}$&0.9280&
$7.5447\times 10^{-3}$&1.0416\\
&  $1/200$ & $1/10000$& $1.0758\times 10^{-2}$&0.7692&
$1.5761\times 10^{-2}$&0.4787\\
\cmidrule(lr{.5em}){2-7}
&  $1/200$ & $1/1000$& $7.6793\times 10^{-3}$&-&
$7.8588\times 10^{-3}$&-\\
&  $1/400$ & $1/2000$& $3.7602\times 10^{-3}$&2.0423&
$3.8637\times 10^{-3}$&2.0340\\
&  $1/800$ & $1/4000$& $1.7934\times 10^{-3}$&2.0967&
$1.8646\times 10^{-3}$&2.0721\\
&  $1/1600$ & $1/8000$& $8.0483\times 10^{-4}$&2.2283&
$8.7074\times 10^{-4}$&2.1414\\
&  $1/1600$ & $1/16000$& $7.9308\times 10^{-4}$&1.0148&
$8.4647\times 10^{-4}$&1.0287\\
\bottomrule
    \end{tabular}
   \caption{Comparison of the $L^1$  relative error between the Direct and Gauge 1 methods at
      $t=0.025$. The 'ratio' column ratio provides comparisons of  the relative $L^1$
      norm error between  the previous and current rows, where
      either $\Delta x$ is fixed or $\Delta x/ \Delta t$ is fixed.}\label{tab:1}
  \end{table}
\begin{table}
    \centering
    \begin{tabular}{ccccccc}
\toprule
\multicolumn{3}{c}{Parameters} &\multicolumn{2}{c}{Direct}&\multicolumn{2}{c}{Gauge 1}\\
\cmidrule(lr{.5em}){1-3}\cmidrule(lr{.5em}){4-5}\cmidrule(lr{.5em}){6-7}
      $\varepsilon$ & $\Delta x$ & $\Delta t$  & $g(W)$ &ratio  & $g(W)$  &ratio    \\
\midrule
\multirow{10}{*}{$10^{-4}$}&  $1/200$ & $1/250$&  $7.4909\times 10^{-3}$&-&$7.2225\times 10^{-4}$&-\\
&  $1/200$ & $1/500$& $7.5636\times 10^{-2}$& 0.0990&
$6.7956\times 10^{-3}$&0.1063\\
&  $1/200$ & $1/1000$& $4.1433\times 10^{-1}$&0.1826&
$2.4506\times 10^{-2}$&0.2773\\
&  $1/200$ & $1/2000$& $8.5937\times 10^{-1}$&0.4821&
$1.4766\times 10^{-2}$&1.6596\\
&  $1/200$ & $1/10000$& 1.0603&0.8105&
$1.8995\times 10^{-2}$&0.7774\\
\cmidrule(lr{.5em}){2-7}
&  $1/200$ & $1/1000$& $4.1433\times 10^{-1}$&-&
$2.4506\times 10^{-2}$&-\\
&  $1/400$ & $1/2000$& $4.0126\times 10^{-1}$&1.0326&
$3.8585\times 10^{-2}$&0.6351\\
&  $1/800$ & $1/4000$& $2.5996\times 10^{-1}$&1.5435&
$2.4447\times 10^{-2}$&1.5783\\
&  $1/1600$ & $1/8000$& $5.4191\times 10^{-1}$&0.4797&
$2.3520\times 10^{-2}$&1.0394\\
&  $1/1600$ & $1/16000$& 1.0294&0.5264&
$1.3054\times 10^{-2}$&1.8017\\
\midrule
\multirow{10}{*}{$10^{-2}$}&  $1/200$ & $1/250$& $5.1377\times 10^{-5}$& -&$4.1271\times 10^{-4}$&-\\
&  $1/200$ & $1/500$& $5.2584\times 10^{-3}$&0.0098&
$2.1808\times 10^{-4}$&1.8925\\
&  $1/200$ & $1/1000$& $1.1354\times 10^{-1}$&0.0463&
$2.8711\times 10^{-3}$&0.0760\\
&  $1/200$ & $1/2000$& $4.9598\times 10^{-1}$&0.2289&
$9.9861\times 10^{-3}$&0.2875\\
&  $1/200$ & $1/10000$& 1.2766&0.3885&
$4.4336\times 10^{-3}$&2.2524\\
\cmidrule(lr{.5em}){2-7}
&  $1/200$ & $1/1000$& $1.1354\times 10^{-1}$&-&
$2.8711\times 10^{-3}$&-\\
&  $1/400$ & $1/2000$& $1.1856\times 10^{-1}$&0.9577&
$2.9617\times 10^{-3}$&0.9694\\
&  $1/800$ & $1/4000$& $1.2408\times 10^{-1}$&0.9555&
$2.9916\times 10^{-3}$&0.9900\\
&  $1/1600$ & $1/8000$& $1.2200\times 10^{-1}$&1.0170&
$2.7207\times 10^{-3}$&1.0996\\
&  $1/1600$ & $1/16000$& $3.9234\times 10^{-1}$&0.3110&
$1.3249\times 10^{-2}$&0.2054\\
\bottomrule
    \end{tabular}
    \caption{Comparison of the total variation relative error between the
       Direct and Gauge 1 methods at
      $t=0.025$. The 'ratio' column provides comparisons of the total variation relative error
      $g(W)$  between  the previous and current rows, where
      either $\Delta x$ is fixed or $\Delta x/ \Delta t$ is fixed.}\label{tab:4}
  \end{table}
In  Table \ref{tab:1} and ~\ref{tab:4}, the relative error in terms
of $L^1$ distance and total variation between the numerical and exact solutions
      at $t=0.025$ are computed. The Gauge 2 method yields quite similar
      result to the
      Direct method. So it is not listed in the table. It can be seen that the Direct method is usually better
than the Gauge 1 method in the $L^1$ norm. To reflect the observation
we made from looking at Figure \ref{fig:1}, we use the total variation norm, which
captures the oscillations better. For the total variation measurement,
it can be seen that the Gauge 1 method is always better in controlling the
oscillations. Another observation can be made from the tables is how
the accuracy is changed with different parameters.  For both two measurements, we test the accuracy with
$\Delta x$ fixed or $\Delta x/ \Delta t$ fixed. In the test where
$\Delta x$ is fixed,
it can also be seen that the relative error in $L^1$
norm can not be reduced much by refining the time mesh from $1/500$ to
$1/10000$ with fixed $\Delta x=1/200$. This feature is the same as the
standard hyperbolic solvers:  better accuracy can not be obtained by
using a smaller $\Delta t$ once the scheme is stable.  In the test
where $\Delta x/ \Delta t$ is fixed, we check how the relative error
is decreasing  with respect to $\Delta x$. It is interesting to see
that the error is not always decreasing. And since the convergence
order for explicit local Lax-Friedrichs scheme with discontinuities is
$\frac{1}{2}$ \cite{leveque_book_2002}, we may expect that the relative error in $L^1$
norm is reduced by $\sqrt{2}$ when the space mesh is refined
by 2 with $\Delta x/ \Delta t$  fixed. However, this is not the case.

\item We will also check the numerical Courant number in tables
  \ref{tab:2} and \ref{tab:3}. It can be seen that the CFL condition
  of our scheme is
  greatly improved from the one for standard hyperbolic
  solver $\Delta t=O(\varepsilon \Delta x)$.
  \begin{table}
    \centering
    \begin{tabular}{ccccc}
\toprule
      $\varepsilon$ & $\Delta x$& $\text{stable  }\Delta t$ & $\max\lambda$  & 
      $(\max\lambda)\frac{\Delta t}{\Delta x}$ \\
\midrule
$10^{-4}$&  $1/100$ & $1/100$& 1.2186& 1.2186\\
$10^{-4}$&  $1/200$ & $1/210$& 1.2669& 1.2066\\
$10^{-4}$&  $1/400$ & $1/430$& 1.1962& 1.1128\\
$10^{-4}$&  $1/800$ & $1/870$& 1.1962& 1.1000\\
$10^{-4}$&  $1/1600$ & $1/1730$& 1.1938& 1.1041\\
\midrule
$10^{-2}$&  $1/100$ & $1/100$& 1.7903& 1.4902\\
$10^{-2}$&  $1/200$ & $1/280$& 1.6571& 1.1836\\
$10^{-2}$&  $1/400$ & $1/580$& 1.6486& 1.1370\\
$10^{-2}$&  $1/800$ & $1/1170$& 1.6486& 1.1272\\
$10^{-2}$&  $1/1600$ & $1/2340$& 1.6645& 1.1381\\
\bottomrule
    \end{tabular}
    \caption{The numerical Courant number for the Direct method. $\max\lambda$
    is the maximal eigenvalue of the explicit part of the scheme for all time steps
    before the waves reach the boundary.}\label{tab:2}
  \end{table}

\begin{table}
    \centering
    \begin{tabular}{ccccc}
\toprule
      $\varepsilon$ & $\Delta x$& $\text{stable  }\Delta t$ & $\max\lambda$  & 
      $\max\lambda\frac{\Delta t}{\Delta x}$ \\
\midrule
$10^{-4}$&  $1/100$ & $1/100$& 1.2151& 1.2151\\
$10^{-4}$&  $1/200$ & $1/210$& 1.2279& 1.1695\\
$10^{-4}$&  $1/400$ & $1/500$& 1.3482& 1.0786\\
$10^{-4}$&  $1/800$ & $1/1340$& 1.4398& 0.8596\\
$10^{-4}$&  $1/1600$ & $1/3130$& 1.3555& 0.69293\\
\midrule
$10^{-2}$&  $1/100$ & $1/100$& 1.6480& 1.6480\\
$10^{-2}$&  $1/200$ & $1/260$& 1.6475& 1.2673\\
$10^{-2}$&  $1/400$ & $1/540$& 1.6475& 1.2204\\
$10^{-2}$&  $1/800$ & $1/1190$& 1.6530& 1.1112\\
$10^{-2}$&  $1/1600$ & $1/2520$& 1.6490& 1.0470\\
\bottomrule
    \end{tabular}
    \caption{The numerical Courant number for the Gauge 1 method. $\max\lambda$
    is the maximal eigenvalue of the explicit part of the scheme for all time steps
    before the waves reach the boundary.}\label{tab:3}
  \end{table}

\item We can also quantify the observation as in Figure
  \ref{fig:3}. In Table \ref{tab:6}, the relative errors of solutions
  in the $L^1$ norm or total variation are summarized.  The data for
  $\varepsilon=10^{-2},10^{-4}$ are the same as those in Table
  \ref{tab:1} and \ref{tab:4}. It can be seen that there is no big
  increase in the relative error  for different $\varepsilon$ with fixed
  $\Delta x,\Delta t$. As
  discussed in section \ref{sec:one-dimens-riem}, the theoretical
  solutions to system \eqref{Eq:rho_1D_eps} and \eqref{Eq:q_1D_eps} with
  positive   $\varepsilon$ are
  converging to the solutions to systems with $\varepsilon=0$ when
  $\varepsilon\to 0$. Thus, the numerical solutions tend to the
  theoretical solutions  systems with $\varepsilon=0$ for fixed
  $\Delta x,\Delta t$ as well. This versifies the AP property.
\begin{table}
    \centering
    \begin{tabular}{ccccccc}
\toprule
\multicolumn{3}{c}{Parameters} &\multicolumn{2}{c}{Direct}&\multicolumn{2}{c}{Gauge 1}\\
\cmidrule(lr{.5em}){1-3}\cmidrule(lr{.5em}){4-5}\cmidrule(lr{.5em}){6-7}
      $\varepsilon$ & $\Delta x$ & $\Delta t$ & $e(W)$  &$g(W)$ & $e(W)$   &$g(W)$    \\
\midrule
\multirow{4}{*}{$10^{-2}$}  
&  $1/200$ & $1/1000$& $7.6793\times 10^{-3}$&$1.1354\times 10^{-1}$&
$7.8588\times 10^{-3}$&$2.8711\times 10^{-3}$\\
&  $1/200$ & $1/2000$& $8.2751\times 10^{-3}$&$4.9598\times 10^{-1}$&
$7.5447\times 10^{-3}$&$9.9861\times 10^{-3}$\\
&  $1/800$ & $1/4000$& $1.7934\times 10^{-3}$&$1.2408\times 10^{-1}$&
$1.8646\times 10^{-3}$&$2.9916\times 10^{-3}$\\
&  $1/1600$ & $1/16000$& $7.9308\times 10^{-4}$&$3.9234\times 10^{-1}$&
$8.4647\times 10^{-4}$&$1.3249\times 10^{-2}$\\
\midrule
\multirow{4}{*}{$10^{-4}$}&    $1/200$ & $1/1000$& $4.2631\times 10^{-3}$&$4.1433\times 10^{-1}$&
$4.9486\times 10^{-3}$&$2.4506\times 10^{-2}$\\
&  $1/200$ & $1/2000$& $5.1528\times 10^{-3}$&$8.5937\times 10^{-1}$&
$5.4107\times 10^{-3}$&$1.4766\times 10^{-2}$\\
&  $1/800$ & $1/4000$& $1.3085\times 10^{-3}$&$2.5996\times 10^{-1}$&
$1.4253\times 10^{-3}$&$2.4447\times 10^{-2}$\\
&  $1/1600$ & $1/16000$& $7.2302\times 10^{-4}$&1.0294&
$7.0713\times 10^{-4}$&$1.3054\times 10^{-2}$\\
\midrule
\multirow{4}{*}{$10^{-8}$}&    $1/200$ & $1/1000$& $7.0600\times 10^{-3}$&$6.4521\times 10^{-1}$&
$5.8159\times 10^{-3}$&$1.4861\times 10^{-2}$\\
&  $1/200$ & $1/2000$& $5.8872\times 10^{-3}$&$6.4091\times 10^{-3}$&
$5.8872\times 10^{-3}$&$6.4091\times 10^{-3}$\\
&  $1/800$ & $1/4000$& $1.7611\times 10^{-3}$&$6.7218\times 10^{-1}$&
$1.4624\times 10^{-3}$&$1.5544\times 10^{-2}$\\
&  $1/1600$ & $1/16000$& $9.1296\times 10^{-4}$&$8.2410\times 10^{-1}$&
$8.0587\times 10^{-4}$&$1.6092\times 10^{-2}$\\
\bottomrule
    \end{tabular}
   \caption{ The  $L^1$  and  total variation relative error of     the  Direct and Gauge 1 methods at
      $t=0.025$ for different $\varepsilon$.}\label{tab:6}
  \end{table}

\item With a slightly variant version of this test case, we can also
  see the role played by the splitting of pressure
$p$. By making the momentum of the left and right states 10 times
smaller, we have another test case $(P1')$ sharing the similar behaviour of
$(P1)$.
\begin{gather}
(P1'): (\rho,q)(x,0)=
    \begin{cases}
      (0.7,0.08), & x\in [0,0.5),\\
 (0.7,-0.08), &  x\in (0.5,1],
    \end{cases}
  \end{gather}
  In Figure \ref{fig:10}, the numerical solutions obtained by the Direct method
with and without $p_0$ are compared. This confirms the observation
made in \cite{2009_Degond_LowMach} for the low Mach number limit.
\begin{figure}
  \centering
  \includegraphics[scale=0.7]{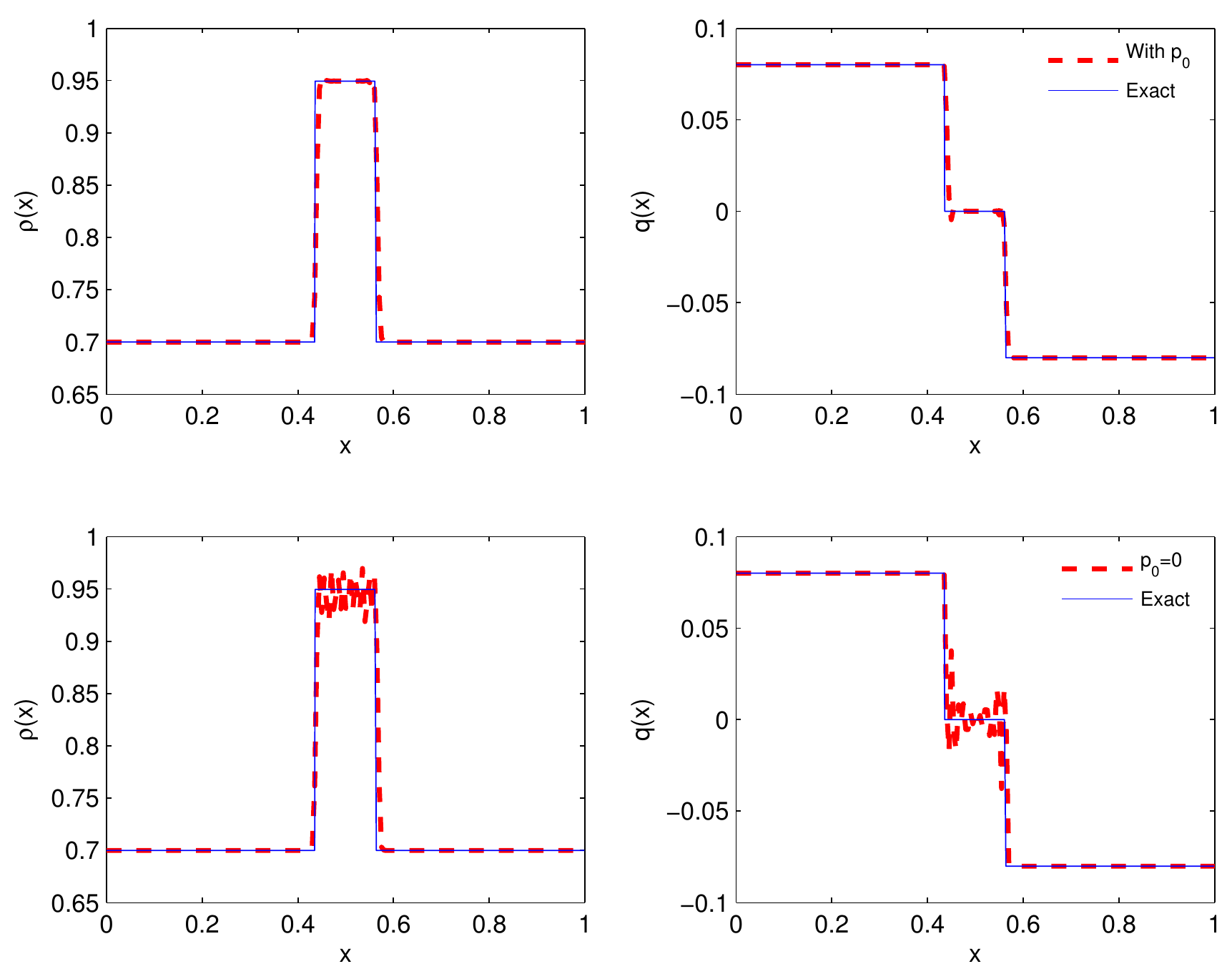}
\caption{The numerical results of the Direct method with and without
  $p_0$ for problem (P1') at $t=0.2$ with $\varepsilon=10^{-4}, \Delta x=5\times 10^{-3}, \Delta
  t=5\times 10^{-4}$. The solid
  lines are the exact solutions. The dashed curves are the
  numerical solutions. The left graphs are for $\rho$, the right ones
  for $q$, both as functions of $x$.}\label{fig:10}
 \end{figure}
It can be seen that  a lot of extra oscillations appear in the numerical
solutions when there is no splitting of pressure p ($p_0=0,p_1=p$). That is the
reason why we should add a $p_0$ term in the explicit part of the
numerical scheme.
\end{enumerate}

\textbf{Example 2. } 
The Riemann problem  $(P2)$ is obtained by exchanging the left and right
states of Riemann problem $(P1)$. So 
the solution to the problem $(P2)$
consists of two rarefaction
waves and a vacuum state appears as the intermediate state. As shown
in proposition  \ref{Prop:Riemann_problem_limit1}, these two rarefaction waves tend
to be contact waves. 

 As in example 1, We choose $\varepsilon=10^{-4}, \Delta x=5\times 10^{-3}, \Delta
  t=5\times 10^{-4}$.
  t=0.0005. 
\begin{figure}
  \centering
\includegraphics[scale=0.7]{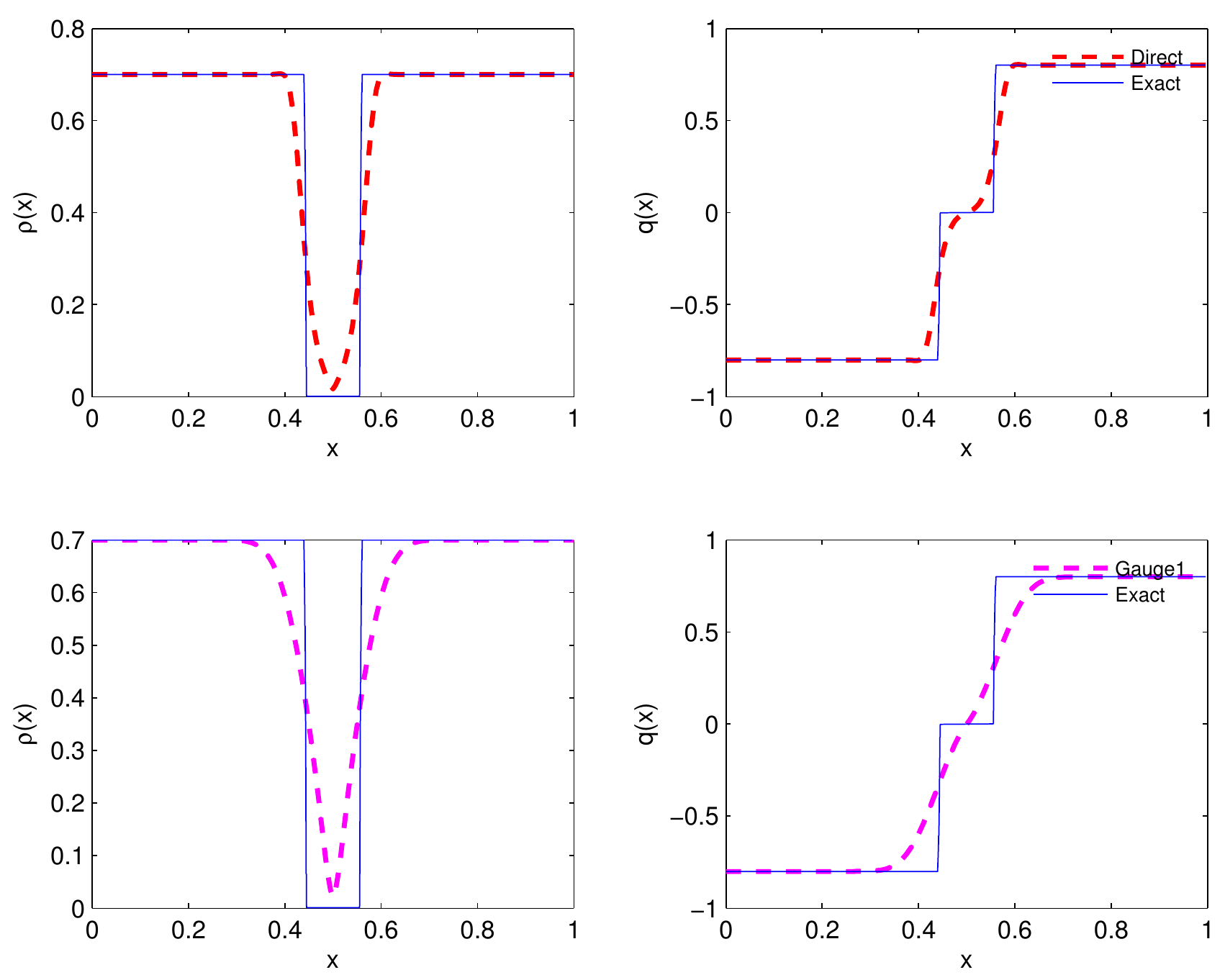}
\caption{The Direct and  Gauge 1
  methods for problem (P2) at $t=0.05$ with $\varepsilon=10^{-4}, \Delta x=5\times 10^{-3}, \Delta
  t=5\times 10^{-4}$. The solid
  lines are the exact solutions. The dashed curves are the
  numerical solutions. The left graphs are for $\rho$, the right ones
  for $q$, both as functions of $x$.}
\label{fig:4}
 \end{figure}
It can be seen from Figure \ref{fig:4} that the Direct method captures the vacuum and rarefaction
waves well. In comparison,
the Gauge 1 method as illustrated in Figure \ref{fig:4} shows larger diffusion.

\textbf{Example 3. } 
The solution to the problem $(P3)$
consists of two Riemann problems: both of them are like the Riemann
problem in (P1).  So there are two congested regions and eventually
they will collide. We are interested in observing how the numerical
methods behave at collision. We fix $\Delta
x=5\times 10^{-3}, \Delta
  t=5\times 10^{-4}$ and choose $\varepsilon= 10^{-4}$ and $10^{-8}$. Since
  only shocks are involved, we will use the Gauge 1 method only.

\begin{figure}
  \centering
\subfigure[$t=0.0790$]{\label{fig:5:a} \includegraphics[scale=0.7]{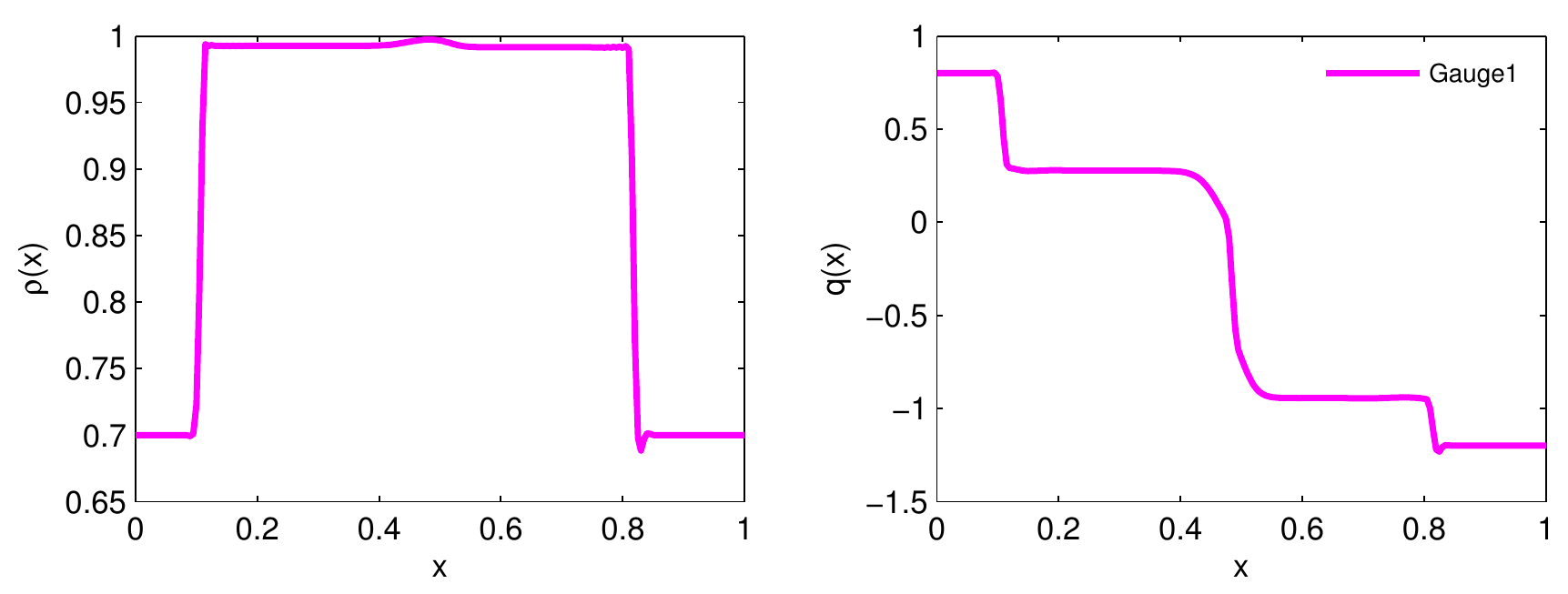}}\\
\subfigure[$t=0.1030$]{\label{fig:5:b} \includegraphics[scale=0.7]{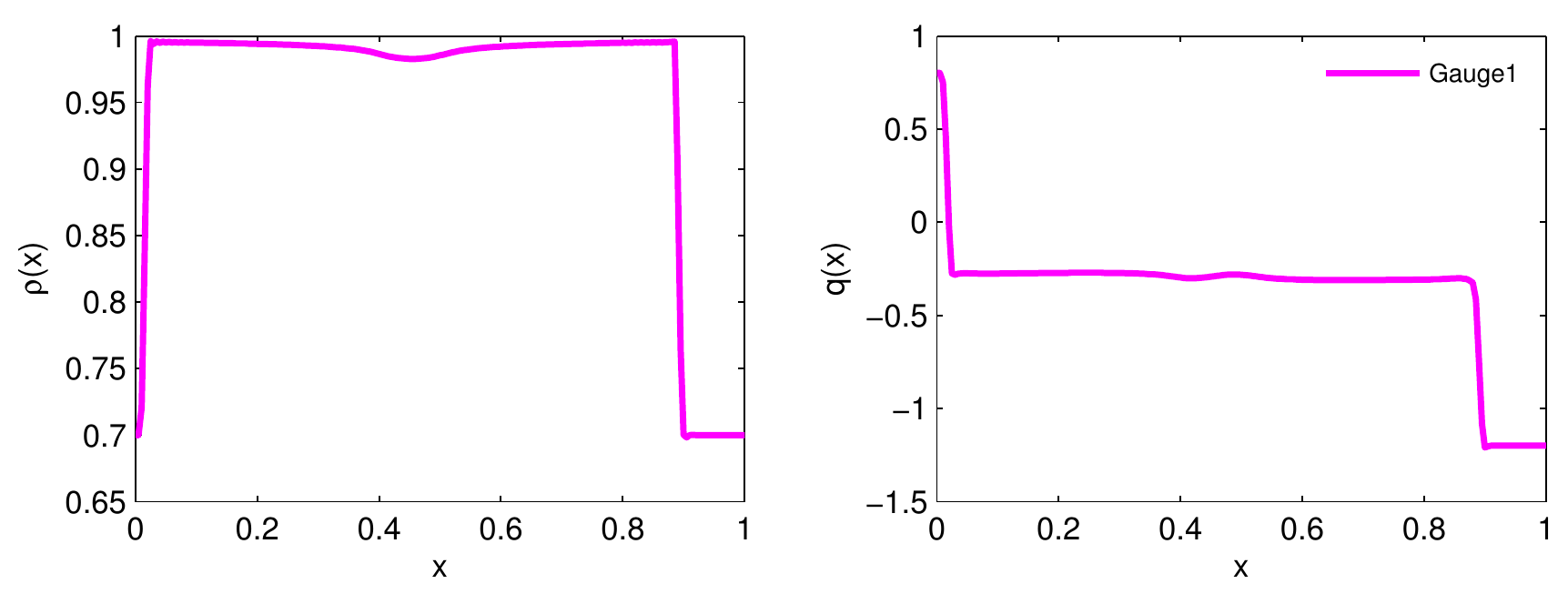}}
\caption{Gauge 1 method for problem (P3) with
  $\varepsilon=10^{-4}, \Delta x=5\times 10^{-3}, \Delta
  t=5\times 10^{-4}$.  The left graphs are for $\rho$, the right ones
  for $q$, both as functions of $x$.}
\label{fig:5}
 \end{figure}

\begin{figure}
  \centering
\subfigure[$t=0.0810$]{\label{fig:6:a} \includegraphics[scale=0.7]{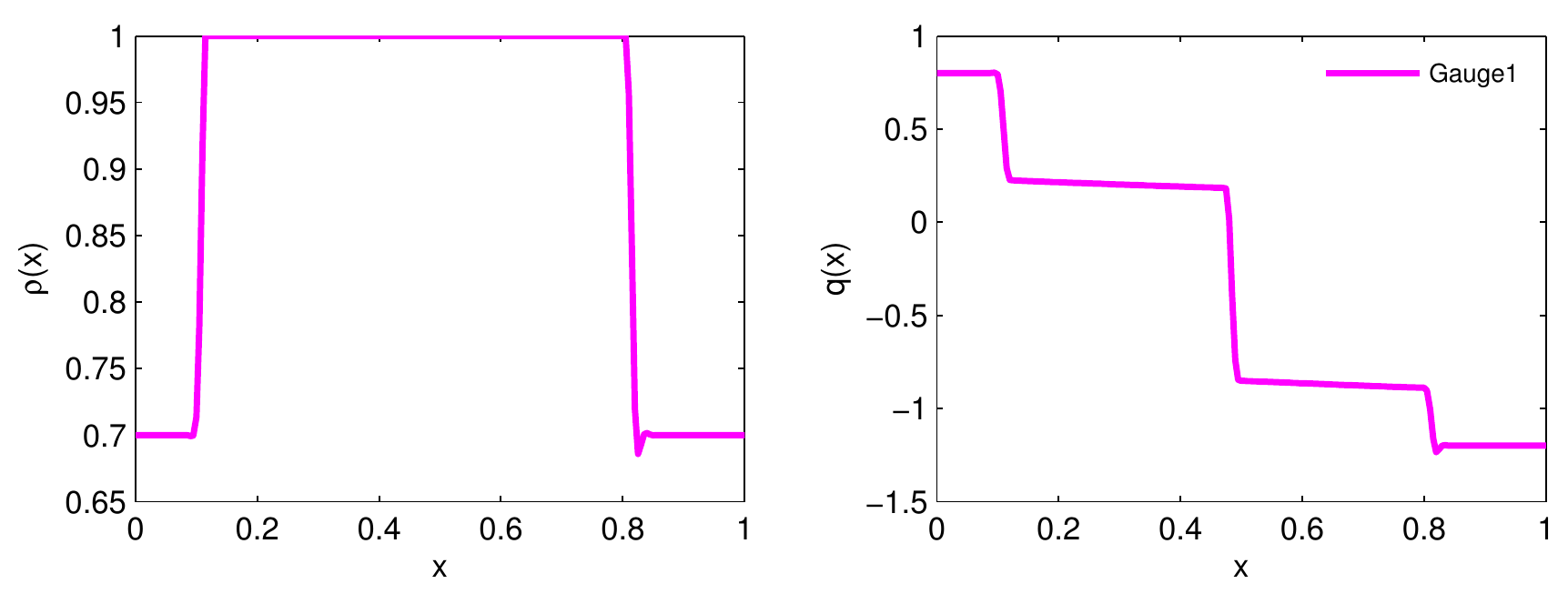}}\\
\subfigure[$t=0.0815$]{\label{fig:6:b} \includegraphics[scale=0.7]{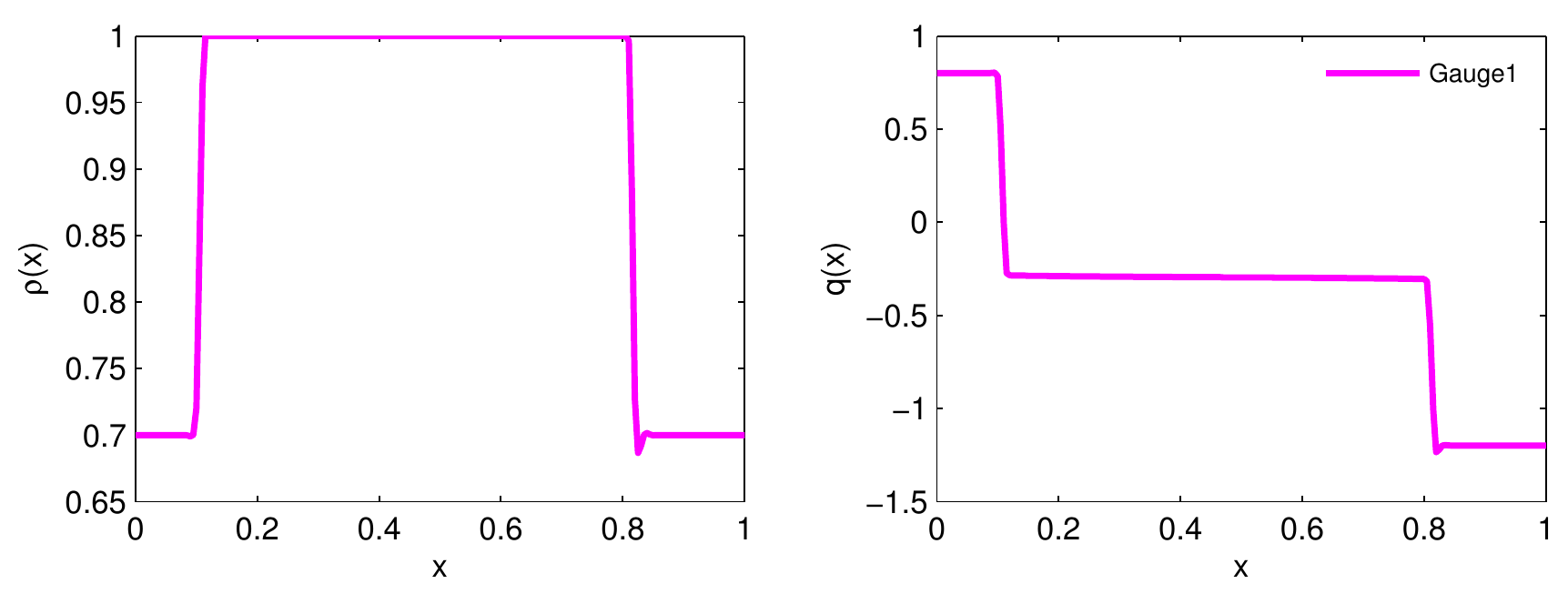}}
\caption{Gauge 1 method for problem (P3) with
  $\varepsilon=10^{-8}, \Delta x=5\times 10^{-3}, \Delta
  t=5\times 10^{-4}$. The left graphs are for $\rho$, the right ones
  for $q$, both as functions of $x$.}
\label{fig:6}
 \end{figure}
From Figures \ref{fig:5} and \ref{fig:6}, as $\varepsilon$ becomes
smaller, it takes shorter time to form a
new congestion region from the two colliding congestion regions: from
$48$ time steps to no more than one time step. It can be seen that as $\varepsilon\to 0$, the collision of these two
congested shocks aggregate instantaneously.

\textbf{Example 4. } 
The solution to the Riemann problem $(P4)$
consists of two shocks with intermediate state away from the
congestion density. So in the second Riemann problem there are two rarefaction
waves and a vacuum state appears as the intermediate state. As shown
in section \ref{sec:euler-with-cong}, these two rarefaction waves tend
to be contact waves. 

 As above, we choose $\varepsilon=10^{-4}, \Delta x=5\times 10^{-3}, \Delta
  t=5\times 10^{-4}$.
\begin{figure}
  \centering
  \includegraphics[scale=0.7]{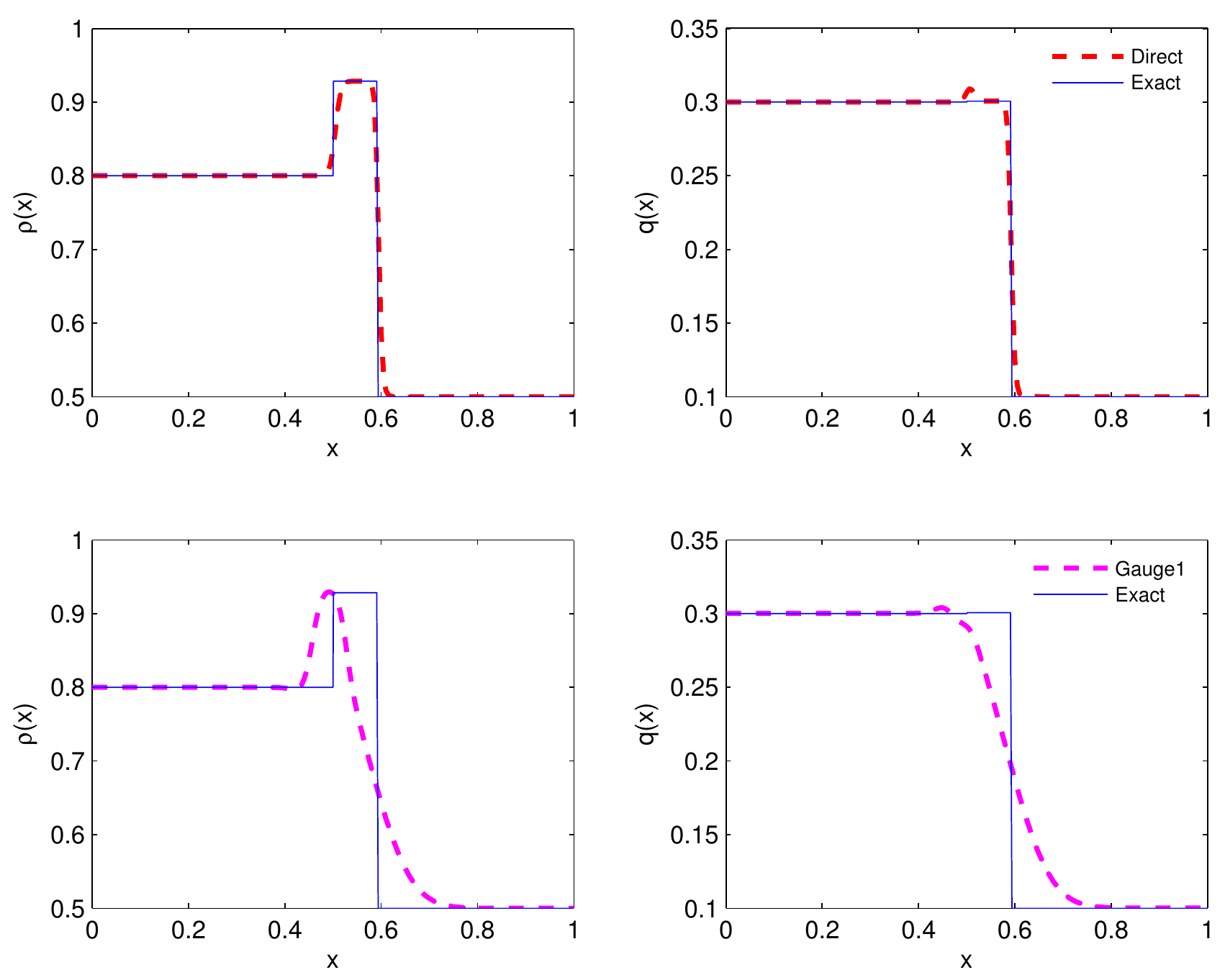}
\caption{The Direct and  Gauge 1
  methods for problem (P4) at $t=0.2$ with $\varepsilon=10^{-4}, \Delta x=5\times 10^{-3}, \Delta
  t=5\times 10^{-4}$. The solid
  lines are the exact solutions. The dashed curves are the
  numerical solutions. The left graphs are for $\rho$, the right ones
  for $q$, both as functions of $x$.}
\label{fig:7}
 \end{figure}
It can be seen from Figure \ref{fig:7} that the Direct method performs
well when the density is far away from congestion. However, the Gauge 1
method does not work well possibly due to the extra diffusion.

\begin{rem}
  All the features described above are preserved when we apply the
  fully implicit method \eqref{eq:17}\eqref{eq:18}(implicit in both
  $\rho$ and $\frac{q\otimes q }{\rho}$) to the Direct and Gauge 1
  method. There is a little improvement in the accuracy but no major one.
\end{rem}

Finally, we can compare the Gauge 1 and Gauge 2 methods.  With the
stencil of \eqref{eq:11} in the same setting as that of 
example 2,  the Gauge 2 method yields almost the same
numerical result as the Direct method (see Fig. \ref{fig:8}).
\begin{figure}
  \centering
  \includegraphics[scale=0.7]{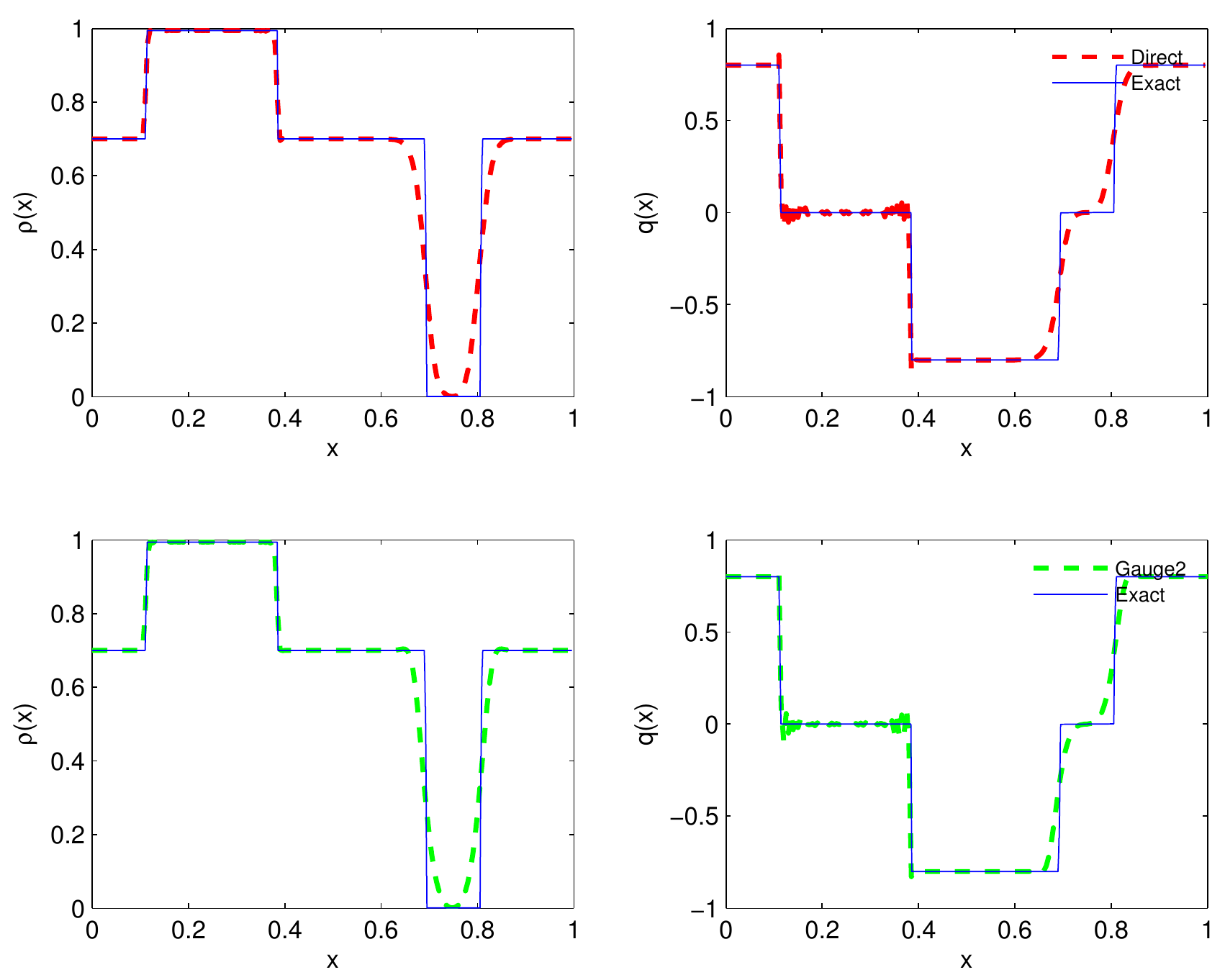}
\caption{The Direct  and  Gauge 2
  methods with discretization \eqref{eq:21} for problem (P2) at $t=0.05$. The solid
  lines are the exact solutions. The dashed curves are the
  numerical solutions. The left graphs are for $\rho$, the right ones
  for $q$, both as functions of $x$. }
\label{fig:8}
 \end{figure}

 \subsection{A two dimensional test case}\label{sec:2d-case-num}

In this section, we will test the Direct and Gauge 1 method in the
2D case. Since there is no theoretical solution in 2D case, only some
phenomena will be presented.

The test example is chosen to illustrate the collision of two
congested regions. It is basically a two dimensional extension of
the test case (P3) with some lateral shift. The initial data of the density  and velocity is:
\begin{gather}
\rho=0.8\times\mathrm{1}_{A\cup B}+0.6\times\mathrm{1}_{[0,1]\times[0,1]\backslash (A\cup B)},\\
  \boldsymbol{q}(x,y,0)=
  \begin{pmatrix}
    1\\ 0
  \end{pmatrix}
\mathrm{1}_A+\begin{pmatrix}
    -1\\ 0
  \end{pmatrix}\mathrm{1}_B,\\
A=\left[\frac{1}{6},\frac{5}{12}\right]\times\left[\frac{1}{3},\frac{7}{12}\right],
\quad B=\left[\frac{7}{12},\frac{5}{6}\right]\times\left[\frac{5}{12},\frac{2}{3}\right].
\end{gather}
The vector field $q$ is plotted in the Figure \ref{fig:17}.
\begin{rem}
  In this test example, the background density is $\rho=0.6$. It is
  even more interesting  to see what happens when the background
  density is close to zero. However, our schemes are not performing
  well in this situation. Indeed, vacuum is a big problem which needs
  some special treatment. And  when the background
  density is close to zero, the system is almost pressureless, which
  is also a difficult problem. These problems may be considered in the
  future work.
\end{rem}

\begin{figure}
  \centering
\subfigure[density $\rho$
]{\label{fig:17:a} \includegraphics[scale=0.5]{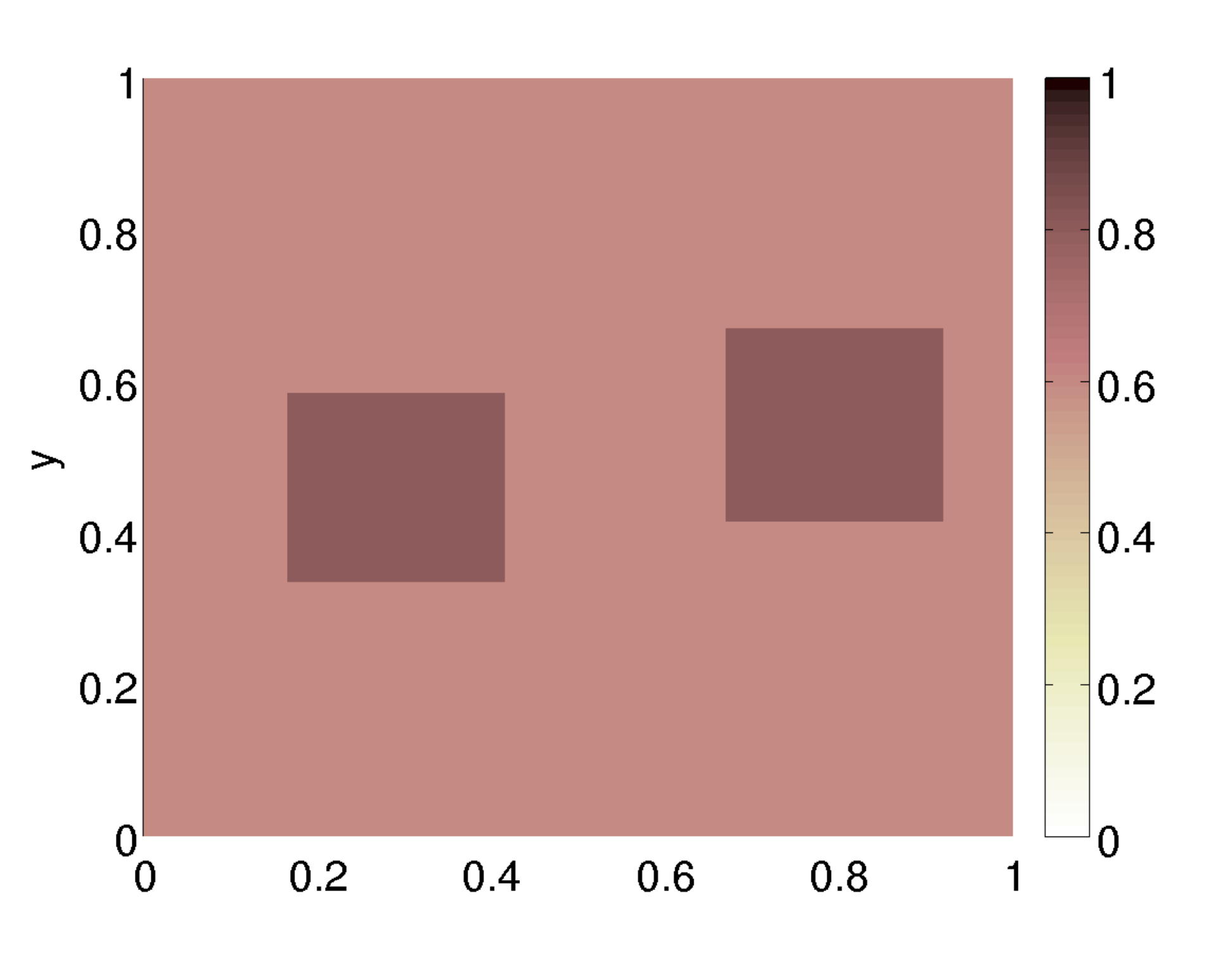}}
\subfigure[ vector field of momentum $\boldsymbol{q}$]{\label{fig:17:b} \includegraphics[scale=0.5]{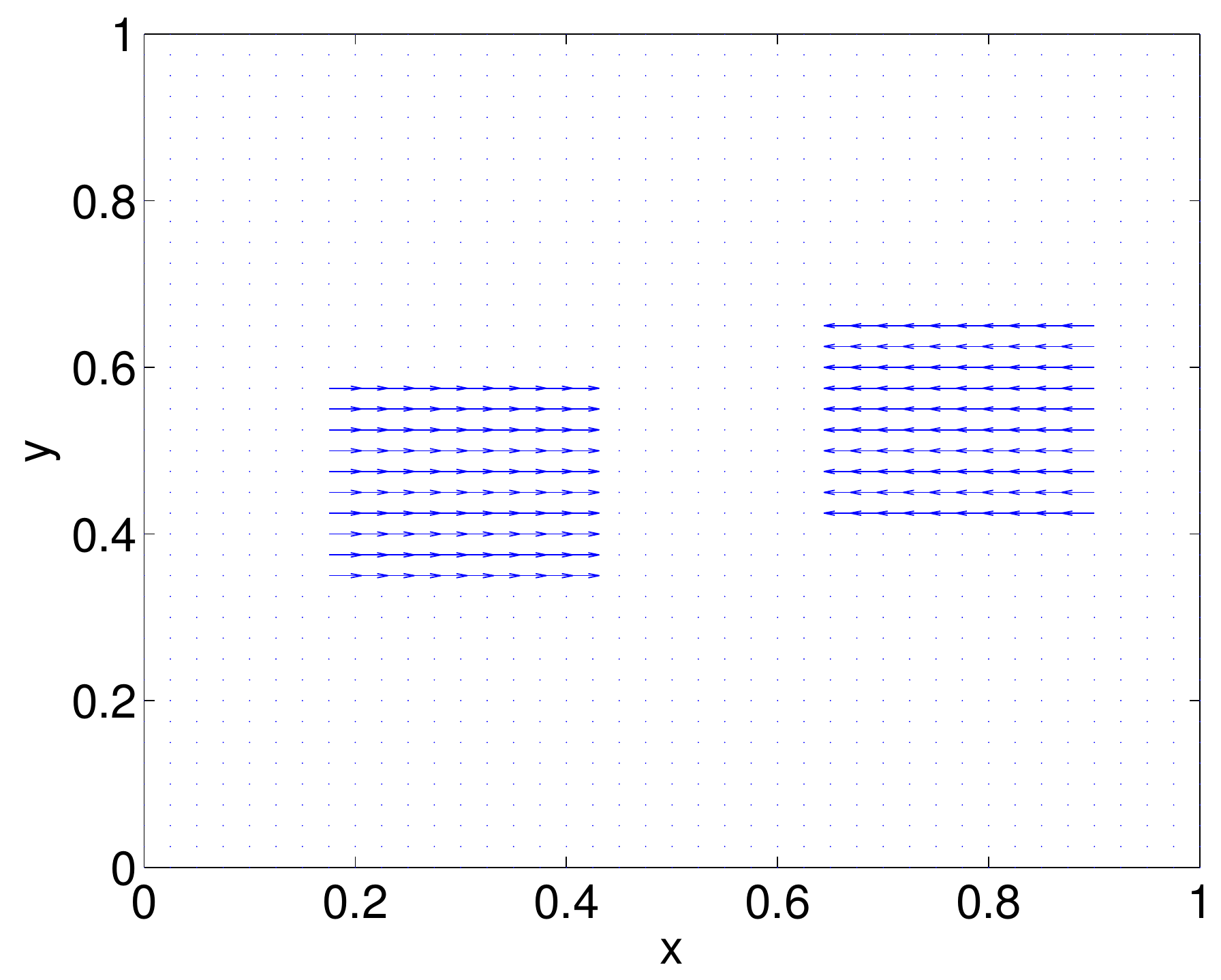}}
    \caption{The initial data of the density and  momentum
    $\boldsymbol{q}$,  both as functions of $x$ and $y$.}
  \label{fig:17}
\end{figure}
 It can be expected that there will be two congested regions forming
 and 
 moving towards each other with two shocks in the front and  two rarefaction
 waves left behind. These two shocks compress the fluid and cause
 congestion. This is reflected in Figure \ref{fig:11}. 
\begin{figure}
  \centering
\subfigure[density $\rho$ ]{\label{fig:11:a} \includegraphics[scale=0.5]{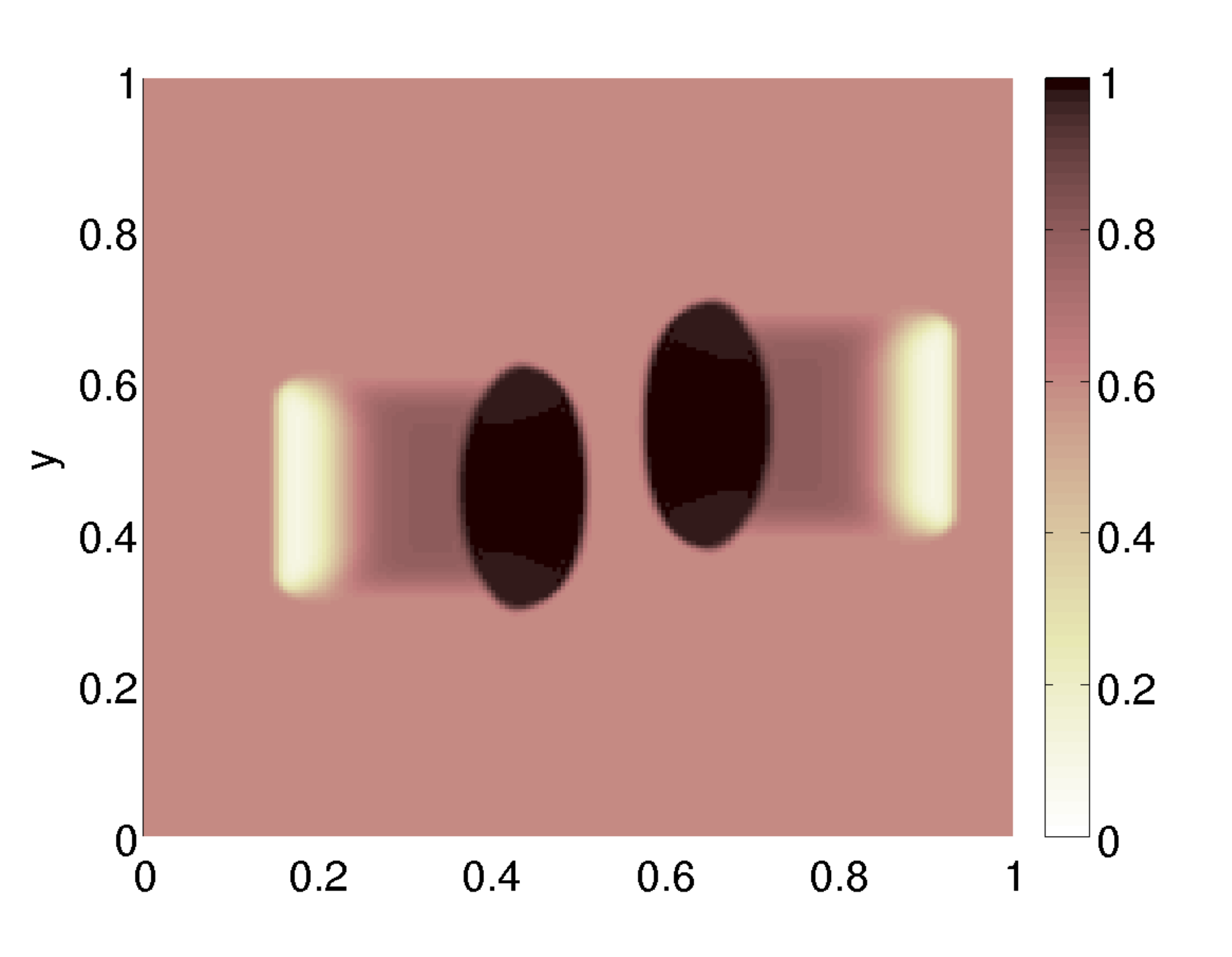}}
\subfigure[ vector field of momentum $\boldsymbol{q}$]{\label{fig:11:b} \includegraphics[scale=0.5]{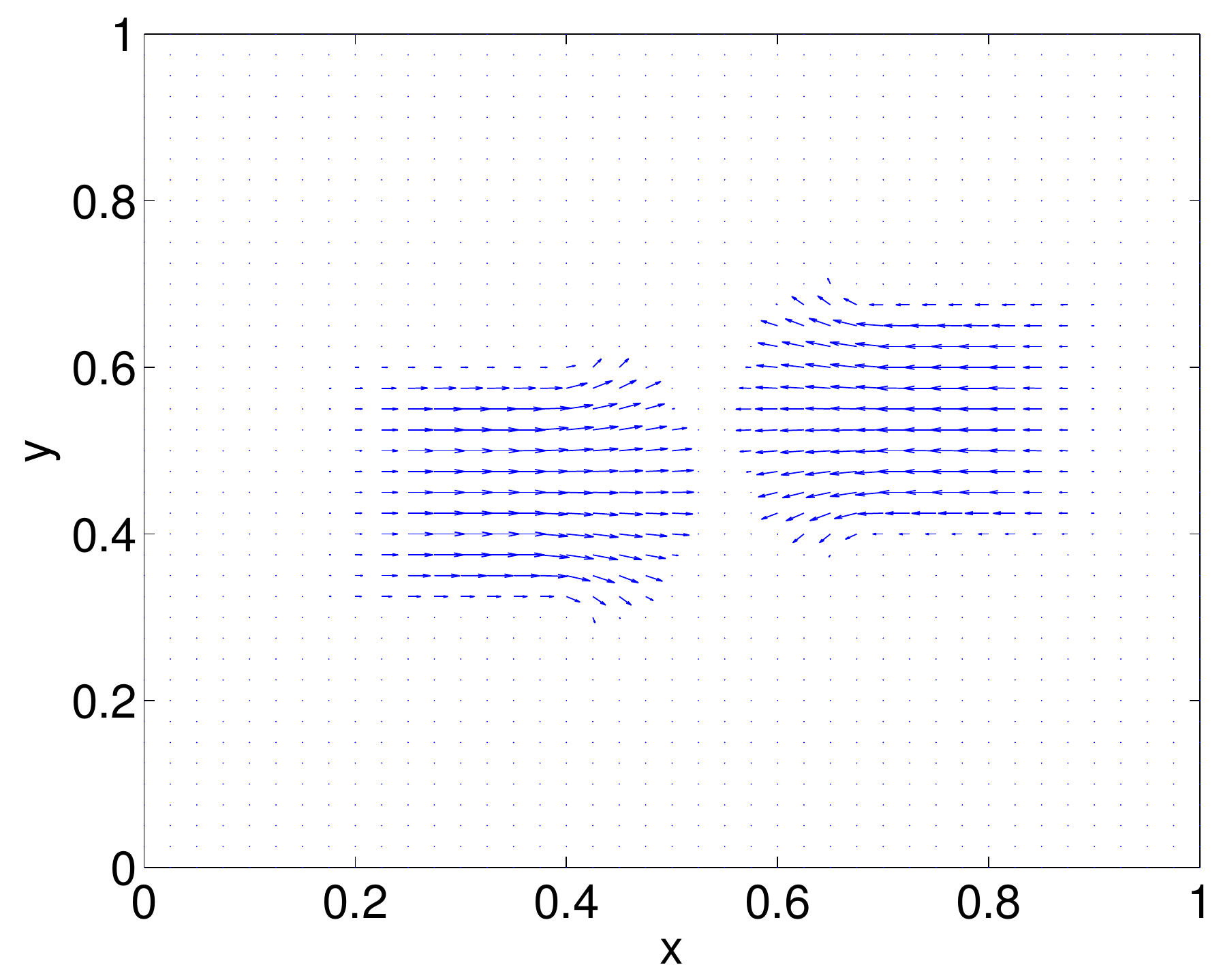}}
\caption{The  Direct method  with
  $\varepsilon=10^{-4}, \Delta x=5\times 10^{-3}, \Delta
  t=5\times 10^{-4}$ at $t=0.05$. The left graph is for $\rho$, the right one
  for vector field $q$, both as functions of $x$ and $y$.}  \label{fig:11}
\end{figure}

These two shocks with opposite directions meet at a later time. The
interaction of these shocks forms a bigger
congestion region with higher density
, which can be illustrated in Figure 
\ref{fig:12}. The interaction
 of the two shocks induces the formation of two new shocks moving in
 the orthogonal direction compared to the initial motion.
\begin{figure}
  \centering
\subfigure[density $\rho$]{\label{fig:12:a} \includegraphics[scale=0.5]{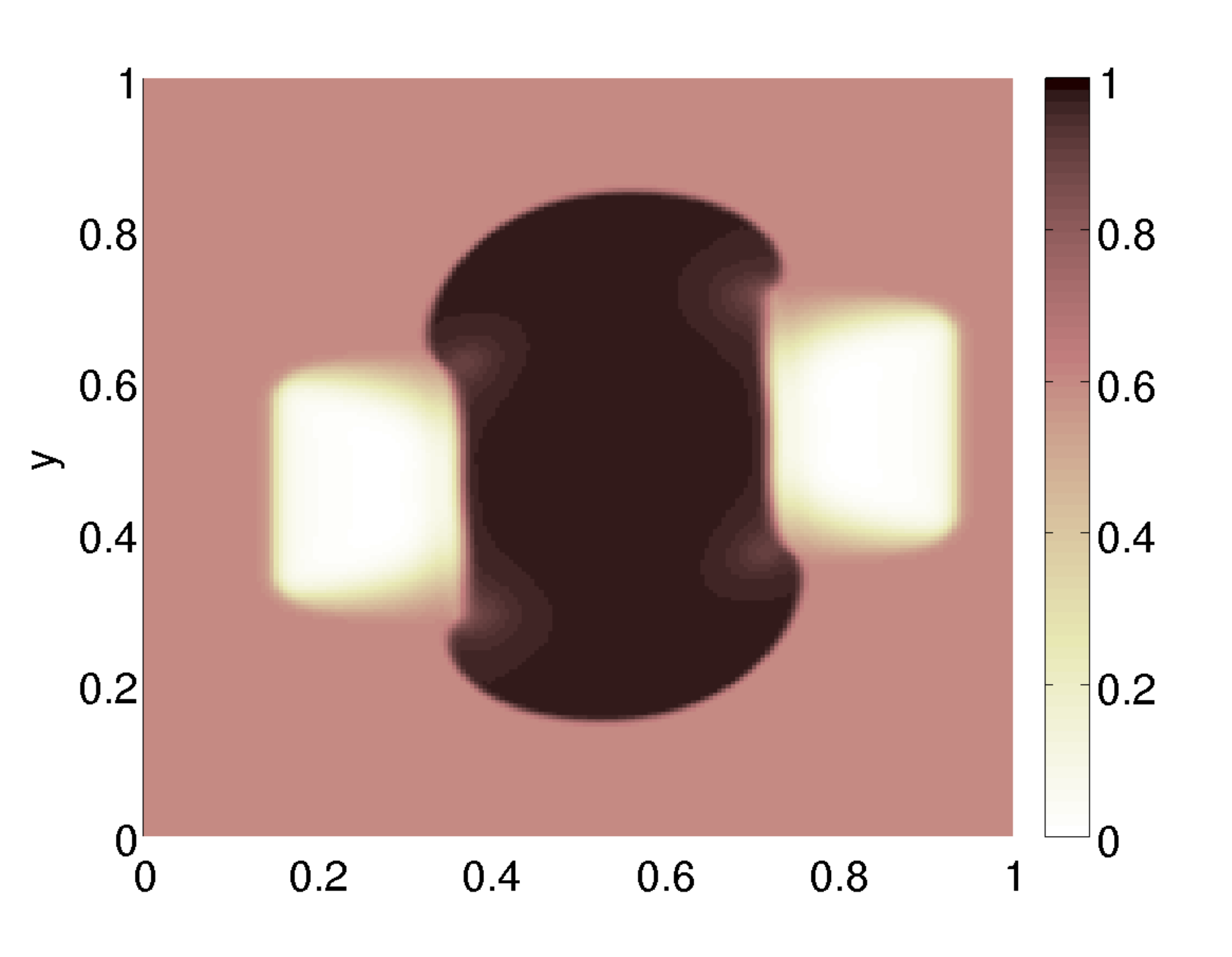}}
\subfigure[ vector field  of momentum $\boldsymbol{q}$]{\label{fig:12:b} \includegraphics[scale=0.5]{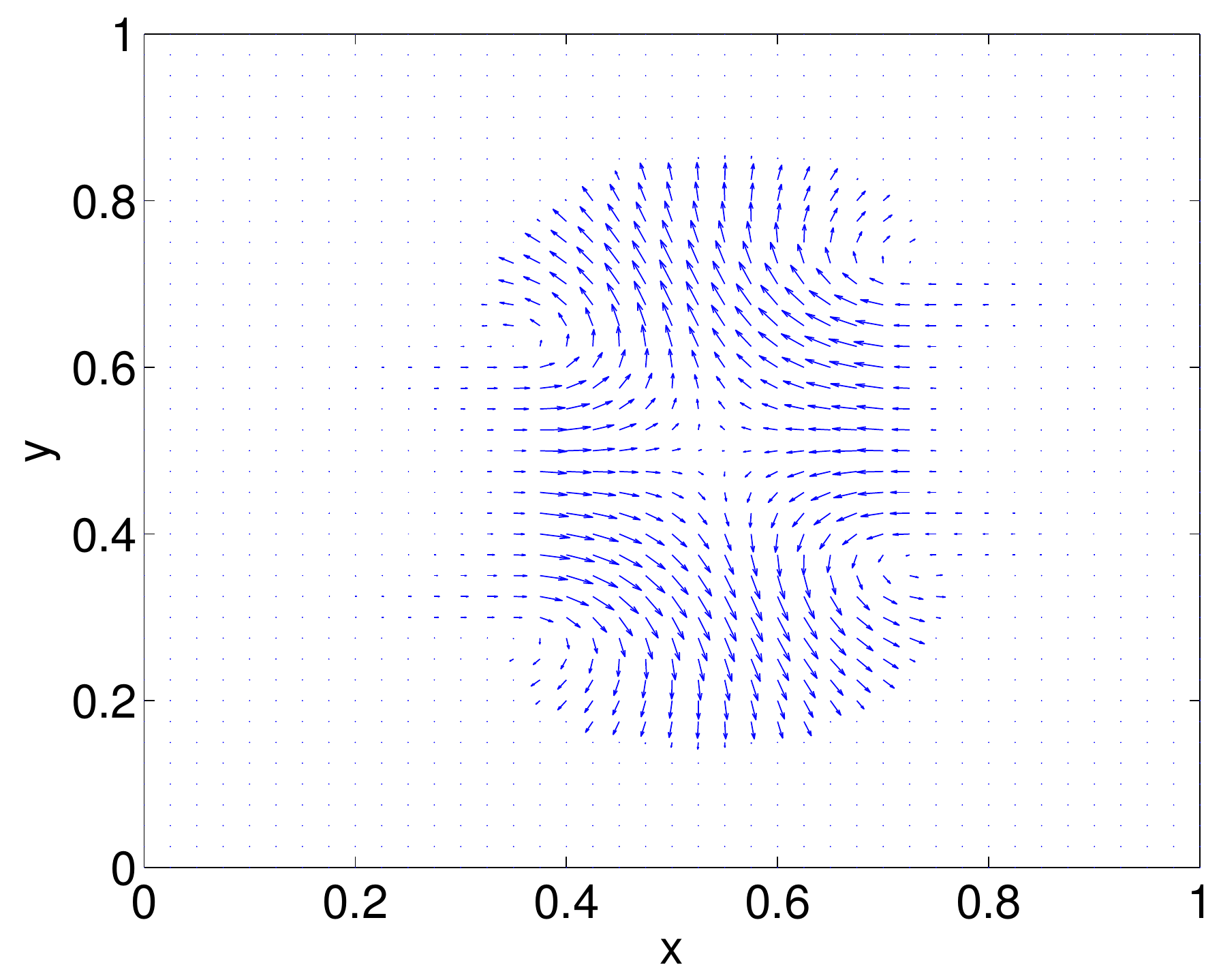}}
 \caption{The Direct method  with
  $\varepsilon=10^{-4}, \Delta x=5\times 10^{-3}, \Delta
  t=5\times 10^{-4}$ at $t=0.2$. The left graph is for $\rho$, the right one
  for vector field $q$, both as functions of $x$ and $y$.}  \label{fig:12}
\end{figure}

The similar result can be obtained by the Gauge 1 method as in Figure
\ref{fig:13} and \ref{fig:14}.
\begin{figure}
  \centering
\subfigure[density $\rho$]{\label{fig:13:a} \includegraphics[scale=0.5]{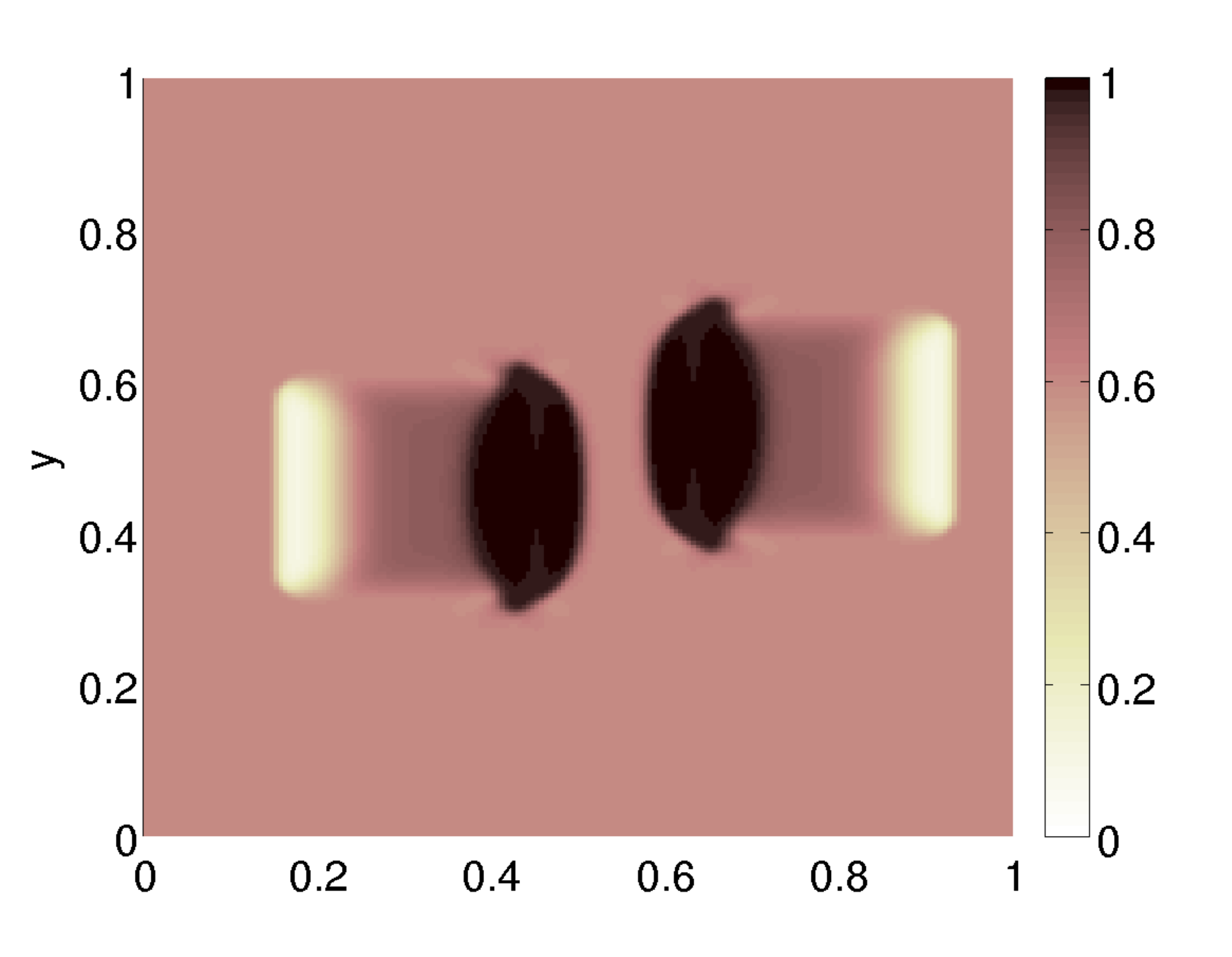}}
\subfigure[vector field  of momentum $\boldsymbol{q}$]{\label{fig:13:b} \includegraphics[scale=0.5]{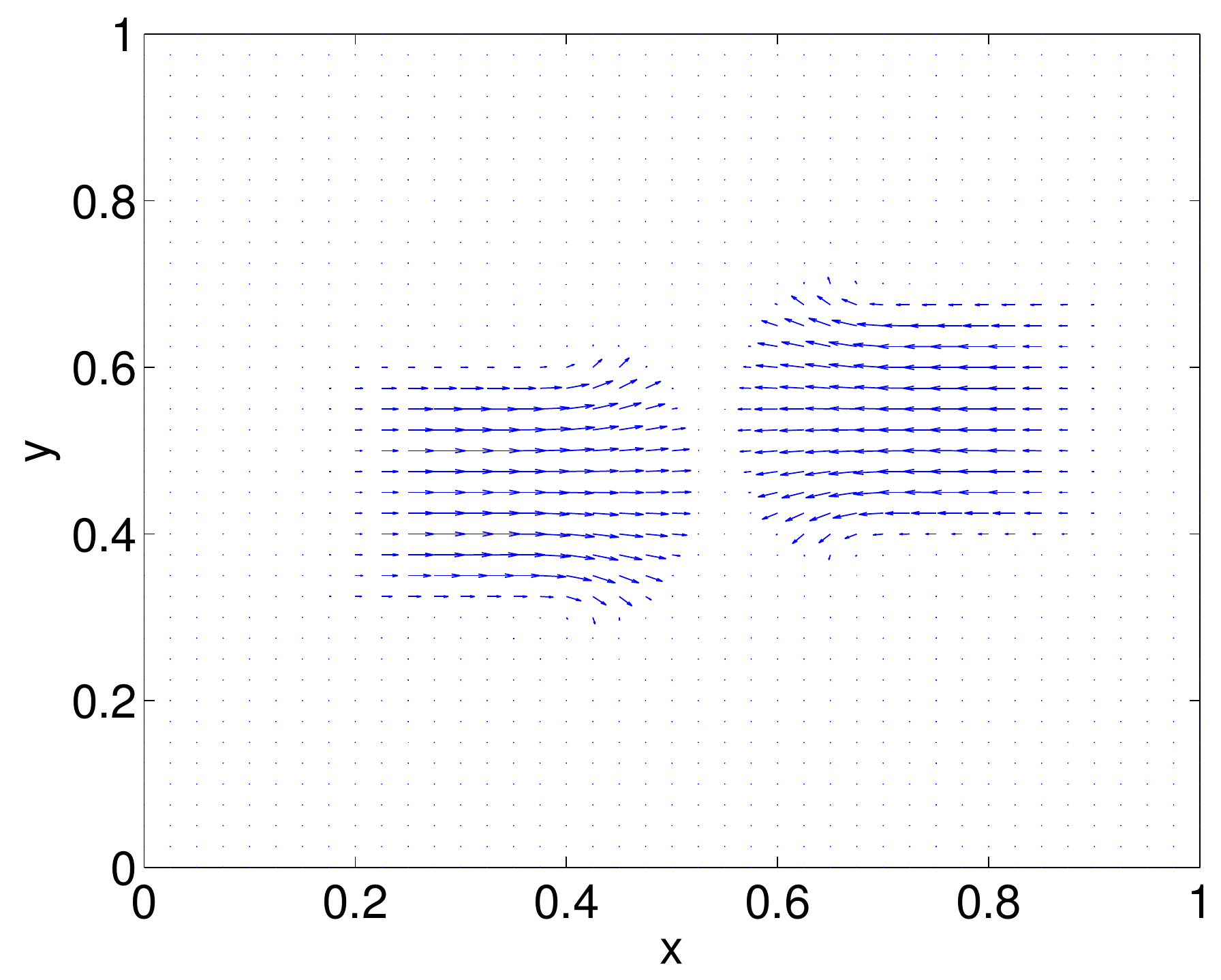}}
\caption{The   Gauge 1 method  with
  $\varepsilon=10^{-4}, \Delta x=5\times 10^{-3}, \Delta
  t=5\times 10^{-4}$ at $t=0.05$. The left graph is for $\rho$, the right one
  for vector field $q$, both as functions of $x$ and $y$.}  \label{fig:13}
\end{figure}
\begin{figure}
  \centering
\subfigure[density $\rho$]{\label{fig:14:a} \includegraphics[scale=0.5]{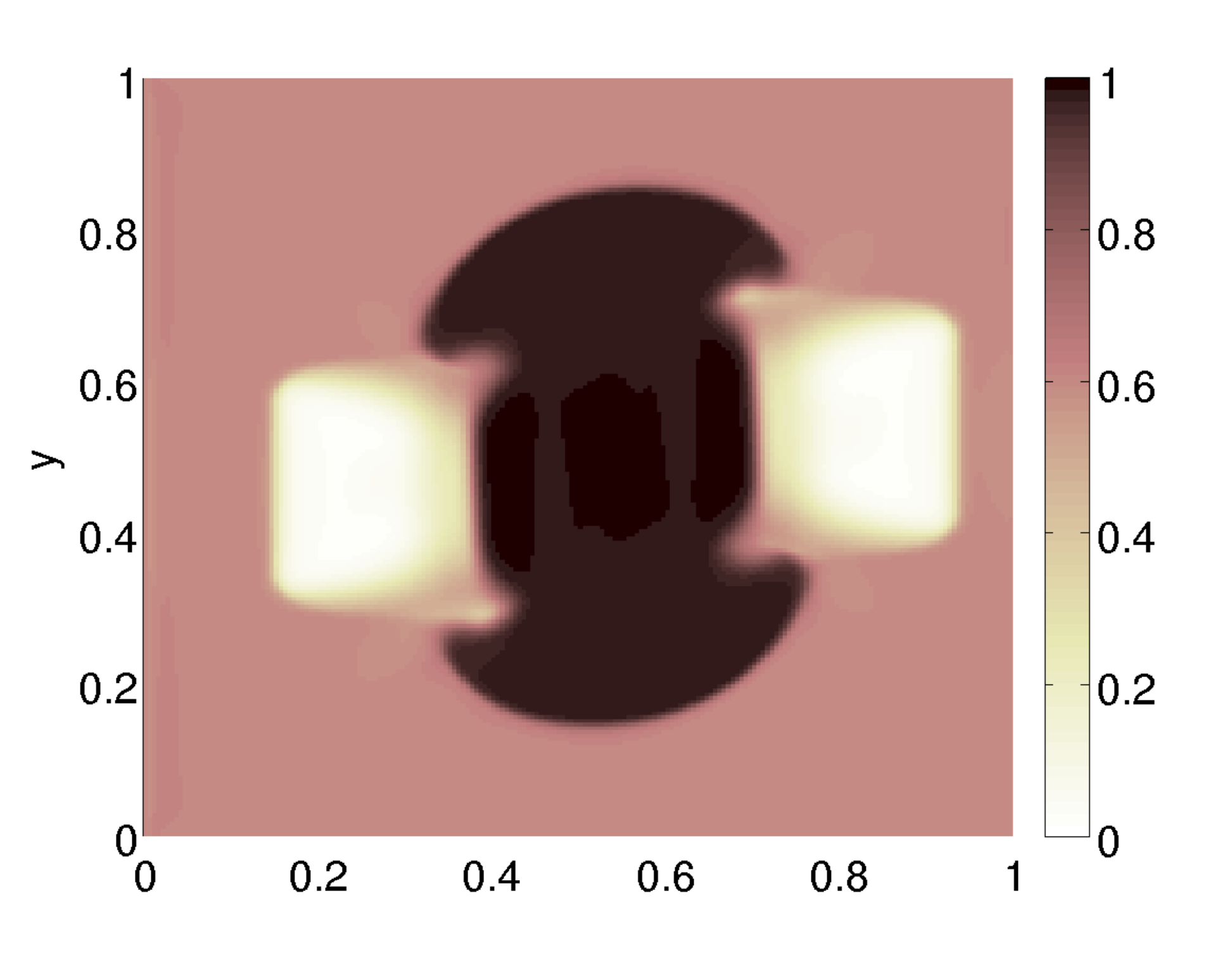}}
\subfigure[ vector field  of momentum $\boldsymbol{q}$]{\label{fig:14:b} \includegraphics[scale=0.5]{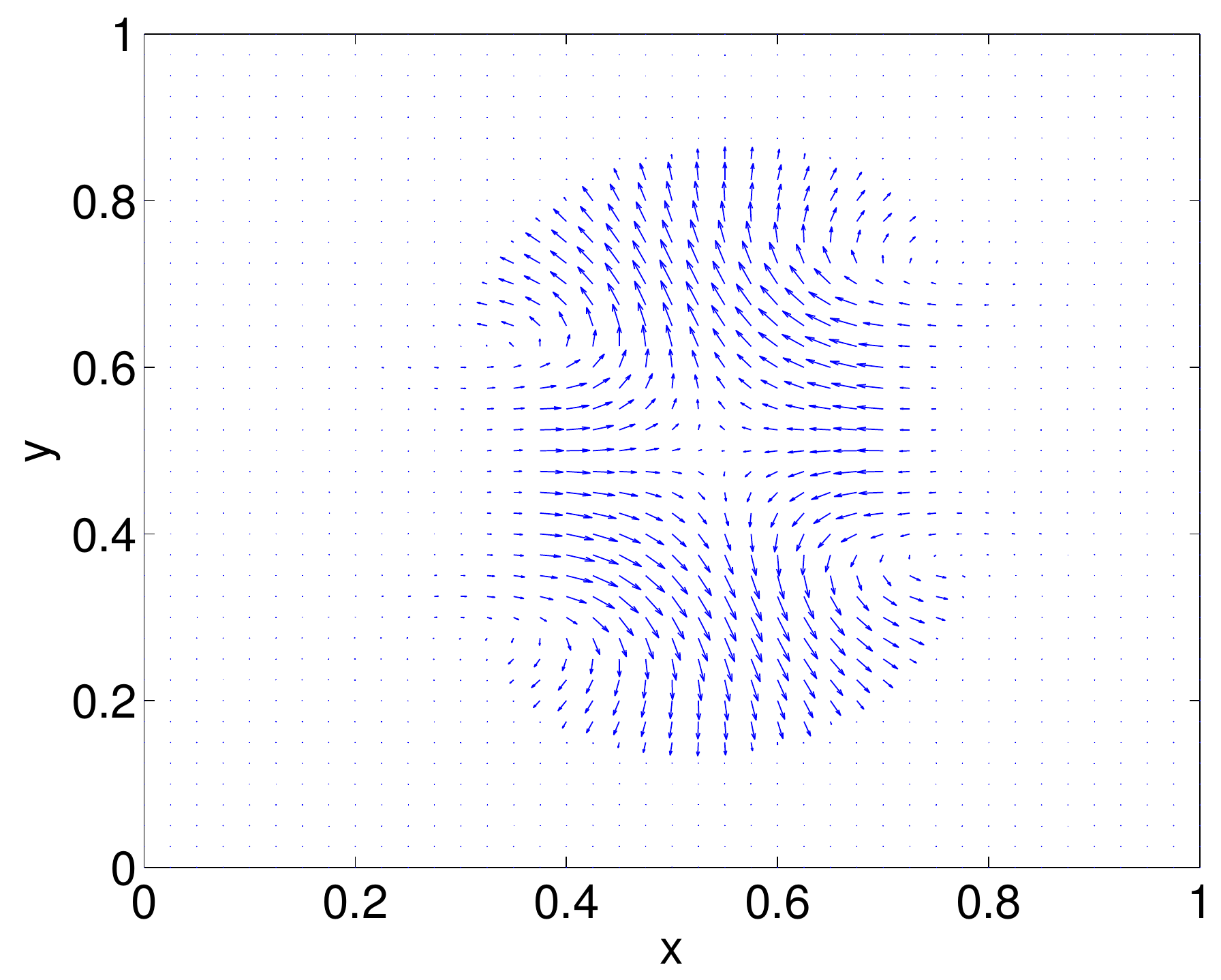}}
 \caption{The   Gauge 1 method  with
  $\varepsilon=10^{-4}, \Delta x=5\times 10^{-3}, \Delta
  t=5\times 10^{-4}$ at $t=0.2$. The left graph is for $\rho$, the right one
  for vector field $q$, both as functions of $x$ and $y$.}  \label{fig:14}
\end{figure}

The difference between the two methods in the two dimensional case is
not striking. In order to illustrate them, we look at a cut  at
$y=0.5$. The cuts are  displayed on Figure \ref{fig:15} and \ref{fig:16}.

\begin{figure}[bp]
  \centering
\subfigure[Direct method]{\label{fig:15:a} \includegraphics[scale=0.4]{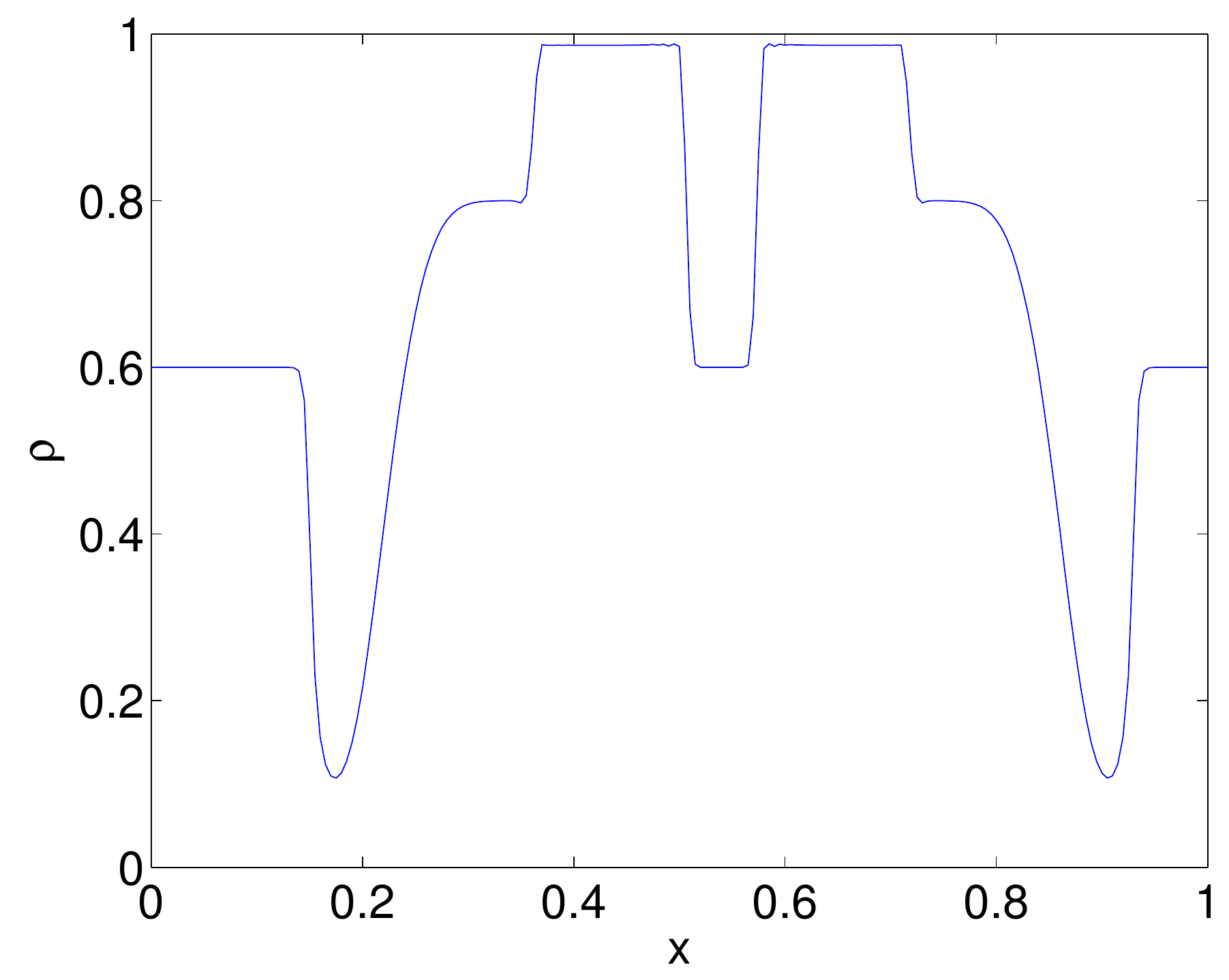}}
\subfigure[Gauge 1 method]{\label{fig:15:b} \includegraphics[scale=0.4]{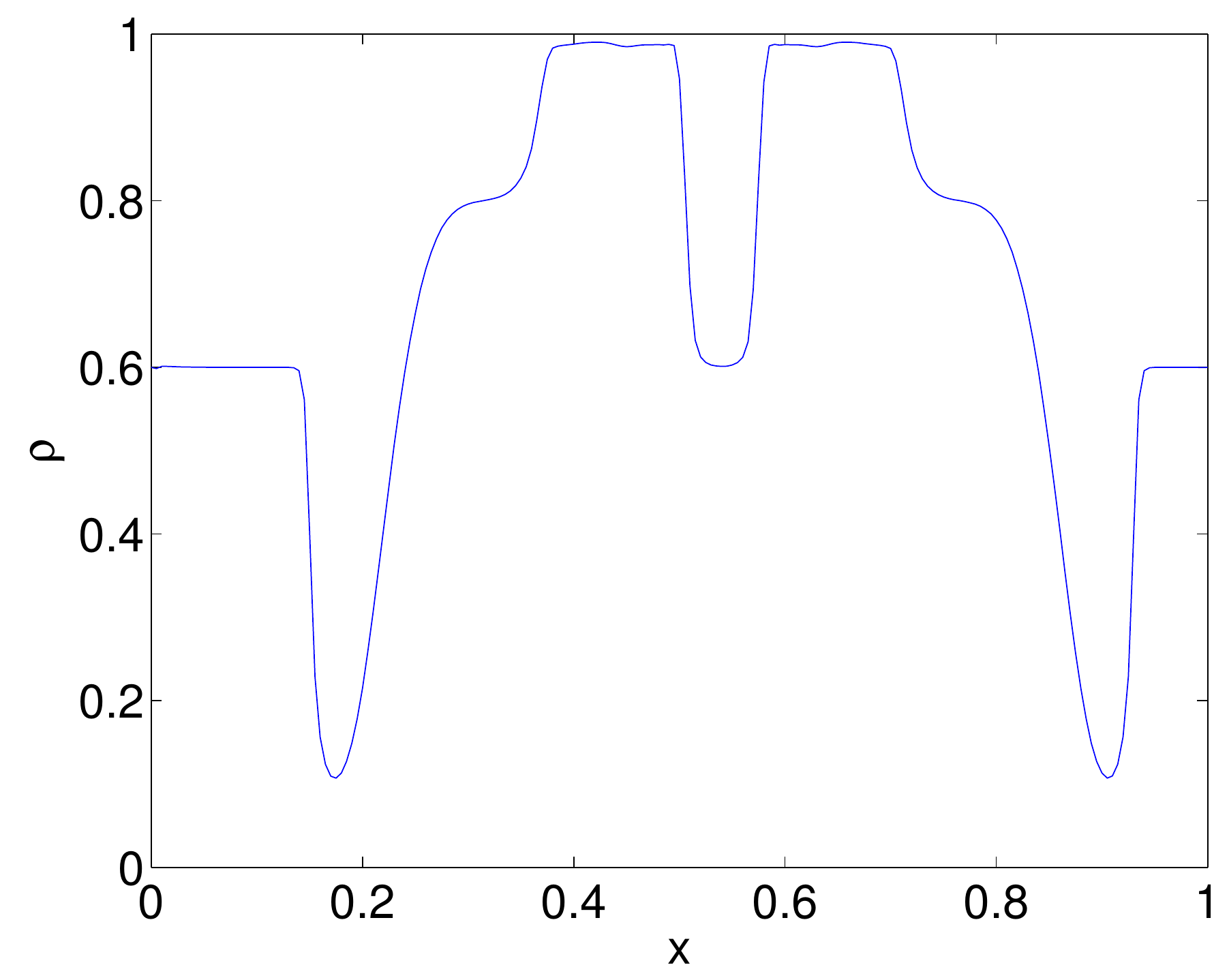}}
\caption{Direct and  Gauge 1 methods  with
  $\varepsilon=10^{-4}, \Delta x=5\times 10^{-3}, \Delta
  t=5\times 10^{-4}$ at $t=0.05$. Cut of
  the density along the line $y=0.5$, as a function of $x$. Left hand:
Direct method; right hand: Gauge 1 method.}  \label{fig:15}
\end{figure}
\begin{figure}[bp]
  \centering
\subfigure[Direct method]{\label{fig:16:a} \includegraphics[scale=0.4]{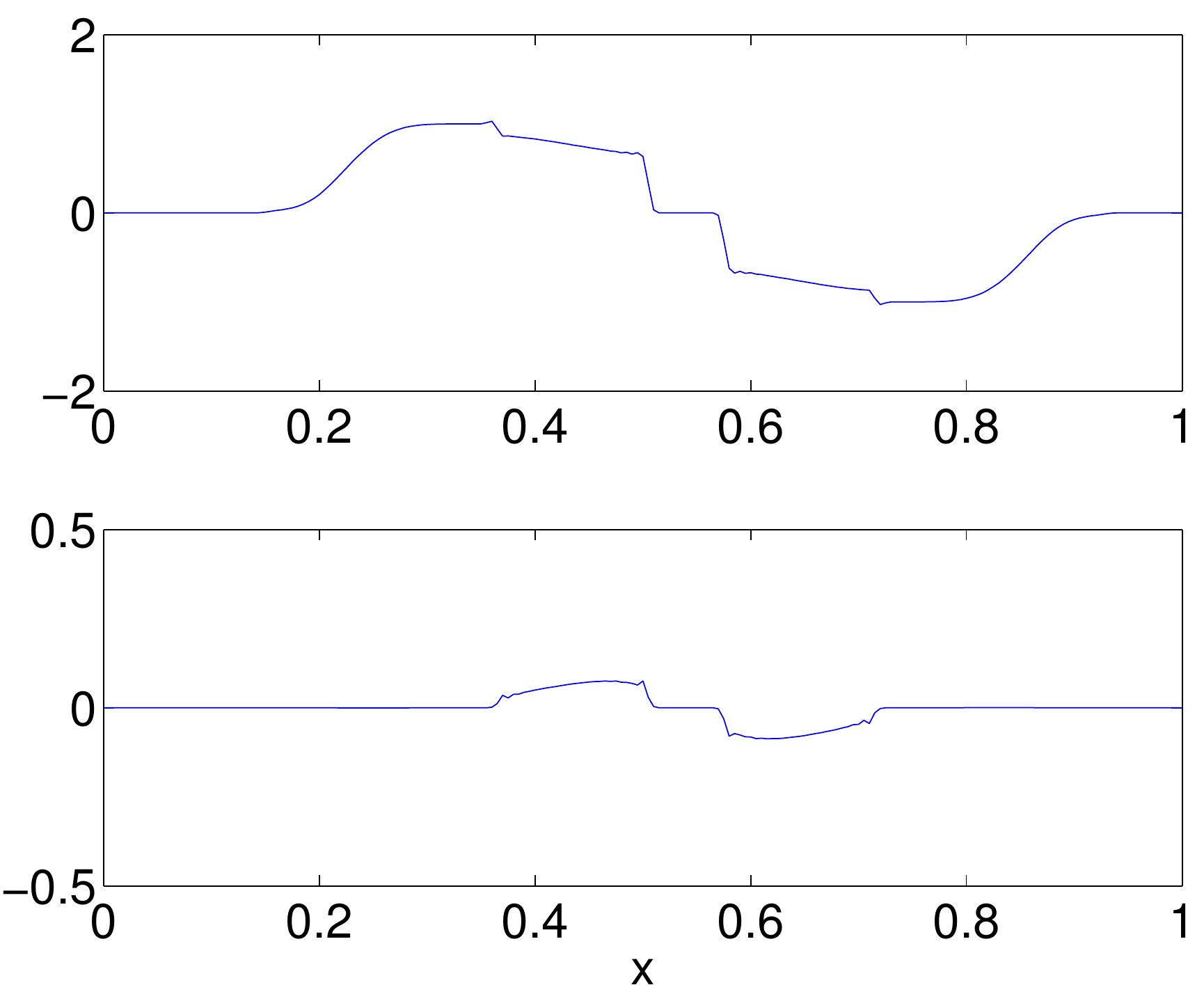}}
\subfigure[Gauge 1 method]{\label{fig:16:b} \includegraphics[scale=0.4]{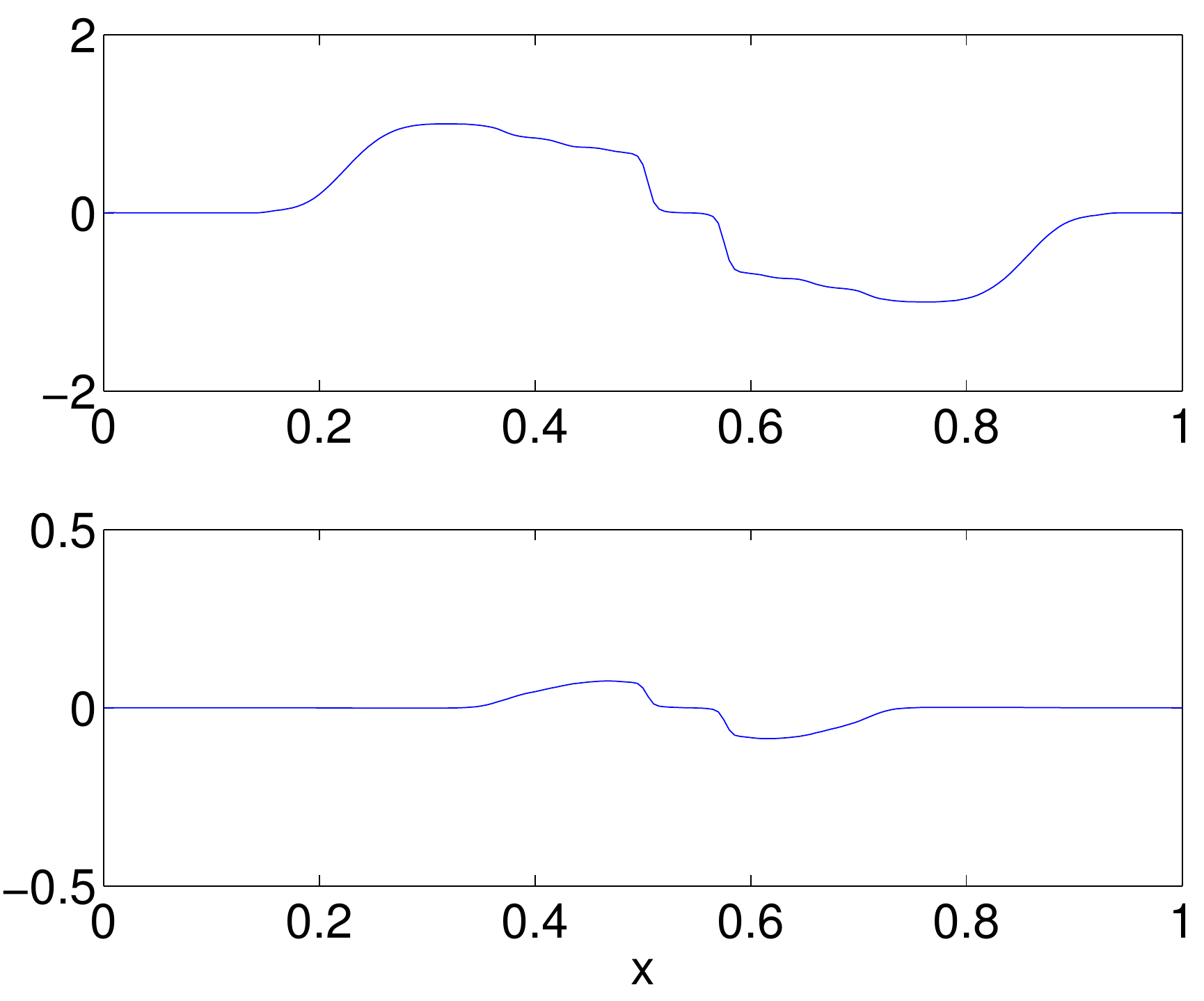}}
\caption{Momentum  of the Direct and  Gauge 1 methods  with
  $\varepsilon=10^{-4}, \Delta x=5\times 10^{-3}, \Delta
  t=5\times 10^{-4}$ at $t=0.05$. Cut of
  the momentum along the line $y=0.5$, as a function of $x$. Left hand:
Direct method; right hand: Gauge 1 method.}  \label{fig:16}
\end{figure}
Similar observations as in the one dimensional case can be made from Figure
\ref{fig:15} and \ref{fig:16}. The Gauge 1 method provides a little  less oscillations
in the congestion region but  also brings more diffusion.
\clearpage

\section{Conclusion}

In this paper, we have studied the Euler system with a maximal density
constraint. A small parameter $\eps$ was introduced to measure the
stiffness of the constraint. As $\eps \rightarrow 0$, the model gives
rise to a two-phase model with congested regions (with maximal
density) and uncongested regions (with density below the maximal
density). 

One-dimensional solutions of this asymptotic problem have
been investigated to provide the information of  the interface
conditions. However, it can not characterize the whole
dynamics and it is hard to extend to higher dimensional cases. Therefore, we have devised asymptotic preserving numerical
schemes, which are valid for all range  of $\eps$ and thus are
capable of  capturing the asymptotic dynamics. Two numerical schemes have been
considered and compared on both one-dimensional and two dimensional
test-cases. They both capture the congested regions well. However, the first method shows some oscillations near
the interface between congested and uncongested regions, while the
second has much less oscillations but is more diffusive. A careful
error analysis in different norms with respect to different parameters (time and space steps,
$\varepsilon$) are conducted for a one dimensional test
case. 

Following \cite{2010_Tang}, the development of
second order schemes will be performed in future works. More robust schemes for
capturing the divergence free constraint in the congested regions could
also be investigated. Coupling our methodology with schemes dedicated to the pressureless gas dynamics could also improve the results in the low density regions. Finally, simulations of the non-conservative model with a
 supplementary geometric constraint on the speed of the flow
 \cite{2010_Congestion} will be an
 interesting problem. 
\\

\noindent\textbf{Acknowledgement} This work has been supported by the Marie Curie Actions of the European 
Commission in the frame of the DEASE project (MEST-CT-2005-021122), by the 
Agence Nationale de la Recherche (ANR) under contracts PANURGE 
(ANR-07-BLAN-0208-03) and PEDIGREE (ANR-08-SYSC-015-01). 

\renewcommand{\thesection}{\Alph{section}}
\setcounter{section}{0} 

\section[Solutions of the one-dimensional problem]{Appendix : Solutions of the one-dimensional problem: the Riemann problem and the cluster collisions}
 \label{sec:one-dimens-riem}
 The one-dimensional version of system (\ref{Eq:Euler_rho_eps})-(\ref{Eq:Euler_q_eps}) can be written as follows:
\begin{eqnarray}
&&\partial_{t}\rho + \partial_{x}q = 0,\label{Eq:rho_1D_eps}\\
&&\partial_{t}q + \partial_{x}\left(\frac{q^2}{\rho}\right) + \partial_{x} (\eps p(\rho)) = 0,\label{Eq:q_1D_eps}
\end{eqnarray}
where $q(x,t)$ is here a scalar function and $x$ is the position in $\R$. In this section, we first investigate  the Riemann problem and the limits of its solutions as $\eps \rightarrow 0$ and then we briefly recall the one-dimensional solutions of the limit system (\ref{Eq:mass})-(\ref{Eq:momentum})-(\ref{Eq:constraint_CongNum}) provided in \cite{2PhaseFlow_Bouchut_al}, which consists of the collision of two finite clusters (domain where $\rho = \rho^{\ast}$).  

\subsection{The one-dimensional Riemann problem} 
The Riemann problem is an initial value problem, where the initial condition is a piece-wise constant function with a discontinuity between two constant states: 
\begin{equation*}
(\rho_{0},q_{0})(x) = \left\{\begin{array}{cl} (\rho_{\ell},q_{\ell}), \text{ for } x < 0,\\
(\rho_{r},q_{r}), \text{ for } x \geq 0.
\end{array}\right.
\end{equation*} 
The solutions of this problem for the Euler system (\ref{Eq:rho_1D_eps})-(\ref{Eq:q_1D_eps}) are well known. So, the strategy is to take the limits of these solutions as $\eps$ goes to zero. This provides the solutions of the Riemann problem for the singular asymptotic limit (\ref{Eq:mass})-(\ref{Eq:momentum})-(\ref{Eq:constraint_CongNum}). Similar studies was carried out in \cite{2003_ChenLiu_FormationDelta} for the isentropic Euler equation, in \cite{2008_Traffic_DegondRascle} for a traffic jam model and in \cite{2010_Congestion} for a herding problem.

\subsubsection{Shock and rarefaction waves}
\label{Appendix:rarefaction_shock}

The material of this section is classical, and is given here only for the reader's convenience. 

 The two characteristic speeds $\lambda_{\pm}^{\eps}$ and the two characteristic fields $\br_{\pm}^{\eps}$ of the one-dimensional system (\ref{Eq:rho_1D_eps})-(\ref{Eq:q_1D_eps}) are:
\begin{equation*}
\lambda_{\pm}^{\eps} = u \pm \sqrt{\eps p'(\rho)},\quad \br_{\pm}^{\eps} = \left(\begin{array}{c} 1\\ u \pm \sqrt{\eps p'(\rho)}\end{array}\right).
\end{equation*}
It can be easily checked that both characteristic fields are genuinely non linear for positive densities. Therefore, the solutions of the Riemann problem are made of constant states separated by rarefaction or shock waves \cite{leveque_book_2002}. 

The \textbf{rarefaction waves} are continuous self similar solutions: $(\rho(x/t),q(x/t))$. Given a state $(\hat\rho,\hat q)$, the states which can be connected to $(\hat\rho,\hat q)$ by a rarefaction wave are those located on the integral curves $i_{\pm}^{\eps}$ of the right eigenvectors of the Jacobian matrix of the flux function issued from $(\hat\rho,\hat q)$. They are given by:
\begin{equation*}
\rho'(s) = 1,\quad i_{\pm}^{\eps\ '}(s) = \hat u \pm  \sqrt{\eps p'(\rho)},\quad \rho(0) = \hat\rho,\quad  i_{\pm}^{\eps}(0) = \hat q,
\end{equation*}
which is equivalent to $i_{\pm}^{\eps\ '}(\rho) =  \hat u \pm  \sqrt{\eps p'(\rho)},\ i_{\pm}^{\eps}(\hat\rho) = \hat q$ and then to equation: 
\begin{equation}
i_{\pm}^{\eps}(\rho) = \rho\hat u \pm \rho\sqrt{\eps}(P(\rho) - P(\hat\rho)),\label{Eq:integral_curve}
\end{equation} 
where $P$ is an antiderivative of $\sqrt{p'(u)}/u$ and $\hat u = \hat q/\hat\rho$. The graph of $i_{-}^{\eps}$ (resp. $i_{+}^{\eps}$) is called the 1-integral curve (resp. the 2-integral curve). The following proposition provides several features of the integral curves.
\begin{Prop} \textbf{(Integral and rarefaction curves)}\label{Rarefaction_wave}
\begin{enumerate} 
\item The 1-integral curve $i_{-}^{\eps}$ (resp. the 2-integral curve $i_{+}^{\eps}$) is concave (resp. convex) as functions of $\rho$ and $i_{-}^{\eps}(0) = i_{+}^{\eps}(0) = 0$.
\item The limit of the integral curves as $\eps$ goes to zero is the union of the straight lines $\left\{ q = \rho \hat u\right\}$ and $\left\{\rho = \rho^{\ast}\right\}$. 
\item Suppose that the state $\hat\rho^{\eps}$ is such that $\hat\rho^{\eps} \rightarrow \rho^{\ast}$ and $\eps p(\hat\rho^{\eps}) \rightarrow \bar{p}$. For all $\rho < \hat\rho^{\eps}$, we have:
\begin{equation*}
\quad \left|\frac{i_{\pm}^{\eps}(\rho)}{\rho} - \hat u\right| \leq \sqrt{\eps}\int_{0}^{\hat\rho^{\eps}}\frac{\sqrt{p'(u)}}{u} du = \underset{\hat\rho^{\eps}\rightarrow\rho^{\ast}}{O}(\eps^{\frac{1}{2\gamma}}) 
\end{equation*}
\item If $(\hat\rho,\hat q)$ is a left state, the right states which can be connected to it by a rarefaction wave are those located on the 1-rarefaction curve $\left\{(\rho,i_{-}^{\eps}(\rho)),\ \rho < \hat\rho\right\}$ or the 2-rarefaction curve $\left\{(\rho,i_{+}^{\eps}(\rho)),\ \rho > \hat\rho\right\}$.

\end{enumerate}
\end{Prop}
\noindent The last point of this proposition stems from the compatibility conditions of the characteristic speeds. The proof of this proposition is classical and omitted. 

A \textbf{shock wave} is a discontinuity between two constant states, $(\hat\rho,\hat q)$ and $(\rho,q)$, travelling at constant speed $\sigma$. The states $(\rho,q)$, which can be connected to $(\hat\rho,\hat q)$ by a shock wave, are determined by the Rankine-Hugoniot conditions:
\begin{equation}
[q] = \sigma[\rho],\quad \left[\frac{q^{2}}{\rho} + \eps p(\rho)\right] = \sigma[q],\label{Eq:Hugoniot}
\end{equation}
where $[f]:= f - \hat f$ for all quantities $f$. Easy computations show that the admissible states are of the form $(\rho,h_{\pm}^{\eps}(\rho))$, where $h_{\pm}^{\eps}$ is : 
\begin{equation}
h_{\pm}^{\eps}(\rho) = \hat u \pm \sqrt{\frac{\rho}{\hat\rho}}\sqrt{(\rho - \hat\rho)(\eps p(\rho) - \eps p(\hat\rho))}.\label{Eq:hugoniot_curve}
\end{equation}
and the shock speeds are:
\begin{equation}
\sigma_{\pm} = \frac{(h_{\pm}^{\eps}(\rho) - \hat q)}{(\rho - \hat\rho)} = \hat u \pm \sqrt{\frac{\rho}{\hat\rho}}\sqrt{\frac{(\eps p(\rho) - \eps p(\hat\rho))}{(\rho - \hat\rho)}}.\label{Eq:shock_speed}
\end{equation}
The graph of $h_{-}^{\eps}$ (resp. $h_{+}^{\eps}$) is called the 1-Hugoniot curve (resp. the 2-Hugoniot curve). The following proposition provides several properties of the Hugoniot curves:
\begin{Prop} \textbf{(Hugoniot and shock curves)}
\begin{enumerate}
\item The 1-Hugoniot function $h_{-}^{\eps}$ (resp. the 2-Hugoniot function $h_{+}^{\eps}$) is concave (resp. convex) and $h_{-}^{\eps}(0) = h_{+}^{\eps}(0) = 0$.
\item The limits of the graphs of $h_{-}^{\eps}$ and $h_{+}^{\eps}$ when $\eps$ goes to zero are the union of the straight lines $\left\{ q = \rho \hat u\right\}$ and $\left\{\rho = \rho^{\ast}\right\}$. This is true also when $\hat\rho  = \hat\rho^{\eps}$ depends on $\eps$ and tends to $\rho^{\ast}$. 
\item If $(\hat\rho,\hat q)$ is a left state, the right states which can be connected to it by an entropic shock wave are those located on the 1-shock curve $\left\{(\rho,h_{-}^{\eps}(\rho)),\ \rho > \hat\rho\right\}$ or the 2-shock curve $\left\{(\rho,h_{+}^{\eps}(\rho)),\ \rho < \hat\rho\right\}$.
\end{enumerate}
\label{Prop:Shock_waves}
\end{Prop}

\noindent The proof of this proposition is classical and omitted.

\begin{figure}
\begin{center}
\includegraphics[viewport=157 370 568 707, scale=0.5]{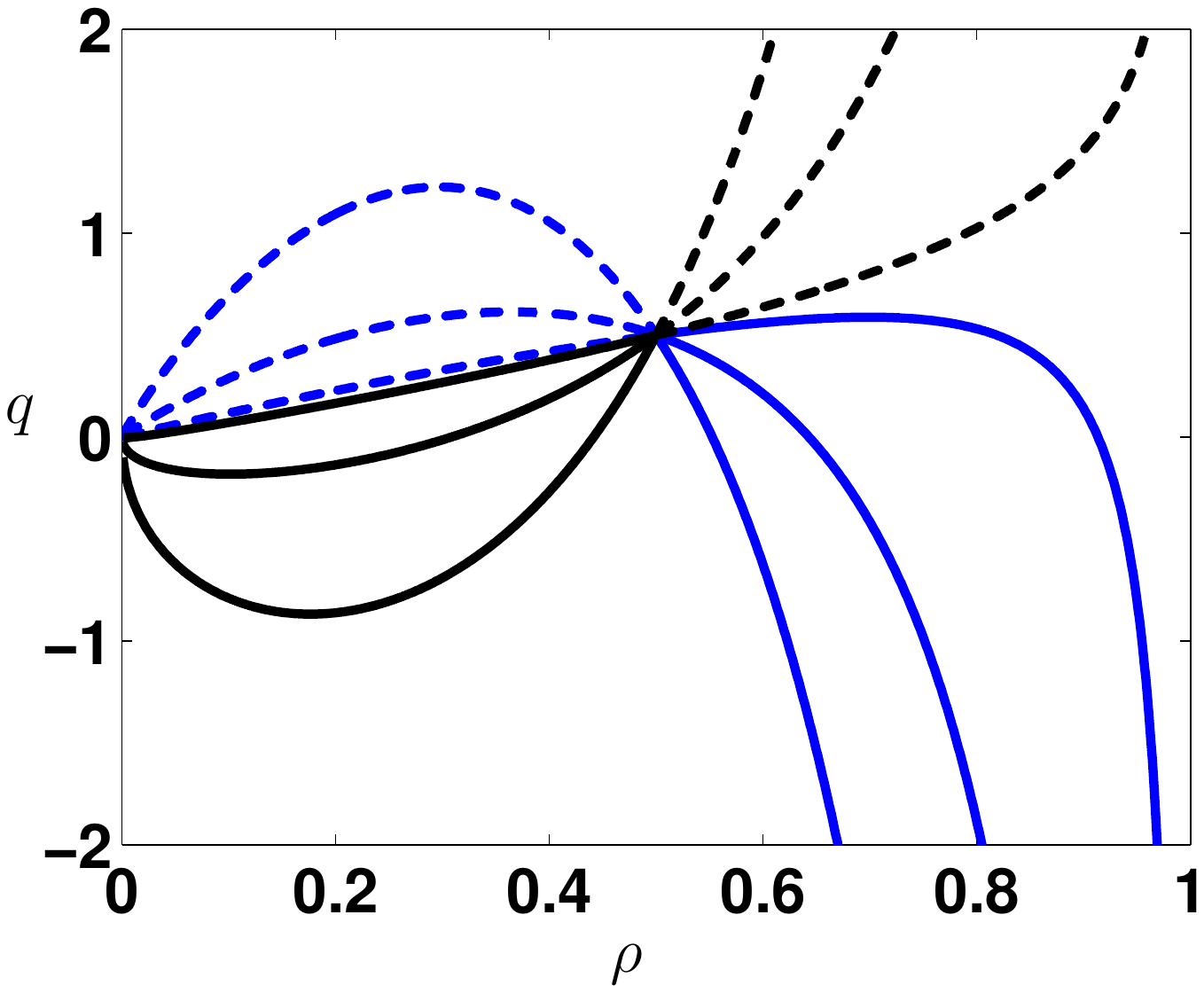}
\caption{Curves issued from $(\hat\rho,\hat q) = (0.5,0.5)$ and $\eps \in \left\{10,1,10^{-2}\right\}$. In dotted line : the rarefaction curves. In solid line : the shock curves. In blue : the 1-curves. In black : the 2-curves. As $\eps \rightarrow 0$, the curves are tending to the straight lines $\left\{q = \hat u\rho\right\}$ and $\left\{\rho = \rho^{\ast}\right\}$. Parameters : $k = 2, \rho^{\ast} = 1$.}
\label{wavecurve_NumCongestion}
\end{center}
\end{figure}

\subsubsection{The Riemann problem for system (\ref{Eq:rho_1D_eps})-(\ref{Eq:q_1D_eps})}

Geometric considerations on the intersection of the integral and Hugoniot curves enable us to solve the Riemann problem \cite{leveque_book_2002}. These arguments are really simplified in the limit $\eps \rightarrow 0$, due to the much simpler behaviour of the integral and Hugoniot curves: they both converge to union of the straight lines $\left\{q = \rho \hat u\right\}$ and $\left\{\rho = \rho^{\ast}\right\}$. This behaviour is illustrated in Fig.~\ref{wavecurve_NumCongestion}. The following proposition provides the nature of the solutions of the Riemann problem for $\eps$ small enough. It depends on the sign of the relative velocity $u_{\ell} - u_{r}$, where $u_{\ell} = q_{\ell}/\rho_{\ell}$ and $u_{r} = q_{r}/\rho_{r}$:

\begin{Prop} Let $(\rho_{\ell},q_{\ell})$, $(\rho_{r},q_{r})$, left and right states. Then the solutions of the Riemann problem related to (\ref{Eq:rho_1D_eps})-(\ref{Eq:q_1D_eps}) have the following forms:
\begin{enumerate}
\item If $u_{\ell} < u_{r}$, then for $\varepsilon$ small enough, the intermediate state is the intersection point of the 1-integral curve issued from $(\rho_{\ell},q_{\ell},\bar{p}_{\ell})$ and the 2-integral curve issued from $(\rho_{r},q_{r},\bar{p}_{r})$. Besides, the intermediate density $\tilde\rho$ is lower than $\rho_{r}$ and $\rho_{\ell}$. This is valid even if $\rho_{\ell}$ or/and $\rho_{r}$ tends to $\rho^\ast$.
\item If $u_{\ell} > u_{r}$, then for $\eps$ small enough, there are two subcases:
\begin{itemize}
\item if $h_{-}^{\eps}(\rho_{r}^{\eps}) > q_{r}$ and $h_{+}^{\eps}(\rho_{\ell}^{\eps}) < q_{\ell}$, the intermediate state is the intersection point of the 1-Hugoniot curve issued from $(\rho_{\ell},q_{\ell})$ and the 2-Hugoniot curve issued from $(\rho_{r},q_{r})$. 
\item if $h_{-}^{\eps}(\rho_{r}^{\eps}) < q_{r}$ (resp. $h_{+}^{\eps}(\rho_{\ell}^{\eps}) > q_{\ell}$), the intermediate state is the intersection point of the 1-Hugoniot curve (resp. 1-integral curve) issued from $(\rho_{\ell},q_{\ell})$ and the 2-integral curve (resp. 2-Hugoniot curve) issued from $(\rho_{r},q_{r})$.
\end{itemize}
\item If $u_{r} = u_{\ell}$ and $\rho_{\ell} < \rho_{r}$ (resp. $\rho_{\ell} > \rho_{r}$), then the intermediate state is the intersection point of the 1-Hugoniot curve (resp. 1-integral curve) issued from $(\rho_{\ell},q_{\ell},\bar{p}_{\ell})$ and the 2-integral curve (resp. 2-Hugoniot curve) issued from $(\rho_{r},q_{r},\bar{p}_{r})$. Besides, the intermediate density $\tilde\rho$ is the interval $[\rho_{\ell},\rho_{r}]$ (resp. $[\rho_{r},\rho_{\ell}]$). This is valid even if $\rho_{\ell}$ or/and $\rho_{r}$ tends to $\rho^\ast$.
\end{enumerate}
\label{Prop:Riemann_pb}
\end{Prop} 
\noindent This proposition results from propositions~\ref{Rarefaction_wave} and \ref{Prop:Shock_waves}. The proof is omitted here (it is an easy adaptation of the proof of Theorem 1 in \cite{2010_Congestion}). Note that the nature of the curves (integral or Hugoniot curve) implies the nature of the waves involved in the Riemann problem.

\subsubsection{The limits of solutions of the Riemann problem.}

We introduce the following initial conditions:
\begin{equation*}
(\rho_{0}^{\eps},q_{0}^{\eps})(x) = \left\{\begin{array}{cl} (\rho_{\ell}^{\eps},q_{\ell}^{\eps}), \text{ for } x < 0,\\
(\rho_{r}^{\eps},q_{r}^{\eps}), \text{ for } x \geq 0.
\end{array}\right.
\end{equation*} 
with $(\rho_{\ell}^{\eps},q_{\ell}^{\eps},\eps p(\rho_{\ell}^{\eps})) \rightarrow (\rho_{\ell},q_{\ell},\bar{p}_{\ell})$ and  $(\rho_{r}^{\eps},q_{r}^{\eps},\eps p(\rho_{r}^{\eps})) \rightarrow (\rho_{r},q_{r},\bar{p}_{r})$ as $\eps$ goes to zero.

The following proposition provides the solution when $\rho_{\ell}, \rho_{r} < \rho^{\ast}$  and so $\bar{p}_{\ell} = \bar{p}_{r} = 0$. The nature of the solution only depends on the sign of the relative velocity $u_{\ell} - u_{r}$, where $u_{\ell} = q_{\ell}/\rho_{\ell}$ and $u_{r} = q_{r}/\rho_{r}$.
\begin{Prop} \textbf{(case $\rho_{\ell} < \rho^{\ast}, \rho_{r} < \rho^{\ast}$)}
\label{Prop:Riemann_problem_limit1}
\begin{enumerate}
	\item If $u_{\ell} - u_{r} < 0$, then the solution consists of two contact waves connecting the two states to the vacuum. This is summarised in the following diagram:
	\begin{equation*}
	(\rho_{\ell},q_{\ell},0)\quad \stackrel{\mbox{contact}}{\longrightarrow}\quad(0,q_{\ell},0) \stackrel{\mbox{vacuum}}{\longrightarrow}\quad(0,q_{r},0)\quad \stackrel{\mbox{contact}}{\longrightarrow}\quad(\rho_{r},q_{r},0).
	\end{equation*}
	\item If $u_{\ell} - u_{r} > 0$, then the solution consists of two shock waves connecting the left state $(\rho_{\ell},q_{\ell},0)$ to an intermediate state $(\rho^{\ast},\tilde{q},\bar{p})$ and then $(\rho^{\ast},\tilde{q},\bar{p})$ to the right state $(\rho_{r},q_{r},0)$: 
	\begin{equation*}
	(\rho_{\ell},q_{\ell},0)\quad \stackrel{\mbox{shock}}{\longrightarrow}\quad(\rho^{\ast},\tilde{q},\bar{p})\quad \stackrel{\mbox{shock}}{\longrightarrow}\quad(\rho_{r},q_{r},0), 
	\end{equation*}
	where the intermediate momentum $\tilde{q}$ and the intermediate pressure $\bar{p}$ are: 
	\begin{eqnarray*}
&&\tilde{q} = u_{\ell}\rho^{\ast} - \sqrt{\frac{\rho^{\ast}}{\rho_{\ell}}}\sqrt{(\rho^{\ast} - \rho_{\ell})\bar{p}} = u_{r}\rho^{\ast} + \sqrt{\frac{\rho^{\ast}}{\rho_{r}}}\sqrt{(\rho^{\ast} - \rho_{r})\bar{p}},\\
&&\bar{p} = \left(u_{\ell} - u_{r}\right)^2\left(\sqrt{\frac{\rho^{\ast} - \rho_{\ell}}{\rho_{\ell}\rho^{\ast}}} + \sqrt{\frac{\rho^{\ast} -\rho_{r}}{\rho_{r}\rho^{\ast}}}\right)^{-2},
\end{eqnarray*}
and the shock speeds $\sigma_{-}$ and $\sigma_{+}$ are given by:
\begin{equation*}
\sigma_{-} = u_{\ell} - \sqrt{\frac{\rho^{\ast}}{\rho_{\ell}(\rho^{\ast} - \rho_{\ell})}}\sqrt{\bar{p}},\quad \sigma_{+} = u_{r} + \sqrt{\frac{\rho^{\ast}}{\rho_{r}(\rho^{\ast} - \rho_{r})}}\sqrt{\bar{p}}.
\end{equation*}
	\item If $u_{\ell} = u_{r}$, then the solution consists of only one contact wave connecting $(\rho_{\ell},q_{\ell},0)$ to $(\rho_{r},q_{r},0)$:  
	\begin{equation*}
	(\rho_{\ell},q_{\ell},0)\quad \stackrel{\mbox{contact}}{\longrightarrow}\quad(\rho_{r},q_{r},0).  
	\end{equation*}
\end{enumerate}
\end{Prop}
\noindent It is the same results as those obtained in \cite{2PhaseFlow_Bouchut_al}, where they directly proved it by defining a notion of entropy solutions for the asymptotic problem.

In all the following proofs, $(\rho_{\ell},q_{\ell})$, $(\tilde\rho,\tilde q)$, $(\rho_{r},q_{r})$ will respectively denote the left, the intermediate and the right states involved in each different Riemann problems. $\lambda_{-}^{\eps}$ (resp. $\lambda_{+}^{\eps}$) will implicitly refer to the first (resp. second) characteristic speed of the left state (resp. right state). The characteristic speeds related to the intermediate state are denoted: $\tilde{\lambda}_{-}^{\eps}$, $\tilde{\lambda}_{+}^{\eps}$.  In the following, $i_{-}^{\eps}$ and $h_{-}^{\eps}$ (resp. $i_{+}^{\eps}$ and $h_{+}^{\eps}$) will refer to the 1-curves (resp. 2-curves) issued from the left state (resp. right state). The notation   $[f]_{\ell} = \tilde f - f_{\ell}$ (resp. $[f]_{r} = \tilde f - f_{r}$) will denote the difference between the intermediate and the left (resp. the right) values of any quantity $f$. 

\begin{proof} 
1. From proposition~\ref{Prop:Riemann_pb}, the solution for small $\eps$ consists of two rarefaction waves. The intermediate density solves equation $i_{-}^{\eps}(\tilde\rho) = i_{+}^{\eps}(\tilde\rho)$, that is:
\begin{equation*}
\tilde\rho u_r \pm \tilde\rho\sqrt{\eps}[P(\rho)]_{r} = \tilde\rho u_\ell \pm \tilde\rho\sqrt{\eps}[P(\rho)]_{\ell}
\end{equation*}
Since $\tilde\rho$ is lower than $\rho_{\ell}$ and $\rho_{r}$ (and then $[P(\rho)]_{r}$ and $[P(\rho)]_{\ell}$ are bounded) and $u_r - u_\ell$ is not zero, it is easy to deduce that the limit solution of this equation is $\tilde\rho = 0$, which defines a vacuum state. Besides, $\lambda_{-}^{\eps}$ and $\tilde{\lambda}_{-}^{\eps}$ tends to $u_{\ell}$ (resp. $\tilde{\lambda}_{+}^{\eps}$ and  $\lambda_{+}^{\eps}$ tends to $u_{r}$). Therefore, the limit waves are contact waves.

2. From proposition~\ref{Prop:Riemann_pb}, the solution for small $\eps$ consists of two shock waves (since $\lim h_{+}^{\eps}(\rho_\ell) = \rho_\ell u_r$ and $\lim_{-}^{\eps} h_{-}^{\eps}(\rho_\ell)  = \rho_r q_\ell$).  We have $h_{-,\ell}^{\eps}(\tilde\rho) = h_{+,r}^{\eps}(\tilde\rho)$, that is:
\begin{equation*}
\tilde{\rho}u_{\ell} - \sqrt{\frac{\tilde{\rho}}{\rho_{\ell}}}\sqrt{[\rho]_{\ell}[\eps p(\rho)]_{\ell}} = \tilde{\rho}u_{r} + \sqrt{\frac{\tilde{\rho}}{\rho_{r}}}\sqrt{[\rho]_{r}[\eps p(\rho)]_{r}},
\end{equation*}
which yields
\begin{equation*}
\left(u_{\ell} - u_{r}\right) = \sqrt{\frac{[\rho]_{\ell}[\eps p(\rho)]_{\ell}}{\rho_{\ell}\tilde{\rho}}} + \sqrt{\frac{[\rho]_{r}[\eps p(\rho)]_{r}}{\rho_{r}\tilde{\rho}}},
\end{equation*}
Since $u_{\ell} - u_{r}$ is different from zero, the limit intermediate pressure,
$\lim \eps p(\tilde\rho^{\eps}) = \lim [\eps p(\rho)]_{\ell}$ $ = \lim [\eps p(\rho)]_{r}$, 
is not zero. Thus $\tilde\rho \rightarrow \rho^{\ast}$ as $\eps \rightarrow 0$. Finally, the limit values of the two shock speeds can be easily inferred from equation (\ref{Eq:shock_speed}).

3. Supposing that $\rho_{\ell} < \rho_{r}$,  then the intermediate density satisfies $\rho_{\ell} \leq \tilde{\rho}^{\eps} \leq \rho_{r}$ and consequently $\lambda_{-}^{\eps},\tilde{\lambda}_{\pm}^{\eps},\lambda_{+}^{\eps} \rightarrow u_{\ell} = u_{r}$ as $\eps$ goes to zero, which yields a unique contact wave.
\end{proof}

%

We now consider the case where the left state is a congested state: $\rho_{\ell} = \rho^{\ast}, \rho_{r} < \rho^{\ast}$ and $\bar{p}_{\ell} < + \infty$. By a symmetry argument, the case $\rho_{\ell} < \rho^{\ast}, \rho_{r} = \rho^{\ast}$ can be easily deduced. The limits of rarefactions waves when one state tends to congestion lead to the so-called declustering waves:
\begin{Def} \textbf{(Declustering waves)} A declustering wave consists of a shock wave between two congested state, with infinite speed, and with a zero pressure for positive time.
\label{Def:Declustering}
\end{Def}
\noindent The following proposition states the solutions of the Riemann problem.
\begin{Prop} \textbf{(case $\rho_{\ell} = \rho^{\ast}, \rho_{r} < \rho^{\ast}$)}
\label{Prop:Riemann_problem_limit2}
\begin{enumerate}
	\item If $u_{\ell} - u_{r} < 0$, then the solution consists of one declustering wave connecting the left state $(\rho^{\ast},q_{\ell},\bar{p}_{\ell})$ to a congested and pressureless state $(\rho^{\ast},q_{\ell},0)$ and then a contact wave connecting $(\rho^{\ast},q_{\ell},0)$ to vacuum and another contact wave connecting vacuum to the right state $(\rho_{r},q_{r},0)$:
\begin{footnotesize}	\begin{equation*}
	(\rho^{\ast},q_{\ell},\bar{p}_{\ell})\ \stackrel{\mbox{declust.}}{\longrightarrow}\ (\rho^{\ast},q_{\ell},0)\ \stackrel{\mbox{contact}}{\longrightarrow}\ (0,q_{\ell},0)\ \stackrel{\mbox{vacuum}}{\longrightarrow}\ (0,q_{r},0)\  \stackrel{\mbox{contact}}{\longrightarrow}\ (\rho_{r},q_{r},0).  
	\end{equation*}\end{footnotesize}
	\item If $u_{\ell} - u_{r} > 0$, then the solution consists of two shock waves connecting the left state $(\rho^{\ast},q_{\ell},\bar{p}_{\ell})$ to an intermediate congested state $(\rho^{\ast},\tilde{q},\bar{p})$ and then connecting this intermediate state to the right state speeds $(\rho_{r},q_{r},0)$: 
	\begin{equation*}
	(\rho^{\ast},q_{\ell},\bar{p}_{\ell})\quad \stackrel{\mbox{shock}}{\longrightarrow}\quad(\rho^{\ast},\tilde{q},\bar{p})\quad \stackrel{\mbox{shock}}{\longrightarrow}\quad(\rho_{r},q_{r},0),  
	\end{equation*}
	where the intermediate momentum $\tilde{q}$ and the intermediate pressure $\bar{p}$ are given by: 	
	\begin{equation*}
	\tilde{q} = \rho^{\ast}u_{\ell},\quad \bar{p} = \frac{\rho^{\ast}\rho_{r}}{\rho^{\ast} - \rho_{r}}(u_{\ell} - u_{r})^{2},
	\end{equation*}
	and the two speed $\sigma_{-}$ and $\sigma_{+}$ are:
	\begin{equation*}
	\sigma_{-} = - \infty,\quad \sigma_{+} = u_{r} + \sqrt{\frac{\rho^{\ast}}{\rho_{r}}}\sqrt{\frac{\bar{p}}{\rho^{\ast} - \rho_{r}}}. 
	\end{equation*}
	\item If $u_{\ell} = u_{r}$, then the solution consists of one declustering wave connecting the left state $(\rho^{\ast},q_{\ell},\bar{p}_{\ell})$ to the intermediate state $(\rho^{\ast},q_{\ell},0)$ and one contact wave connecting the intermediate state to the right state $(\rho_{r},q_{r},0)$: 
	\begin{equation*}
	(\rho^{\ast},q_{\ell},\bar{p}_{\ell})\quad \stackrel{\mbox{declust.}}{\longrightarrow}\quad(\rho^{\ast},q_{\ell},0)\quad \stackrel{\mbox{contact}}{\longrightarrow}\quad(\rho_{r},q_{r},0).  
	\end{equation*}  
\end{enumerate}
\end{Prop}

\begin{proof} 1.  From proposition~\ref{Prop:Riemann_pb}, the solution for small $\eps$ consists of two rarefaction waves and it can be easily checked that the limit intermediate density is zero (thanks to proposition~\ref{Rarefaction_wave}, point 3, $\sqrt{\eps}[P(\rho)]_{\ell} = O(\eps^{\frac{1}{2\gamma}})$). Since $\rho_{\ell}$ tends to $\rho^{\ast}$, we have $\lim\lambda_{-}^{\eps} = - \infty$ and $\lim\tilde\lambda_{-}^{\eps} = u_{\ell}$. The limit of the 2-rarefaction wave is a contact wave (since  $\lim\tilde\lambda_{+}^{\eps} = \lim\lambda_{+}^{\eps} = u_r$). Let us look at the limit of the 1-rarefaction wave. 

For each possible speed $s \in [\lambda_{-}^{\eps},u_{\ell}]$ of the rarefaction wave connecting the left state to the intermediate state, we have: 
\begin{equation}
s = \lambda_{-}(\rho(s),q(s)) = \frac{q(s)}{\rho(s)} - \sqrt{\eps p'(\rho(s))}.\label{Eq:rarefaction_proof}
\end{equation}
The state $(\rho(s),q(s))$ belongs to the integral curve and then for $\rho(s) < \rho_{\ell}$ and according to proposition~\ref{Rarefaction_wave}, point 3, $q(s)/\rho(s)$ tends to $u_{\ell}$. If $s \neq u_{\ell}$, equation (\ref{Eq:rarefaction_proof}) implies that $\rho(s)$ has to tend to $\rho^{\ast}$ and that $\lim \eps p(\rho(s)) = 0$ since $\lim \eps p'(\rho(s))$ is finite. This yields the definition of a declustering wave, given in definition~\ref{Def:Declustering}. 

2. We have $\lim h_{-}^{\eps}(\rho_{r}) = \rho_r u_\ell > q_r$. According to proposition~\ref{Prop:Riemann_pb}, in order to have a solution which is the limit of two shock waves, we have to prescribe: 
\begin{eqnarray*}
h_{+}^{\eps}(\rho_{\ell})  = \rho_{\ell}u_{r} + \sqrt{\eps}\sqrt{\frac{\rho_{\ell}}{\rho_{r}}}\sqrt{[\rho]_{r}^{\ell}[p(\rho)]_{r}^{\ell}} < q_{\ell} = \rho_{\ell}u_{\ell},
\end{eqnarray*}
and in the limit:
\begin{equation*}
u_{r} + \sqrt{\frac{(\rho_{\ast} - \rho_{r})}{\rho_{r}\rho_{\ast}}}\sqrt{\bar{p}_{\ell}} \leq u_{\ell}.
\end{equation*}
If $h_{+}^{\eps}(\rho_{\ell}) > q_{\ell}$, the solution is the limit of a 1-rarefaction wave and a 2-shock wave. According to this discussion, we have to consider two cases but we will see that the limits of the two cases are the same. 

The first case corresponds to the limit of two shock waves. The intermediate density is then greater than the left one and so tends to $\rho^{\ast}$ too. Besides we have: 
\begin{equation}
\tilde{\rho} u_{\ell} - \sqrt{\frac{\tilde{\rho}}{\rho_{\ell}}}\sqrt{[\rho]_{\ell}[\eps p(\rho)]_{\ell}} = \tilde{\rho} u_{r} + \sqrt{\frac{\tilde{\rho}}{\rho_{r}}}\sqrt{[\rho]_{r}[\eps p(\rho)]_{r}},\label{Ineq:u}
\end{equation}
and then by dividing by $\tilde{\rho}$, we get:
\begin{equation*}
u_{\ell} - \sqrt{\frac{[\rho]_{\ell}[\eps p(\rho)]_{\ell}}{\tilde{\rho}\rho_{\ell}}} = u_{r} + \sqrt{\frac{[\rho]_{r}[\eps p(\rho)]_{r}}{\tilde{\rho}\rho_{r}}} < u_{\ell}.
\end{equation*}
The last inequality implies that the limit of $\eps p(\tilde{\rho})$ is finite. We denote it by $\bar{p}$. Taking the limit in the equality (\ref{Ineq:u}), we obtain
\begin{equation*}
u_{\ell} = u_{r} + \sqrt{\frac{(\rho^{\ast} - \rho_{r})}{\rho^{\ast}\rho_{r}}}\sqrt{\bar{p}}.
\end{equation*}
Hence we find the value of $\bar{p}$ as in proposition~\ref{Prop:Riemann_problem_limit2}. The propagation speeds are easily deduced from the limit $\eps \rightarrow 0$ in (\ref{Eq:shock_speed}).

 If the solution is the limit of a combination of a rarefaction wave and a shock wave, then the intermediate states satisfies.
\begin{equation*}
\tilde{\rho} u_{\ell} - \tilde{\rho}\sqrt{\eps}[P(\rho)]_{\ell} = \tilde{\rho} u_{r} + \sqrt{\frac{\tilde{\rho}}{\rho_{r}}}\sqrt{[\rho]_{r}[\eps p(\rho)]_{r}}.
\end{equation*}
If the intermediate density $\tilde\rho^{\eps}$ did tend to $\rho^{\ast}$, then we would obtain $u_{\ell} = u_r$ (thanks to proposition \ref{Rarefaction_wave}), which is impossible. Thus the intermediate density tends to $\rho^{\ast}$ and the previous expression tends to: 
\begin{equation*}
\rho^{\ast} u_{\ell} = \rho^{\ast} u_{r} + \sqrt{\frac{\rho^{\ast}}{\rho_{r}}}\sqrt{(\rho^{\ast} - \rho_{r})\bar{p}},
\end{equation*}
which yields the expected result. Since the pressure is positive, $\tilde{\lambda}_{-}^{\eps}$ tends to $-\infty$ which implies that the rarefaction wave turns into a shock wave with infinite propagation speed, i.e. a declustering wave.

3. From proposition~\ref{Prop:Riemann_pb}, the solution is the limit of a combination of a 1-rarefaction wave and a 2-shock wave. Using the fact that $\tilde\rho u_{\ell} = \tilde\rho u_{r}$, the intermediate state $(\tilde\rho,\tilde q)$ satisfies:
\begin{equation*}
\tilde{\rho}\eps[P(\rho)]_{\ell} = \sqrt{\frac{\rho}{\rho_{\ell}}}\sqrt{[\rho]_{r}[\eps p(\rho)]_{r}}.  
\end{equation*}
So $[\rho]_{r}[\eps p(\rho)]_{r} \rightarrow  0$ (since $\eps P(\tilde{\rho}) < \eps P(\rho_{\ell})$ tends to zero) and either $\tilde{\rho}$ tends to $\rho_{r}$ or $\eps p(\tilde{\rho})$ tends to zero. Actually, whatever the limit value of the intermediate density, the intermediate state disappears. Indeed, let us consider all the possible cases.  

Either $\tilde{\rho} \rightarrow \rho_{r}$. Then the 2-shock wave disappears and the 1-rarefaction wave tends to the sum of a declustering wave and a contact wave, which can be proven as in the case 1 of this proof.

Or $\lim \tilde{\rho} \in ]\rho_{r},\rho^{\ast}[$. It is easy to check that the 2-shock wave becomes a contact wave and the intermediate state disappears since $\tilde\lambda_{+}^{\eps}, \tilde\lambda_{-}^{\eps} \rightarrow u_{\ell}$. Like in the case 1 of this proof, the 1-rarefaction wave leads to declustering wave and a contact wave, which superimposes on the one coming from the 2-shock wave.

Or $\lim \tilde{\rho} = \rho^{\ast}$. Then we know that $\eps p(\tilde{\rho})$ tends to zero. Thus the 1-rarefaction wave tends to a declustering wave and the 2-shock wave tends to a contact wave. 

In all the cases, the limit solution is the one given in proposition~\ref{Prop:Riemann_problem_limit2}.
\end{proof}

Finally, we consider the case where both the left and right asymptotic states are congested: $\rho_{\ell} = \rho_{r} = \rho^{\ast}$. Besides, we assume that $\bar{p}_{\ell}$ and $\bar{p}_{r}$ are finite.
\begin{Prop} \textbf{(case $\rho_{\ell} = \rho^{\ast}, \rho_{r} = \rho^{\ast}$, $\rho^{\eps}_{\ell} > \rho^{\eps}_{r}$)}
\label{Prop:Riemann_problem_limit3}
\begin{enumerate}
	\item If $u_{\ell} - u_{r} < 0$, then the solution consists of one declustering wave connecting the left state $(\rho^{\ast},q_{\ell},\bar{p}_{\ell})$ to a congested and pressureless state $(\rho^{\ast},q_{\ell},0)$, then two contact waves connecting $(\rho^{\ast},q_{\ell},0)$ to vacuum and then vacuum to $(\rho^{\ast},q_{r},0)$ and another declustering wave connecting $(\rho^{\ast},q_{r},0)$ to the right state $(\rho_{r},q_{r},0)$:
	\begin{footnotesize}\begin{equation*}
	(\rho^{\ast},q_{\ell},\bar{p}_{\ell})\stackrel{\mbox{declust.}}{\longrightarrow}(\rho^{\ast},q_{\ell},0) \stackrel{\mbox{contact}}{\longrightarrow}(0,q_{\ell},0)\stackrel{\mbox{vacuum}}{\longrightarrow}(0,q_{r},0)\stackrel{\mbox{contact}}{\longrightarrow}(\rho^{\ast},q_{r},0) \stackrel{\mbox{declust.}}{\longrightarrow}(\rho^{\ast},q_{r},\bar{p}_{r}). 
	\end{equation*}\end{footnotesize}
	\item If $u_{\ell} - u_{r} > 0$, then the solution consists of two shock waves with infinite propagation speed connecting the left state $(\rho^{\ast},q_{\ell},\bar{p}_{\ell})$ to an intermediate congested state $(\rho^{\ast},\tilde{q},+ \infty)$ with infinite pressure and then this intermediate state to the right state $(\rho^{\ast},q_{r},\bar{p}_{r})$: 
	\begin{equation*}
	(\rho^{\ast},q_{\ell},\bar{p}_{\ell})\quad \stackrel{\mbox{shock}}{\longrightarrow}\quad(\rho^{\ast},\tilde{q},+ \infty)\quad \stackrel{\mbox{shock}}{\longrightarrow}\quad(\rho^{\ast},q_{r},\bar{p}_{r}),  
	\end{equation*}
where $\tilde{q}$ is the unique solution of: 
	\begin{equation*}
	\frac{[q]_{\ell}}{[q]_{r}} = \left(\frac{\bar p_{r}}{\bar p_{\ell}}\right)^{\frac{1}{2\gamma}}.
	\end{equation*}
	\item If $u_{\ell} = u_{r}$, then the solution consists of a uniform constant state $(\rho^{\ast},q_{r},\bar p_{r})$.
\end{enumerate}
\end{Prop}

\begin{proof} 1. The arguments are similar to those used in the proof of the first case of the previous proposition.

2. From proposition~~\ref{Prop:Riemann_pb}, the solution is the limit of two shock waves if for small $\eps$, we have the following inequalities:
\begin{eqnarray*}
&&h_{+}^{\eps}(\rho_{\ell}^{\eps})  = \rho_{\ell}^{\eps}u_{r} + \sqrt{\eps}\sqrt{\frac{\rho_{\ell}^{\eps}}{\rho_{r}^{\eps}}}\sqrt{[\rho]_{r}^{\ell}[p(\rho)]_{r}^{\ell}} < q_{\ell} = \rho_{\ell}^{\eps}u_{\ell},\\
&&h_{-}^{\eps}(\rho_{r}^{\eps})  = \rho_{r}^{\eps}u_{\ell} + \sqrt{\eps}\sqrt{\frac{\rho_{r}^{\eps}}{\rho_{\ell}^{\eps}}}\sqrt{[\rho]_{r}^{\ell}[p(\rho)]_{r}^{\ell}} > q_{r} = \rho_{r}^{\eps}u_{r}.
\end{eqnarray*}
Here, these inequalities are always satisfied : their limit is $u_{r} < u_{\ell}$ since $\eps[p(\rho)]_{r}^{\ell}$ is bounded and $[\rho]_{r}^{\ell} \rightarrow 0$.
The intermediate density $\tilde\rho$ satisfies:
\begin{equation*}
u_{\ell} - \sqrt{\frac{[\rho]_{\ell}[\eps p(\rho)]_{\ell}}{\tilde{\rho}\rho_{\ell}}} = u_{r} + \sqrt{\frac{[\rho]_{r}[\eps p(\rho)]_{r}}{\tilde{\rho}\rho_{r}}}.
\end{equation*}
Therefore, $\eps p(\tilde{\rho})$ cannot be bounded (which would imply $u_{\ell} = u_{r}$ since $[\rho]_{\ell}$ and $[\rho]_{r}$ tends to zero). So $\eps p(\tilde{\rho})$ tends to $+\infty$ as $\eps \rightarrow 0$. From the Rankine-Hugoniot relations, we have 
\begin{eqnarray*}
&& [\rho q + \eps p(\rho)]_{\ell}[\rho]_{\ell} = [q]_{\ell}^2,\\
&& [\rho q + \eps p(\rho)]_{r}[\rho]_{r} = [q]_{r}^2
\end{eqnarray*}
Taking the limit of their quotient, we obtain:
\begin{equation}
\lim \frac{[\eps p(\rho)]_{\ell}[\rho]_{\ell}}{[\eps p(\rho)]_{r}[\rho]_{r}} = \left(\frac{[q]_{\ell}}{[q]_{r}}\right)^2,\label{Eq:collision_Riemann}
\end{equation} 
Besides, we have:
\begin{equation*}
\frac{[\eps p(\rho)]_{\ell}[\rho]_{\ell}}{[\eps p(\rho)]_{r}[\rho]_{r}} \underset{\eps \rightarrow 0}{\sim} \frac{[\rho]_{\ell}}{[\rho]_{r}} \underset{\eps \rightarrow 0}{\sim} \frac{\rho^{\ast} - \rho_{\ell}}{\rho^{\ast} - \rho_{r}}
\end{equation*}
where the last equivalence results from the fact that $(\rho^{\ast} - \tilde\rho) = o(\eps^{\frac{1}{\gamma}})$ and $(\rho^{\ast} - \rho_{\ell,r}) = O(\eps^{\frac{1}{\gamma}})$. Finally, we have:
\begin{equation*}
\frac{\rho^{\ast} - \rho_{\ell}}{\rho^{\ast} - \rho_{r}} = \left(\frac{\eps p(\rho_r)}{\eps p(\rho_\ell)}\right)^{\frac{1}{\gamma}} \underset{\eps \rightarrow 0}{\rightarrow} \left(\frac{\bar p_r}{\bar p_\ell}\right)^{\frac{1}{\gamma}} 
\end{equation*}
This last result, combined with eq. (\ref{Eq:collision_Riemann}), provides an equation for the intermediate momentum.

3. Let us suppose that $\rho_{\ell}^{\eps} > \rho_{r}^{\eps}$. The intermediate density satisfies $\rho_{r}^{\eps}<\tilde\rho^{\eps}<\rho_{\ell}^{\eps}$ and so tends to $\rho^{\ast}$. The intermediate momentum $\tilde{q}$ tends to $\rho^{\ast}u_{\ell}$. The 1-rarefaction wave tends to a shock wave with infinite propagation speed (since $\tilde{\lambda}_{-}^{\eps}$ and $\lambda_{-}^{\eps}$ tend to $- \infty$), i.e. a declustering wave.  We have now to determine $\bar{p}$. We have 
\begin{equation*}
\sqrt{\frac{\tilde\rho}{\rho_{r}}}\sqrt{[\rho]_{r}[\eps p(\rho)]_{r}} = - \sqrt{\eps}[P(\rho)]_{\ell}.
\end{equation*}
Since $\tilde\rho - \rho_{\ell} \rightarrow 0$, we have $-[P(\rho)]_{\ell} \sim (\rho_{\ell} - \tilde\rho)P'(\rho_{\ell})$. We have $P'(\rho_{\ell}) \sim C(\rho^{\ast} - \rho_{\ell})^{-\frac{\gamma+1}{2}}$. Then we have $P'(\rho_{\ell}) \sim C(\eps^{\frac{1}{\gamma}})^{-\frac{\gamma+1}{2}} = C \eps^{-\frac{1}{2} - \frac{1}{2\gamma}}$ (since $(\rho_{\ast} - \rho_{\ell}) = O(\eps^{\frac{1}{\gamma}})$). Besides $(\rho_{\ell} - \tilde\rho) \leq (\rho_{\ast} - \rho_{r}) = O(\eps^{\frac{1}{\gamma}})$ and $[\rho]_{r} = O(\eps^{\frac{1}{\gamma}})$ and , so we have
\begin{equation*}
[\eps p(\rho)]_{r} = \frac{\rho_{r}}{\tilde\rho}\frac{(-\sqrt{\eps}[P(\rho)]_{\ell})^{2}}{[\rho]_{r}} = O(\eps^{\frac{1}{\gamma}}).
\end{equation*} 
So $\eps p(\tilde{\rho})$ tends to $\bar{p}_{r}$. The 2-shock wave disappears (since the intermediate and the right state are identical). Then the limit solution consists of an instantaneous propagation of the right state. 
\end{proof}

\subsection{The one-dimensional cluster collisions} 

The Riemann problem where both the left and right states are congested is ill-posed since infinite pressure may appear to correct the discontinuity in the pressure in the incompressible domain. Like in \cite{2010_Congestion,2PhaseFlow_Bouchut_al,2002_ExistPressLess_Berth}, we have to restrict to the collision of finite congested domains. Consider two one dimensional clusters which collides at a time $t_{c}$. Before collision, the left (resp. right) cluster at time $t < t_{c}$ extends between $a_{\ell}(t)$ and $b_{\ell}(t)$ (resp. $a_{r}(t)$ and $b_{r}(t)$) and moves with speed:
\begin{equation*}
u_{\ell} = a_{\ell}'(t) = b_{\ell}'(t)\quad (\text{resp. } u_{r} = a_{r}'(t) = b_{r}'(t)).
\end{equation*} After collision, the two clusters aggregate and form a new cluster at time $t > t_{c}$ extending between $a(t)$ and $b(t)$ and moving with speed $u = a'(t) = b'(t)$. Therefore, $\rho$ and $u$ are given for $t < t_{c}$ by:
\begin{equation*}
\rho = \rho^{\ast} \mathrm{1}_{[a_{\ell}(t),b_{\ell}(t)]} + \rho^{\ast}\mathrm{1}_{[a_{r}(t),b_{r}(t)]},\quad u = u_{\ell}\mathrm{1}_{[a_{\ell}(t),b_{\ell}(t)]} + u_{r}\mathrm{1}_{[a_{r}(t),b_{r}(t)]},
\end{equation*}
and for $t > t_{c}$ by:
\begin{equation*}
\rho = \rho^{\ast} \mathrm{1}_{[a(t),b(t)]},\quad u = u \mathrm{1}_{[a(t),b(t)]}.
\end{equation*}
where $\mathrm{1}_{I}$ denotes the indicator function of interval $I$.
We denote by $m = b_{\ell}(t_{c}) = a_{r}(t_{c})$ the collision point. We look for a pressure written as $\bar{p}(x,t) = \pi(x)\delta(t-t_{c})$. Conditions to have such kind of solution for the one-dimensional version of (\ref{Eq:mass})-(\ref{Eq:momentum})-(\ref{Eq:constraint_CongNum}) was obtained in \cite{2PhaseFlow_Bouchut_al}. We report them in the following proposition. 
\begin{Prop}\label{Prop:cluster_dynamics} 1- Supposing that $\bar{p}(x,t) = \pi(x)\delta(t-t_{c})$ where $\pi$ is continuous and zero outside the clusters, then $u$ and $\pi$ have to satisfy 
\begin{eqnarray}
&&(u - u_{\ell})(m - a(t_{c})) + (u - u_{r})(b(t_{c}) - m) = 0,\nonumber\\
&&\label{Eq:collision_u_Euler}\\
&&\pi(x) = \left\{\begin{array}{ll}\rho^{\ast}(u - u_{\ell})(m - x)\\
 \quad + \rho^{\ast}(u - u_{r})(b(t_{c}) - m),&\text{ if }x \in [a(t_{c}),m],\\
\rho^{\ast}(u - u_{r})(b(t_{c}) - x),&\text{ if }x \in [m,b(t_{c})],\end{array}\right. \label{Eq:collision_pi_Euler}
\end{eqnarray}
2 - Under conditions (\ref{Eq:collision_u_Euler})-(\ref{Eq:collision_pi_Euler}), $(\rho,\rho u,\bar{p})$ is a solution (in a distributional sense) to the one-dimensional version of system (\ref{Eq:mass})-(\ref{Eq:momentum})-(\ref{Eq:constraint_CongNum}).
\end{Prop}

\noindent Remark : We note that the solutions where the two clusters aggregate after the collision is one among possibly multiple solutions. Another solution is given by a bounce of the two clusters one against each other leading to a reflection of the velocities. There are infinitely many such reflections which preserve total momentum.

\section[Appendix : The two dimensional full time and space
  discretization]{Appendix : The two dimensional full time and space
  discretization}
\label{sec:2d-case}

\paragraph{The Direct method: }
We consider the 2D case with domain $\Omega= [0,1]\times
[0,1]$. Denote $(x_i,y_j)=(i\Delta x,j \Delta y),\,
i=0,\cdots,M_1;j=0,\cdots,M_2$, where $M_1=1/\Delta x,M_2=1/\Delta
y$. Let $U=(\rho,{\boldsymbol{q}})^T,{\boldsymbol{q}}=(q_1,q_2)^T$ and $U_{i,j}=U(x_j,y_j)$. To
simplify the notation, we define
\begin{gather*}
  \boldsymbol{F}(U)=
  \begin{pmatrix}
   \frac{q_1^2}{\rho}+\varepsilon p_0(\rho)\\
\frac{q_1q_2}{\rho}
  \end{pmatrix}, 
\boldsymbol{G}(U)=
  \begin{pmatrix}
   \frac{q_1q_2}{\rho}\\
\frac{q_2^2}{\rho}+\varepsilon p_0(\rho)
  \end{pmatrix}.
\end{gather*}
Then \eqref{Eq:Euler_rho} and \eqref{Eq:Euler_q} can be written as
\begin{align*}
&\partial_t \rho+ \partial_x q_1+\partial_y q_2=0,\\
&\partial_t{\boldsymbol{q}} +\partial_x \boldsymbol{F}(U)+\partial_y \boldsymbol{G}(U)+\nabla_{\boldsymbol{x}}(\varepsilon p_1(\rho))=0.
\end{align*}
Denote
\begin{align*}
   &\boldsymbol{F}^{n}=
  \begin{pmatrix}
\frac{(q_1^n)^2}{\rho^n}+\varepsilon p_0(\rho^n)\\
\frac{q_1^nq_2^n}{\rho^n}
  \end{pmatrix},\quad 
\boldsymbol{G}^{n}=
  \begin{pmatrix}
\frac{q_1^nq_2^n}{\rho^n}\\
\frac{(q_2^n)^2}{\rho^n}+\varepsilon p_0(\rho^n)
  \end{pmatrix},\\
&\nabla_{i,j}(\varepsilon p_1)^{n+1}=
\begin{pmatrix}
D_{i,j}^x(\varepsilon  p_1(\rho))^{n+1}\\
D_{i,j}^y(\varepsilon p_1(\rho))^{n+1}
\end{pmatrix}=
\begin{pmatrix}
D_{i,j}^x(\varepsilon  p_1(\rho^{n+1}))\\
D_{i,j}^y(\varepsilon p_1(\rho^{n+1}))
\end{pmatrix},
\end{align*}
where $D_{i,j}^xu,D_{i,j}^yu$ are the centered difference operators for
any scalar functions $u$ defined as follows
\begin{gather*}
  D_{i,j}^xu=\frac{u_{i+1,j}-u_{i-1,j}}{2\Delta x},\quad D_{i,j}^yu=\frac{u_{i,j+1}-u_{i,j-1}}{2\Delta y}. 
\end{gather*}
We also define the eigenvalues of the Jacobian matrix for two
one-dimensional hyperbolic system as follows:
\begin{gather*}
  \lambda^{(1)}=\frac{q_1}{\rho},\frac{q_1}{\rho}\pm \sqrt{\varepsilon
    p_0'(\rho)}, \quad \lambda^{(2)}=\frac{q_2}{\rho},\frac{q_2}{\rho}\pm \sqrt{\varepsilon p_0'(\rho)}.
\end{gather*}
With the above notations, the full discretization of the scheme takes the following form:
\begin{align*}
&\frac{\rho_{i,j}^{n+1}-\rho_{i,j}^{n}}{\Delta t}+ \frac{1}{\Delta x}\bigg(Q^{n+\frac{1}{2}}_{i+\frac{1}{2},j}-Q^{n+\frac{1}{2}}_{i-\frac{1}{2},j}\bigg)+ \frac{1}{\Delta y}\bigg(\tilde{Q}^{n+\frac{1}{2}}_{i,j+\frac{1}{2}}-\tilde{Q}^{n+\frac{1}{2}}_{i,j-\frac{1}{2}}\bigg)=0,\\
   &\frac{{\boldsymbol{q}}_{i,j}^{n+1}-{\boldsymbol{q}}_{i,j}^{n}}{\Delta t}+ \frac{1}{\Delta x}\bigg(\boldsymbol{F}^{n}_{i+\frac{1}{2},j}-\boldsymbol{F}^{n}_{i-\frac{1}{2},j}\bigg)+ \frac{1}{\Delta y}\bigg(\boldsymbol{G}^{n}_{i,j+\frac{1}{2}}-\boldsymbol{G}^{n}_{i,j-\frac{1}{2}}\bigg)+\nabla_{i,j}(\varepsilon p_1)^{n+1}=0,
 \end{align*}
where the fluxes are 
\begin{align*}
&Q^{n+\frac{1}{2}}_{i+\frac{1}{2},j}=\frac{1}{2}\left\{(q_1)^{n+1}_{i+1,j}+(q_1)^{n+1}_{i,j}\right\}-\frac{1}{2}C_{i+\frac{1}{2},j}(\rho^n_{i+1,j}-\rho^n_{i,j}),\\
&\tilde{Q}^{n+\frac{1}{2}}_{i,j+\frac{1}{2}}=\frac{1}{2}\left\{(q_2)^{n+1}_{i,j+1}+(q_2)^{n+1}_{i,j}\right\}-\frac{1}{2}C_{i,j+\frac{1}{2}}(\rho^n_{i,j+1}-\rho^n_{i,j}),\\
&\boldsymbol{F}^{n}_{i+\frac{1}{2},j}=\frac{1}{2}\left\{\boldsymbol{F}^{n}_{i+1,j}+\boldsymbol{F}^{n}_{i,j}\right\}-\frac{1}{2}C_{i+\frac{1}{2},j}({\boldsymbol{q}}^n_{i+1,j}-{\boldsymbol{q}}^n_{i,j}),\\
&\boldsymbol{G}^{n}_{i,j+\frac{1}{2}}=\frac{1}{2}\left\{\boldsymbol{G}^{n}_{i,j+1}+\boldsymbol{G}^{n}_{i,j}\right\}-\frac{1}{2}C_{i,j+\frac{1}{2}}({\boldsymbol{q}}^n_{i,j+1}-{\boldsymbol{q}}^n_{i,j}).
\end{align*}

and
\begin{gather*}
  C_{i+\frac{1}{2},j}=\max
  \{|\lambda^{(1)}_{i,j}|,|\lambda^{(1)}_{i+1,j}|,|\lambda^{(2)}_{i,j}|,|\lambda^{(2)}_{i+1,j}|\},\\
 C_{i,j+\frac{1}{2}}=\max
  \{|\lambda^{(1)}_{i,j}|,|\lambda^{(1)}_{i,j+1}|,|\lambda^{(2)}_{i,j}|,|\lambda^{(2)}_{i,j+1}|\}.
\end{gather*}

Similar to the one-dimensional case, by inserting the momentum equation into the density equation, we can get
the following elliptic equation
\begin{align}
  \begin{split}
    &\rho^{n+1}_{i,j}  - \frac{\Delta t^2}{4}\bigg\{\frac{1}{\Delta
      x^{2}}\left[\varepsilon p_1(\rho_{i+2,j}^{n+1}) -2 \varepsilon
      p_1(\rho_{i,j}^{n+1}) + \varepsilon
      p_1(\rho_{i-2,j}^{n+1})\right]\\
&\qquad +\frac{1}{\Delta y^{2}}\left[\varepsilon p_1(\rho_{i,j+2}^{n+1}) -2 \varepsilon p_1(\rho_{i,j}^{n+1}) + \varepsilon p_1(\rho_{i,j-2}^{n+1})\right]\bigg\}\\
   = &\ \rho^{n}_{i,j}-\Delta t(D^x_{i,j}q_1^{n} +
   D^y_{i,j}q_2^{n})\\
&+\frac{\Delta t^2}{2} \bigg\{ \frac{1}{\Delta x^2}\left[(\boldsymbol{F}_{i+3/2,j}^{n})^{(1)} - (\boldsymbol{F}_{i+1/2,j}^{n})^{(1)}
      - (\boldsymbol{F}_{i-1/2,j}^{n})^{(1)} + (\boldsymbol{F}_{i-3/2,j}^{n})^{(1)}\right]\\
&\qquad +\frac{1}{\Delta x\Delta y}\left[(\boldsymbol{G}_{i+1,j+1/2}^{n})^{(1)} - (\boldsymbol{G}_{i+1,j-1/2}^{n})^{(1)}
      - (\boldsymbol{G}_{i-1,j+1/2}^{n})^{(1)} + (\boldsymbol{G}_{i-1,j-1/2}^{n})^{(1)}\right]\\
&\qquad +\frac{1}{\Delta x\Delta y}\left[(\boldsymbol{F}_{i+1/2,j+1}^{n})^{(2)} - (\boldsymbol{F}_{i-1/2,j+1}^{n})^{(2)}
      - (\boldsymbol{F}_{i+1/2,j-1}^{n})^{(2)} + (\boldsymbol{F}_{i-1/2,j-1}^{n})^{(2)}\right]\\
&\qquad +\frac{1}{\Delta y^2}\left[(\boldsymbol{G}_{i,j+3/2}^{n})^{(2)} - (\boldsymbol{G}_{i,j+1/2}^{n})^{(2)}
      - (\boldsymbol{G}_{i,j-1/2}^{n})^{(2)} +
      (\boldsymbol{G}_{i,j-3/2}^{n})^{(2)}\right]\bigg\}\\
&+\frac{\Delta t}{2\Delta x}\bigg
[C_{i+\frac{1}{2},j}(\rho^n_{i+1,j}-\rho^n_{i,j})-C_{i-\frac{1}{2},j}(\rho^n_{i,j}-\rho^n_{i-1,j})\bigg]\\
&+\frac{\Delta t}{2\Delta y}\bigg [C_{i,j+\frac{1}{2}}(\rho^n_{i,j+1}-\rho^n_{i,j})-C_{i,j-\frac{1}{2}}(\rho^n_{i,j}-\rho^n_{i,j-1})\bigg].
  \end{split}\label{eq:19}
\end{align}
Here $(\boldsymbol{F}_{i+1/2,j}^{n})^{(1)}$ is the first component of the vector
$\boldsymbol{F}_{i+1/2,j}^{n}$. Also similar to the one-dimensional case, we solve
the above elliptic equation to get first $p_1^{n+1}$ then $\rho^{n+1}$.
 And once $\rho^{n+1}$ is derived, we can get ${q}^{n+1}$ explicitly.
 \begin{gather*}
 \boldsymbol{q}_{i,j}^{n+1}= \boldsymbol{q}_{i,j}^{n}- \frac{\Delta
   t}{\Delta
   x}\bigg(\boldsymbol{F}^{n}_{i+\frac{1}{2},j}-\boldsymbol{F}^{n}_{i-\frac{1}{2},j}\bigg)-
 \frac{\Delta t}{\Delta
   y}\bigg(\boldsymbol{G}^{n}_{i,j+\frac{1}{2}}-\boldsymbol{G}^{n}_{i,j-\frac{1}{2}}\bigg)-\Delta
 t \nabla_{i,j}(\varepsilon p_1)^{n+1}.
 \end{gather*}

\paragraph{Gauge method: }
The Gauge method can be implemented in a similar way as the 1D
case. Indeed, we have the same elliptic equation for $p_1$ \eqref{eq:19} and the following equations:
\begin{align*}
 &\frac{1}{4\Delta x^{2}}\left[
      \varphi_{i+2,j}^{n+1} -2 \varphi_{i,j}^{n+1} + 
      \varphi_{i-2,j}^{n+1}\right]+\frac{1}{4\Delta
      y^{2}}\left[ \varphi_{i,j+2}^{n+1} -2 
      \varphi_{i,j}^{n+1} +   \varphi_{i,j-2}^{n+1}\right]\\
&\qquad = \frac{\rho_{i,j}^{n+1}-\rho_{i,j}^{n}}{\Delta t}-\frac{1}{2\Delta
  x}\bigg[C_{i+\frac{1}{2},j}(\rho^n_{i+1,j}-\rho^n_{i,j})-C_{i-\frac{1}{2},j}(\rho^n_{i,j}-\rho^n_{i-1,j})\bigg]\\
&\qquad\quad -\frac{1}{2\Delta  y}\bigg[C_{i,j+\frac{1}{2}}(\rho^n_{i,j+1}-\rho^n_{i,j})-C_{i,j-\frac{1}{2}}(\rho^n_{i,j}-\rho^n_{i,j-1})\bigg],\\
&\frac{1}{4\Delta
      x^{2}}\left[P_{i+2,j}^{n+1} -2   P_{i,j}^{n+1} + P_{i-2,j}^{n+1}\right]+\frac{1}{4\Delta y^{2}}\left[ P_{i,j+2}^{n+1} -2  P_{i,j}^{n+1} +   P_{i,j-2}^{n+1}\right]\\
&\qquad =  -\frac{1}{2} \bigg\{ \frac{1}{\Delta x^2}\left[(\boldsymbol{F}_{i+3/2,j}^{n})^{(1)} - (\boldsymbol{F}_{i+1/2,j}^{n})^{(1)}
      - (\boldsymbol{F}_{i-1/2,j}^{n})^{(1)} + (\boldsymbol{F}_{i-3/2,j}^{n})^{(1)}\right]\\
&\qquad\quad +\frac{1}{\Delta x\Delta y}\left[(\boldsymbol{G}_{i+1,j+1/2}^{n})^{(1)} - (\boldsymbol{G}_{i+1,j-1/2}^{n})^{(1)}
      - (\boldsymbol{G}_{i-1,j+1/2}^{n})^{(1)} + (\boldsymbol{G}_{i-1,j-1/2}^{n})^{(1)}\right]\\
&\qquad\quad +\frac{1}{\Delta x\Delta y}\left[(\boldsymbol{F}_{i+1/2,j+1}^{n})^{(2)} - (\boldsymbol{F}_{i-1/2,j+1}^{n})^{(2)}
      - (\boldsymbol{F}_{i+1/2,j-1}^{n})^{(2)} + (\boldsymbol{F}_{i-1/2,j-1}^{n})^{(2)}\right]\\
&\qquad\quad +\frac{1}{\Delta y^2}\left[(\boldsymbol{G}_{i,j+3/2}^{n})^{(2)} - (\boldsymbol{G}_{i,j+1/2}^{n})^{(2)}
      - (\boldsymbol{G}_{i,j-1/2}^{n})^{(2)} +(\boldsymbol{G}_{i,j-3/2}^{n})^{(2)}\right]\bigg\},\\
&\frac{\boldsymbol{a}_{i,j}^{n+1}-\boldsymbol{a}_{i,j}^{n}}{\Delta
    t}+ \frac{1}{\Delta
    x}\bigg(\boldsymbol{F}^{n}_{i+\frac{1}{2},j}-\boldsymbol{F}^{n}_{i-\frac{1}{2},j}\bigg)+
  \frac{1}{\Delta
    y}\bigg(\boldsymbol{G}^{n}_{i,j+\frac{1}{2}}-\boldsymbol{G}^{n}_{i,j-\frac{1}{2}}\bigg)+\nabla_{i,j}P^{n+1}=0,\\
&\boldsymbol{q}_{i,j}^{n+1}=\boldsymbol{a}_{i,j}^{n+1}-\nabla_{i,j}\varphi^{n+1}.
\end{align*}
This is the Gauge 2 method. And similar to the one dimensional case,
we will mainly test the Gauge 1 method with a  smaller
stencil in the Laplace equation of $\varphi$ and $P$:
\begin{align*}
    &\frac{1}{\Delta x^{2}}\left[
      \varphi_{i+1,j}^{n+1} -2 \varphi_{i,j}^{n+1} + 
      \varphi_{i-1,j}^{n+1}\right]+\frac{1}{\Delta
      y^{2}}\left[ \varphi_{i,j+1}^{n+1} -2 
      \varphi_{i,j}^{n+1} +   \varphi_{i,j-1}^{n+1}\right]\\
&\qquad = \frac{\rho_{i,j}^{n+1}-\rho_{i,j}^{n}}{\Delta t}-\frac{1}{2\Delta
  x}\bigg[C_{i+\frac{1}{2},j}(\rho^n_{i+1,j}-\rho^n_{i,j})-C_{i-\frac{1}{2},j}(\rho^n_{i,j}-\rho^n_{i-1,j})\bigg]\\
&\qquad\quad -\frac{1}{2\Delta  y}\bigg[C_{i,j+\frac{1}{2}}(\rho^n_{i,j+1}-\rho^n_{i,j})-C_{i,j-\frac{1}{2}}(\rho^n_{i,j}-\rho^n_{i,j-1})\bigg],\\
&\frac{1}{\Delta
      x^{2}}\left[P_{i+1,j}^{n+1} -2   P_{i,j}^{n+1} + P_{i-1,j}^{n+1}\right]+\frac{1}{\Delta y^{2}}\left[ P_{i,j+1}^{n+1} -2  P_{i,j}^{n+1} +   P_{i,j-1}^{n+1}\right]\\
   &\qquad = -\frac{1}{2} \bigg\{ \frac{1}{\Delta x^2}\left[(\boldsymbol{F}_{i+3/2,j}^{n})^{(1)} - (\boldsymbol{F}_{i+1/2,j}^{n})^{(1)}
      - (\boldsymbol{F}_{i-1/2,j}^{n})^{(1)} + (\boldsymbol{F}_{i-3/2,j}^{n})^{(1)}\right]\\
&\qquad\quad +\frac{1}{\Delta x\Delta y}\left[(\boldsymbol{G}_{i+1,j+1/2}^{n})^{(1)} - (\boldsymbol{G}_{i+1,j-1/2}^{n})^{(1)}
      - (\boldsymbol{G}_{i-1,j+1/2}^{n})^{(1)} + (\boldsymbol{G}_{i-1,j-1/2}^{n})^{(1)}\right]\\
&\qquad\quad +\frac{1}{\Delta x\Delta y}\left[(\boldsymbol{F}_{i+1/2,j+1}^{n})^{(2)} - (\boldsymbol{F}_{i-1/2,j+1}^{n})^{(2)}
      - (\boldsymbol{F}_{i+1/2,j-1}^{n})^{(2)} + (\boldsymbol{F}_{i-1/2,j-1}^{n})^{(2)}\right]\\
&\qquad\quad +\frac{1}{\Delta y^2}\left[(\boldsymbol{G}_{i,j+3/2}^{n})^{(2)} - (\boldsymbol{G}_{i,j+1/2}^{n})^{(2)}
      - (\boldsymbol{G}_{i,j-1/2}^{n})^{(2)} +(\boldsymbol{G}_{i,j-3/2}^{n})^{(2)}\right]\bigg\}.
\end{align*}

 Different from the one-dimensional case, now we are facing three
elliptic equations in each time step. In section
\ref{sec:2d-case-num}, we will see some simulation results and
comparison for these two methods.

\clearpage



\begin{thebibliography}{99}

\bibitem{1998_Anderson_DiffInterface} D.M. Anderson, G.B. McFadden, A.A. Wheeler, Diffuse-interface methods in fluid mechanics, Annu. Rev. Fluid. Mech., 30(1), 139-165 (1998)

\bibitem{2006_ArmDegRing} D. Armbruster, P. Degond, C. Ringhofer, A Model for the Dynamics of large Queuing Networks and Supply Chains, SIAM J. Appl. Math., 66(3), 896-920 (2006) 

\bibitem{2008_ModCrowd_Bellomo} N. Bellomo, C. Dogbe, On the modelling crowd dynamics from scaling to hyperbolic macroscopic models, Math. Models Methods Appl. Sci., 18, 1317-1345 (2008)

\bibitem{2006_BertBreussTit_pressureless} C. Berthon, M. Breuss, M.O. Titeux, A relaxation scheme for the approximation of the pressureless Euler equations, Numerical Methods for Partial Differential Equations, 22(2), 484-505 (2006)

\bibitem{2003_ChenLiu_FormationDelta} G.Q. Chen, H. Liu, Formation of delta-shocks and vacuum states in the vanishing pressure limit of solutions to the Euler equations for isentropic fluids, SIAM J. Math. Anal., 34(4), 925-938 (2003)

\bibitem{1998_BijlWesseling} H. Bijl, P. Wesseling, A unified method for computing incompressible and compressible flows in boundary-fitted coordinates, J. Comput. Phys., 141(2), 153-173 (1998)

\bibitem{2007_ChertKurgRyk_StickyPart}  A. Chertock, A. Kurganov, Y. Rykov, A new sticky particle method for pressureless gas dynamics, SIAM J. Numer. Anal., 45(6), 2408-2441 (2007)

\bibitem{2002_ExistPressLess_Berth} F. Berthelin, Existence and weak stability for a pressureless model with unilateral constraint, Mathematical Models and Methods in Applied Sciences, 12(2), 249-272 (2002)

\bibitem{2008_Traffic_DegondRascle} F. Berthelin, P. Degond, M. Delitala, M. Rascle, A model for the formation and the evolution of traffic jams, Arch. Rational Mech. Anal., 187, 185-220 (2008)

\bibitem{2008_Trafficflow_Berthelin_D} F. Berthelin, P. Degond, V. Le Blanc, S. Moutari, M. Rascle, J. Royer, A traffic-flow model with constraints for the modeling of traffic jams, Math. Models Methods Appl. Sci., 18, 1269-1298 (2008)

\bibitem{PGD_Bouchut} F. Bouchut, On zero pressure gas dynamics, Advances in kinetic theory and computing: selected papers, Ser. Adv. Math. Appl. Sci., 22, 171-190 (1994)

\bibitem{2PhaseFlow_Bouchut_al} F. Bouchut, Y. Brenier, J. Cortes, J.-F. Ripoll, A hierarchy of models for two-phase flows, J. Nonlinear Sci., 10, 639-660 (2000)

\bibitem{2004_NumericPGD_BouchutJin} F. Bouchut, S. Jin, X. Li, Numerical Approximations of Pressureless and Isothermal Gas Dynamics, SIAM J. Num. Anal., 41(1), 135-158 (2004)

\bibitem{1998_BrGr_sticky} Y. Brenier, E. Grenier, Sticky particles and scalar conservation laws, SIAM J. Num. Anal., 2317-2328 (1998)

\bibitem{2009_TraficNum_D_Delitala} P. Degond, M. Delitala, Modelling and simulation of vehicular traffic jam formation, Kinetic Related Models, 1(2), 279-293 (2009) 

\bibitem{2007_Low_mach_D_SJ_JGL} P. Degond, S. Jin, J.-G. Liu, Mach-number uniform asymptotic-preserving gauge schemes for compressible flows, Bulletin of the institute of Mathematics, Academia Sinica, 2(4), 851-892 (2007)

\bibitem{2010_Congestion} P. Degond, L. Navoret, R. Bon, D. Sanchez, Congestion in a Macroscopic Model of Self-driven Particles Modeling Gregariousness, J. Stat. Phys., 138(1-3), 85-125 (2010)

\bibitem{2009_Degond_LowMach} P. Degond, M. Tang, Asymptotic preserving method for the incompressible low mach number limit of the Isentropic Euler equation (2009) (in prep.)

\bibitem{2001_GuillardMurrone} H. Guillard, A. Murrone, On the behavior of upwind schemes in the low mach number limit: II, Godunov type schemes, INRIA research report \# 4189 (2001)

\bibitem{1999_GuillardViozat} H. Guillard, C. Viozat, On the behaviour of upwind schemes in the low Mach number limit, Comput. Fluids, 28(1), 63-86 (1999)

\bibitem{1971_HarlowAmsden_ICE_LM} F.H. Harlow, A.A. Amsden, A numerical fluid dynamics calculation method for all flow speeds, J. Comput. Phys., 8(2), 197-213 (1971) 

\bibitem{1986_VanDerWaalsRiem_Hattori} H. Hattori, The Riemann problem for a van der Waals fluid with entropy rate admissibility criterion, Isothermal case, Archive for Rational Mechanics and Analysis, 92(3), 247-263 (1986)

\bibitem{1981_HirtNichols_VOF} C.W. Hirt, B.D. Nichols, Volume of fluid (VOF) method for the dynamics of free boundaries, J. Comp. Phys., 39(1), 201-225 (1981)


\bibitem{2005_Kadioglu} S.Y. Kadioglu, M. Sussman, S. Osher, J.P. Wright, M. Kang, A second order primitive preconditioner for solving all speed multi-phase flows, J. Comput. Phys., 209(2), 477-503 (2005)

\bibitem{1995_Klein_LM} R. Klein, Semi-implicit extension of a Godunov-type scheme based on low Mach number asymptotics I: one-dimensional flow, J. Comput. Phys., 121(2), 213--237 (1995)


\bibitem{2004_Leveque_Dust} R.J. LeVeque, The dynamics of pressureless dust clouds and delta waves, J. Hyperbolic Differ. Equ., 1, 315-328 (2004)

\bibitem{leveque_book_2002} R.J. LeVeque, \textit{Finite volume methods for hyperbolic problems}. Cambridge Univ. Press (2002)

\bibitem{2008_LiGu} X. Li, C. Gu, An all-speed Roe-type scheme and its asymptotic analysis of low Mach number behaviour, J. Comput. Phys., 227(10), 5144-5159 (2008)

\bibitem{1999_barotropic_LionsMasmoudi} P.L. Lions, N. Masmoudi, On a free boundary barotropic model, Annales de l'Institut Henri Poincare/Analyse non lineaire, 16(3), 373-410 (1999)

\bibitem{2003_Munz_MPV} C.D.Munz, S. Roller, R. Klein, K.J. Geratz, The extension of incompressible flow solvers to the weakly compressible regime, Computers and Fluids, 32(2), 173-196 (2003)

\bibitem{2005_Nerinckx} K. Nerinckx, J. Vierendeels, E. Dick, Mach-uniformity through the coupled pressure and temperature correction algorithm, J. Comput. Phys., 206, 597-623 (2005) 

\bibitem{2003_LevelSet_BookOsher} S. Osher, R.P. Fedkiw, \textit{Level set methods and dynamic implicit surfaces}, Springer, 2003

\bibitem{1980_Patankar_book} S. V. Patankar, Numerical heat transfer and fluid flow, New York: McGraw-
Hill, (1980)

\bibitem{2009_Rauwoens} P. Rauwoens, J. Vierendeels,  E. Dick,, B. Merci, , A conservative discrete compatibility-constraint low-Mach pressure-correction algorithm for time-accurate simulations of variable density flows, J. Comput. Phys., 228(13), 4714-4744 (2009)

\bibitem{2010_Tang} M. Tang, Second order all speed method for isentropic Euler equations, in prep. 

\bibitem{2001_FrontTrackMultiphase_tryggvason} G. Tryggvason, B. Bunner, A. Esmaeeli, D. Juric, N. Al-Rawahi, W. Tauber, J. Han, S. Nas, Y.-H. Jan, A front-tracking method for the computations of multiphase flow,J. Comp. Phys., 169(2), 708-759 (2001)

\bibitem{1999_Turkel_Precond} E. Turkel, Preconditioning techniques in computational fluid dynamics, Annu. Rev. Fluid Mech., 31, 385-416 (1999)

\bibitem{2003_HeulWesseling} D.R. van der Heul, C. Vuik, P. Wesseling, A conservative pressure-correction method for flow at all speeds, Comput. Fluids, 32(8), 1113-1132 (2003)

\bibitem{2002_WallPierceMoin} C. Wall, C.D. Pierce, P. Moin, A Semi-implicit Method for Resolution of Acoustic Waves in Low Mach Number Flows, J. Comput. Phys., 181(2), 545-563 (2002),


\end{thebibliography}
\end{document}